\documentclass[10pt,reqno]{article}

\usepackage[a4paper,margin=1.15in]{geometry}

\usepackage[english]{babel}
\usepackage{microtype}
\usepackage{titlesec}

\usepackage{amsmath, amsthm, amssymb, mathtools}
\usepackage{thmtools}
\usepackage{latexsym}

\usepackage[T1]{fontenc}
\usepackage{bbm}
\usepackage[cal=boondoxo]{mathalfa}

\usepackage{xspace}
\usepackage{setspace}
\usepackage{natbib}
\usepackage{tabularx}
\usepackage{stmaryrd}
\usepackage{dsfont}
\usepackage{euscript}
\usepackage{enumitem}
\usepackage{algpseudocode}
\usepackage{algorithm}

\usepackage{xcolor}

\usepackage[utf8]{inputenc}

\usepackage{xcolor}
\usepackage{makeidx}
\usepackage[bibliography=common]{apxproof}
\usepackage[normalem]{ulem}
\usepackage[
pdftex,
colorlinks=true,
linkcolor=blue!50!black,
citecolor=blue!50!black,
urlcolor=blue!50!black
]{hyperref}

\theoremstyle{plain}
\newtheorem{theorem}{Theorem}[section]
\newtheorem{lemma}[theorem]{Lemma}
\newtheorem{corollary}[theorem]{Corollary}
\newtheorem{proposition}[theorem]{Proposition}

\theoremstyle{definition}

\newtheorem{example}[theorem]{Example}
\newtheorem{assumption}[theorem]{Assumption}

\theoremstyle{remark}
\newtheorem{remark}[theorem]{Remark}
\newtheorem*{remark*}{Remark}

\newcommand{\cB}{\mathcal{B}}
\newcommand{\cD}{\mathcal{D}}
\newcommand{\cE}{\mathcal{E}}
\newcommand{\cF}{\mathcal{F}}
\newcommand{\cG}{\mathcal{G}}
\newcommand{\cH}{\mathcal{H}}

\newcommand{\cL}{\mathcal{L}}

\newcommand{\cN}{\mathcal{N}}
\newcommand{\cP}{\mathcal{P}}
\newcommand{\cR}{\mathcal{R}}

\newcommand{\cU}{\mathcal{U}}

\newcommand{\cX}{\mathcal{X}}
\newcommand{\cY}{\mathcal{Y}}

\newcommand{\cC}{\mathcal{C}}

\newcommand{\wh}{\widehat}
\newcommand{\wt}{\widetilde}
\newcommand{\qandq}{\quad \text{and} \quad}
\newcommand{\one}{\mathbf{1}}

\newcommand{\bbR}{\mathbb{R}}
\newcommand{\bbE}{\mathbb{E}}
\newcommand{\bbP}{\mathbb{P}}
\newcommand{\bbN}{\mathbb{N}}
\newcommand{\bbZ}{\mathbb{Z}}
\newcommand{\bbT}{\mathbb{T}}

\newcommand{\bbC}{\mathbb{C}}

\newcommand{\term}[1]{\mathcal{T}_\text{#1}}
\newcommand{\abs}[1]{| #1 |}

\newcommand{\LRabs}[1]{\left| #1 \right|}
\newcommand{\norm}[1]{\| #1 \|}
\newcommand{\bnorm}[1]{\big\| #1 \big\|}
\newcommand{\Bnorm}[1]{\Big\| #1 \Big\|}
\newcommand{\LRnorm}[1]{\left\| #1 \right\|}

\newcommand{\opnorm}[1]{\vert\!\vert\!\vert #1 \vert\!\vert\!\vert}

\newcommand{\la}{\langle}
\newcommand{\ra}{\rangle}
\newcommand{\Tr}{\operatorname{Tr}}

\newcommand{\MMD}{\mathrm{MMD}}
\newcommand{\mix}{\mathrm{mix}}
\newcommand{\Fix}{\mathrm{Fix}}
\newcommand{\Ran}{\mathrm{Ran}}
\newcommand{\rel}{\mathrm{rel}}

\newcommand{\Lip}{\mathrm{Lip}}
\newcommand{\TV}{\mathrm{TV}}
\newcommand{\Var}{\operatorname{Var}}
\newcommand{\Cov}{\operatorname{Cov}}
\newcommand{\diam}{\operatorname{diam}}
\newcommand{\Span}{\operatorname{span}}
\newcommand{\eff}{\mathrm{eff}}
\newcommand{\supp}{\operatorname{supp}}
\newcommand{\Law}{\operatorname{Law}}
\DeclareMathOperator*{\argmin}{arg\,min}

\hypersetup{
  pdftitle    = {Markov Chain Convergence in Kernel-Induced Geometry},
  pdfauthor   = {Oleksii Kachaiev, Luc Brogat-Motte, Silvia Villa, Lorenzo Rosasco},
  pdfsubject  = {Mathematics Subject Classification (2020): 60J05, 37A25, 46E22, 62M05, 68Q32},
  pdfkeywords = {Markov chains, contraction, maximum mean discrepancy, reproducing kernel Hilbert spaces,
                 kernel mean embeddings, occupation measure, ergodicity, mixing,
                 learning dynamical systems, single trajectory, regularized least squares,
                 learning rates, effective dimension},
  pdflang     = {en},
  pdfdisplaydoctitle = true,
  pdfstartview = FitH,
  bookmarksnumbered = true,
  bookmarksopen = true,
  bookmarksopenlevel = 1,
  colorlinks  = true,
  linkcolor  = blue!60!black,
  citecolor  = blue!60!black,
  urlcolor   = blue!60!black,
}

\title{
Markov Chain Convergence in Kernel-Induced Geometry
}

\author{%
  Oleksii Kachaiev$^{1}$, Luc Brogat-Motte$^{2}$, Silvia Villa$^{1}$, Lorenzo Rosasco$^{2,3}$\\
  \rule{0pt}{1.6em}
  \small $^1$MaLGa center, DIMA, Università degli Studi di Genova, Genoa, Italy\\
  \small $^2$SAIL Unit, Istituto Italiano di Tecnologia, Genoa, Italy\\
  \small $^3$MaLGa Center, DIBRIS, Università degli Studi di Genova,
    Genoa, Italy\\
  \rule{0pt}{1.4em}%
  \small Correspondence to: \texttt{oleksii.kachaiev@edu.unige.it}%
}

\date{}

\begin{document}

\maketitle
\begin{abstract}
We study the convergence of Markov chains in the geometry induced by Hilbertian embeddings associated with positive definite kernels, also called Maximum Mean Discrepancy (MMD). 
Under a contraction assumption in the kernel-induced geometry, and a standard spectral gap assumption,  we establish non-asymptotic, high-probability bounds for the convergence of the empirical occupation measure to the invariant measure.  The obtained bounds separate the roles of the two conditions:  the spectral gap determines the effective sample size in the leading term, while  the kernel contraction enters only through a linear correction. As an application, we study one-step prediction of dynamical systems from a single trajectory, deriving fast learning rates for a nonlinear least-squares estimator. 
The rates highlight the interplay between mixing properties in the kernel-induced geometry and the  classical source and  capacity conditions  used in the analysis of least-squares.
\end{abstract}

\section{Introduction}

Let \((X_t)_{t\geq0}\) be a Markov chain on \(\cX\) with
transition kernel \(P\).
Its \emph{empirical occupation measure} up to time \(T\) is
\begin{equation}
\wh\mu_T:=\frac{1}{T}\sum_{t=0}^{T-1}\delta_{X_t}.
\end{equation}
This random probability measure records the proportion of time
spent by the trajectory in each measurable set. An important question in the analysis of Markov chains (and, more broadly, of dynamical systems) is whether there exists a probability measure \(\pi\) such that
\(\wh\mu_T\to\pi\) as \(T\to\infty\). Such a limit, when it exists, describes the statistical behavior of trajectories over long times~\citep{eckmann1985ergodic,young2002srb}. For example, if \((X_t)_{t\geq0}\) are independent random variables on a separable metric space \( \cX \) with common law \(\pi\), the empirical law of large numbers gives, almost surely,
\begin{equation}
\label{eq:weak-conv}
\wh\mu_T \rightharpoonup \pi,
\qquad \text{or, equivalently,} \qquad
\int\!f \, d \wh\mu_T
= \frac{1}{T}\sum_{t=0}^{T-1} f(X_t)
\to \int\! f  \, d\pi
\end{equation}
for every bounded continuous test function \(f\). For general Markov chains, however, convergence to a deterministic limiting
measure is not automatic. For classical conditions ensuring almost-sure convergence of
\(\wh\mu_T\) as in~\eqref{eq:weak-conv}, see, e.g.,~\citet{meyn2012markov}.

\medskip
To quantify the discrepancy between \(\wh\mu_T\) and \(\pi\),
we need to choose a distance on the space of probability measures.
A flexible framework is provided by \emph{integral probability
metrics} (IPMs),
\begin{equation}
d_{\cF}(\mu,\nu)
:= \sup_{f\in\cF}
\LRabs{\int f\,d\mu-\int f\,d\nu},
\label{eq:ipm-intro}
\end{equation}
where \(\cF\) is a class of measurable test functions integrable
with respect to the measures under consideration.
The choice of \(\cF\) determines which differences between
probability measures are detected, and hence the notion of
convergence being studied.
For example, the class of measurable functions satisfying
\(\norm{f}_\infty\leq1\) gives the total variation norm,
while the class of \(1\)-Lipschitz functions gives the Wasserstein
distance \(W_1\) on probability measures with finite first moments.
Further examples and general properties of IPMs can be found
in \citet{muller1997integral} and \citet{rachev2013methods}.

\medskip
Our aim is to obtain non-asymptotic bounds on
\(d_{\cF}(\wh\mu_T,\pi)\), for suitable choices of \(\cF\),
with explicit dependence on the trajectory length \(T\).
Such bounds are useful in sampling, where they quantify
the error incurred when approximating a target distribution
by a finite trajectory of a Markov chain%
~\citep{roberts2004general,joulin2010curvature}.
Related questions also arise in the analysis of stochastic
optimization algorithms through their interpretation as Markov
chains~\citep{dieuleveut2020bridging,yu2021analysis}. 

\medskip
Non-asymptotic bounds have
been studied in several probability metrics.
For uniformly ergodic Markov chains on countable state spaces,
bounds in total variation follow from the concentration and
moment inequalities of \citet{paulin2015concentration}.
Bounds in \(W_1\) under suitable contractivity assumptions
have been obtained by
\citet{boissard2011simple,kloeckner2020empirical,riekert2022convergence}.
For \(W_p\), \citet{fournier2015rate} establish moment bounds
for empirical measures of independent samples and Markov chains.
Related bounds for occupation measures of continuous-time
ergodic Markov processes are obtained by
\citet{schilling2025convergence} under Wasserstein contractivity
or functional inequalities, together with moment assumptions
on the limiting measure.

\medskip
In this work, we consider IPMs obtained from  reproducing kernel Hilbert space (RKHS), known as
\emph{maximum mean discrepancies} (MMD).
Let \(k:\cX\times\cX\to\bbR\) be a measurable positive-definite
kernel with separable RKHS \(\cH\), and let
\(\phi(x):=k(x,\cdot)\) be its canonical feature map.
Taking \(\cF\) to be the unit ball of \(\cH\)
in~\eqref{eq:ipm-intro}, Hilbert-space duality gives
\begin{equation}
\MMD_\phi(\mu,\nu)
=
\LRnorm{
\int \phi(x)\,\mu(dx)
-
\int \phi(x)\,\nu(dx)
}_{\cH},
\label{MMD_intro}
\end{equation}
whenever the two Bochner integrals exist;
see, e.g., \citet{smola2007hilbert,gretton2012kernel}.
The same construction applies to any measurable feature map
\(\phi:\cX\to\cH\) into a separable Hilbert space, subject to
the same integrability requirement. The resulting discrepancy
is the MMD associated with the kernel
\(k(x,y)=\langle\phi(x),\phi(y)\rangle_{\cH}\).
We develop this feature-map formulation in
Section~\ref{sec:kernel-geometry}.

\medskip
When the feature map is uniformly bounded, that is,
\(\sup_{x\in\cX}\norm{\phi(x)}_{\cH}<\infty\),
the discrepancy is well defined for all probability measures,
including empirical occupation measures.
For independent observations with common law \(\pi\),
\(\MMD_\phi(\wh\mu_T,\pi)\) admits bounds of order
\(T^{-1/2}\), both in expectation and at each fixed confidence
level~\citep{gretton2012kernel}.
For dependent observations, \citet{cherief2022finite} obtain
high-probability bounds under weak-dependence assumptions
in their analysis of minimum-MMD estimators.
More recently, \citet{wolfer2025variance} derive such bounds
for stationary \(\varphi\)- and \(\beta\)-mixing sequences.

\paragraph{Contributions.}
Our first contribution is a non-asymptotic high-probability
bound for empirical occupation measures of Markov chains
on a general state space.
Assume that the chain is stationary, has an absolute
\(L^2(\pi)\) spectral gap \(1-\rho>0\), and satisfies
the following uniform geometric relaxation condition:
\begin{equation}
\label{eq:main-thm-condition}
\sup_{x\in\cX}
\MMD_\phi(P^n(x,\cdot),\pi)
\leq C\theta^n,
\qquad n\geq0,
\end{equation}
for some \(C<\infty\) and \(\theta\in[0,1)\).
Then, for every \(T\geq1\) and \(\delta\in(0,1)\),
with probability at least \(1-\delta\),
\begin{equation}
\MMD_\phi(\wh\mu_T,\pi)
\lesssim
\norm{\phi}_\pi
\sqrt{\frac{\tau_\rel\log(4/\delta)}{T}}
+
\frac{\tau_\theta\sqrt{\tau_\rel}\log(4/\delta)}{T},
\end{equation}
where \(\tau_\rel:=(1-\rho)^{-1}\),
\(\tau_\theta:=C/(1-\theta)\), and
\(\norm{\phi}_\pi^2:=\int\norm{\phi(x)}_\cH^2\,\pi(dx)\);
the implicit constant is universal.
The leading term depends on the spectral gap and the second
moment of the feature map, while the relaxation constants
\(C\) and \(\theta\) enter only through the \(T^{-1}\) term.
We extend this bound to initial laws \(\nu\ll\pi\) with
\(d\nu/d\pi\in L^p(\pi)\) for some \(p>1\), leaving the
leading term unchanged; see
Theorems~\ref{thm:mmd-occupation-measure-whp}
and~\ref{thm:mmd-occupation-measure-whp-nonstationary}. We also give sufficient conditions
for~\eqref{eq:main-thm-condition}.

\medskip
Our second contribution concerns kernel ridge regression
(KRR) for learning the one-step prediction function
\begin{equation}
f_\star(x):=\int x'\,P(x,dx')
\end{equation}
of a real-valued
Markov chain from a single trajectory.
The relevant MMD geometry arises from the covariance
fluctuations controlling estimator stability.
Writing \(\Sigma\) and \(\wh\Sigma\) for the population
and empirical covariance operators, and denoting
\(\Sigma_\lambda:=\Sigma+\lambda I\), we have
\begin{equation}
\LRnorm{
\Sigma_\lambda^{-1/2}
(\wh\Sigma-\Sigma)
\Sigma_\lambda^{-1/2}
}_{\cL_2(\cH)}
=
\MMD_{\Phi_\lambda}(\wh\mu_T,\pi),
\end{equation}
where
\(\Phi_\lambda(x):=
\Sigma_\lambda^{-1/2}
[\phi(x)\otimes\phi(x)]
\Sigma_\lambda^{-1/2}\).
We obtain relaxation bounds in this
\(\lambda\)-dependent geometry from contraction in a fixed
Hilbert scale associated with the covariance-feature kernel
\(K(x,y)=k(x,y)^2\).

\medskip
For stationary trajectories, under the spectral and
Hilbert-scale assumptions,
and classical source and capacity conditions of
Theorem~\ref{thm:final-sample-rate}, we obtain
\begin{equation}
\cE(\wh f_{\lambda_T},f_\star)
=
\bnorm{\wh f_{\lambda_T}-f_\star}_\pi^2
\lesssim
T^{-\frac{2r}{2r+p}},
\qquad
\lambda_T\asymp T^{-\frac{1}{2r+p}}.
\end{equation}
For every \(\delta\in(0,1)\), this bound holds with
probability at least \(1-\delta\) for each sufficiently
large \(T\), up to logarithmic factors in \(1/\delta\).
Here, \(r\in[1/2,1]\) is the source exponent and
\(p=1/b\in(0,1)\) is the capacity exponent, characterized
by \(\cN(\lambda)\lesssim\lambda^{-p}\).
The rate matches the classical source--capacity exponent
for i.i.d.\ KRR~\citep{caponnetto2007optimal}, with explicit
dependence on the chain in the finite-sample bounds. By incorporating source and capacity information, this
sharpens the \(T^{-1/2}\) excess-risk rate of
\citet{kachaiev2026learning}, however, under different assumptions. Our work considers only stationary initialization while relaxing uniform ergodicity in total variation to absolute spectral gap, and the
Hilbert-scale contraction described above.

\paragraph{Structure of the paper.}
Section~2 collects background on Markov operators, martingales,
integral probability metrics, and spectral gaps, together with
preliminary results used throughout the paper.
Section~\ref{sec:kernel-geometry} introduces uniform MMD
relaxation and presents the occupation measure bounds,
followed by sufficient conditions and examples based on
metric comparison and contraction in kernel-induced
Hilbert scales.
Section~\ref{sec:stats-learning} applies these results to
KRR and derives the learning guarantees.
The proofs, together with the concentration inequalities they
rely on, are collected in Appendix~A; Appendix~B contains the
Fourier computations for the torus examples, including the
verification of all assumptions of
Theorem~\ref{thm:final-sample-rate} for a noisy doubling map.

\section{Background and preliminaries}

\paragraph{Tools and notation.}
Let \( (\cX,d) \) be a Polish space endowed with Borel \( \sigma \)-algebra \( \cB(\cX) \). Let \( \cP(\cX) \) be a space of probability measures on \( (\cX, \cB(\cX)) \). In the context where \( \cX \) is a normed space, \( \cP_2(\cX) \) denotes the subset of probability measures on \( \cX \) with finite second moment, i.e., \( \int \norm{x}^2 \, \mu(dx) < \infty \). Define \( \diam(\cX):=\sup_{x,y\in\cX}d(x,y) \). 

\medskip
Let \( (X_t)_{t \ge 0} \) be a Markov process taking values in \( (\cX,\cB(\cX)) \), with associated Markov transition kernel \( P:\cX\times\cB(\cX)\to[0,1] \). Assume that \( P \) admits an invariant probability measure \( \pi\in\cP(\cX) \), that is, \( \pi P=\pi \), where, for \( \mu\in\cP(\cX) \), we write \( \mu P \) for the probability measure
\begin{equation}
\label{eq:measure-action}
(\mu P)(A)
:=\int P(x,A)\,\mu(dx),
\qquad A\in\cB(\cX).
\end{equation}
The process \( (X_t)_{t \ge 0} \) is called \emph{stationary} if \( X_0 \sim \pi \). We also use the notation \( P^n(x,\cdot):=\delta_xP^n \) for \( n \ge 0 \). Recall that for a fixed \( T > 0 \),
\begin{equation}
\wh \mu_T := \frac{1}{T}\sum_{t=0}^{T-1} \delta_{X_t}
\end{equation}
denotes empirical \emph{occupation measure}. For a measure \(\mu\) and a measurable function \(f:\cX\to\bbR\), let \(\mu(f):=\int\! f\,d\mu\).

\medskip
Let \( L^2(\pi) \) denote the Hilbert space of  square-integrable functions \( f : \cX \to \bbR \), endowed with the inner product \( \la f, g \ra_\pi := \pi(fg) \). Denote the subspace of \( \pi \)-centered functions by \( L_0^2(\pi) := \left\{ f \in L^2(\pi) : \pi(f) = 0 \right\} \). For a separable Hilbert space \( \cY \), we let \( L^2(\cX, \cB(\cX), \pi; \cY) \), abbreviated \( L^2(\pi; \cY) \), denote the Bochner space of strongly \( \cB(\cX) \)-measurable and Bochner square-integrable maps \( f : \cX \to \cY \) with
\begin{equation}
\norm{f}_\pi^2 := \norm{f}_{L^2(\pi; \cY)}^2 = \int \norm{f(x)}_\cY^2\, \pi(dx) < \infty.
\end{equation}
Denote by  \( \cL(\cH_1, \cH_2) \) be  the space of bounded linear operators between two separable Hilbert spaces \( \cH_1 \) and \( \cH_2 \), and    \( \cL(\cH_1) \) when \( \cH_1 = \cH_2 \). For \(A\in\cL(\cH_1)\), denote the operator norm by \(\opnorm{A}_{\cH_1}:=\norm{A}_{\cH_1 \to \cH_1}\). Let \( \cL_2(\cH_1, \cH_2) \) be the space of Hilbert--Schmidt operators, and  \( \cL_2(\cH_1) \) when \( \cH_1 = \cH_2 \), equipped with the norm \( \norm{A}_{\cL_2(\cH_1, \cH_2)}^2 = \sum_i \norm{A e_i}_{\cH_2}^2 \), where \( \left\{ e_i \right\}_i \) is any orthonormal basis of \( \cH_1 \). We denote by \( I_d \) the \( d \times d \) identity matrix.

\paragraph{Martingales.} Let \( (\cF_t)_{t\geq0} \) be a filtration.  A sequence of \(\cH\)-valued random variables \( (d_t)_{t\geq1} \) is called a \emph{martingale difference sequence} with respect to \( (\cF_t)_t \) if \(d_t\) is \(\cF_t\)-measurable and for every \( t \ge 1 \), \( \bbE [d_t \mid \cF_{t-1}] = 0 \). Setting \( M_0 := 0 \) and \( M_T := \sum_{t=1}^T d_t \) define an \( \cH \)-valued \emph{martingale}. Its \emph{predictable quadratic variation} is
\begin{equation}
\label{eq:quadratic-variance}
\la M \ra_T
:= \sum_{t=1}^{T} \bbE_{t-1} \! \norm{ d_t }^2,
\end{equation}
where \( \bbE_j [\cdot] := \bbE [ \cdot \mid \cF_j ] \). For background on martingales and martingale difference sequences, we refer to~\citet{hall1980martingale,Williams_1991}.

\paragraph{Integral probability metrics.}
Let \(\cF\) be a class of measurable real-valued functions on \(\cX\). For \(\mu,\nu\in\cP(\cX)\) such that the integrals below are well defined, the \emph{integral probability metric (IPM)} generated by \(\cF\) is
\begin{equation}
\label{eq:ipm-definition}
d_{\cF}(\mu,\nu)
:=\sup_{f\in\cF}\LRabs{\mu(f)-\nu(f)}.
\end{equation}
The choice of \(\cF\) determines the geometry on the space of probability measures. In particular, for the class
\begin{equation}
\cF_{\TV}
:= \left\{f:\cX\to[0,1]: f \text{ is measurable}\right\},
\end{equation}
the corresponding IPM is the total variation distance:
\begin{equation}
d_{\cF_{\TV}}(\mu,\nu)
= \sup_{A\in\cB(\cX)}\abs{\mu(A)-\nu(A)}
= \norm{\mu-\nu}_{\TV}.
\label{eq:tv-ipm-def}
\end{equation}
Suppose that \((\cX,d)\) is a bounded space, i.e., \( \diam(\cX) < \infty \). Fix \(x_0\in\cX\) and consider the anchored Lipschitz class
\begin{equation}
\cF_{\Lip(d)}
:= \left\{ f:\cX\to\bbR:
\abs{f(x)-f(y)}\leq d(x,y) \; \forall x,y\in\cX,
\; f(x_0)=0 \right\}.
\end{equation}
Since \(\cX\) is bounded, every probability measure has a finite first moment, so by the Kantorovich--Rubinstein duality~\citep[Remark~6.5]{villani2009optimal},
\begin{equation}
d_{\cF_{\Lip(d)}}(\mu,\nu)
= W_1(\mu,\nu),
\qquad \mu,\nu\in\cP(\cX).
\end{equation}
Finally, when \(\cF\) is the unit ball of a reproducing kernel Hilbert space (RKHS), the resulting IPM is the MMD~\eqref{MMD_intro}, which we will study in detail in the following section.
For a systematic treatment of IPMs, their generating classes, and the corresponding domains of probability measures, see~\citet{muller1997integral}. The Lipschitz case and its restriction to measures with finite first moment are discussed in~\citet[Chapters~5--6]{villani2009optimal}, while higher-order smooth test classes and the associated moment conditions belong to the theory of Zolotarev metrics; see \citet{zolotarev1984probability}.

\medskip
The following proposition collects elementary properties of IPMs that we will use to compare convergence bounds in different geometries and to derive bounds uniform over all initial distributions from one-step contraction.

\begin{proposition}[Comparison and uniform relaxation of IPMs]
\label{prop:ipm-comparison-relaxation}
Let \(P\) be a Markov kernel with invariant probability measure \(\pi\), and let \(\cF,\cF_1,\cF_2\) be nonempty classes of bounded measurable real-valued functions on \(\cX\). The corresponding IPMs are allowed to take the value \(+\infty\).

\begin{enumerate}
\item[\textup{(i)}]
If \(\cF_1\subseteq\Lambda\cF_2+\bbR\one\) for some \(\Lambda>0\), then
\begin{equation}
d_{\cF_1}(\mu,\nu)
\leq \Lambda d_{\cF_2}(\mu,\nu),
\qquad \mu,\nu\in\cP(\cX).
\label{eq:ipm-comparison-scaled}
\end{equation}
Consequently, if
\[
\sup_{x\in\cX}d_{\cF_2}\bigl(P^n(x,\cdot),\pi\bigr)
\leq C\theta^n,
\qquad n \ge 0,
\]
for some \(0\leq C<\infty\) and \(\theta\in[0,1)\), then
\begin{equation}
\sup_{x\in\cX}d_{\cF_1}\bigl(P^n(x,\cdot),\pi\bigr)
\leq \Lambda C\theta^n,
\qquad n \ge 0.
\label{eq:weaker-ipm-uniform-convergence}
\end{equation}

\item[\textup{(ii)}]
For every \(n \ge 0\),
\begin{equation}
\sup_{\mu\in\cP(\cX)}d_{\cF}(\mu P^n,\pi)
= \sup_{x\in\cX}d_{\cF}\bigl(P^n(x,\cdot),\pi\bigr).
\label{eq:ipm-uniform-initial-distributions}
\end{equation}

\item[\textup{(iii)}]
Suppose that
\[
R_{\cF}:=\sup_{x\in\cX}d_{\cF}(\delta_x,\pi)<\infty
\]
and, for some \(\theta\in[0,1)\),
\begin{equation}
d_{\cF}(\mu P,\pi)
\leq\theta d_{\cF}(\mu,\pi),
\qquad \mu\in\cP(\cX).
\label{eq:ipm-one-step-toward-invariant}
\end{equation}
Then
\begin{equation}
\sup_{x\in\cX}d_{\cF}\bigl(P^n(x,\cdot),\pi\bigr)
\leq R_{\cF}\theta^n,
\qquad n \ge 0.
\label{eq:ipm-contraction-uniform-relaxation}
\end{equation}
\end{enumerate}
\end{proposition}

\begin{proof}
For \textup{(i)}, every \(f\in\cF_1\) can be written as \(f=\Lambda g+c\one\), with \(g\in\cF_2\) and \(c\in\bbR\). Since probability measures have equal total mass,
\begin{equation}
\LRabs{(\mu-\nu)(f)}
=\Lambda\LRabs{(\mu-\nu)(g)}
\leq\Lambda d_{\cF_2}(\mu,\nu).
\end{equation}
Taking the supremum over \(f\in\cF_1\) proves \eqref{eq:ipm-comparison-scaled}, from which \eqref{eq:weaker-ipm-uniform-convergence} follows immediately. For \textup{(ii)}, fix \(\mu\in\cP(\cX)\) and \(f\in\cF\). By the definition of \(\mu P^n\),
\begin{equation}
\LRabs{(\mu P^n-\pi)(f)}
= \LRabs{\int \bigl(P^n(x,\cdot)-\pi\bigr)(f)\,\mu(dx)}
\leq \sup_{x\in\cX}d_{\cF}\bigl(P^n(x,\cdot),\pi\bigr).
\end{equation}
Taking the supremum over \(f\) and then over \(\mu\) gives one inequality. The reverse inequality follows by choosing \(\mu=\delta_x\). Finally, iterating \eqref{eq:ipm-one-step-toward-invariant} from \(\delta_x\) gives
\begin{equation}
d_{\cF}\bigl(P^n(x,\cdot),\pi\bigr)
\leq\theta^n d_{\cF}(\delta_x,\pi)
\leq R_{\cF}\theta^n.
\end{equation}
Taking the supremum over \(x\) proves \textup{(iii)}.
\end{proof}

\paragraph{Contraction and duality.} As usual, we use the same symbol \(P\) for the transition kernel and the associated Markov operator. More precisely, if \( \cU \) is a separable Banach space and \(f:\cX \to \cU \) is a measurable function for which the following integral is well defined, we set
\begin{equation}
\label{eq:markov-operator}
(Pf)(x)
:=\int f(x')\,P(x,dx'),
\qquad x\in\cX,
\end{equation}
where the integral is understood in the Bochner sense when \(f\) is not scalar-valued. Let \(\cB\) be a Banach space of \(\cU\)-valued measurable functions such that \(P:\cB\to\cB\) is bounded. Define the fixed-point space
\begin{equation}
\label{eq:fixed-point-decomposition}
\Fix(P):=\{f\in\cB:Pf=f\},
\end{equation}
and assume there exists a bounded projector \(\Pi:\cB\to\Fix(P)\) satisfying \(\Pi P=\Pi\). The centered space \( \cB_0:=\ker(\Pi) \) is then \(P\)-invariant. Indeed, if \(f\in \cB_0\), then \( \Pi(Pf) = \Pi f = 0 \), so \(Pf\in \cB_0\). By definition of \( \Fix(P) \), any strict norm decay can only be expressed on the complement \( \cB_0 \). Accordingly, writing \(P_0:=P|_{\cB_0}\), we assume that
\begin{equation}
\label{eq:banach-space-contraction}
\norm{P_0f}_{\cB}
\leq \rho \norm{f}_{\cB},
\qquad f\in\cB_0,
\end{equation}
for some \(\rho \in[0,1)\); and we call~\eqref{eq:banach-space-contraction} a \emph{strict contraction} of \( P \) on \( \cB_0 \). In the scalar-valued case \(\cU =\bbR\), contraction on \(\cB_0\) is equivalent to contraction of the Banach adjoint on \(\cB_0^\ast\), since \( \opnorm{P_0^\ast}_{\cB_0^\ast} = \opnorm{P_0}_{\cB_0} \). For probability measures whose restrictions define continuous functionals on \(\cB_0\), introduce the corresponding IPM
\begin{equation}
d_{\cB_0}(\mu,\nu)
:= \sup_{\substack{f\in\cB_0\\\norm{f}_{\cB}\leq1}} \LRabs{\mu(f)-\nu(f)}.
\label{eq:d-b0-sup}
\end{equation}
The identity \((\mu P)(f)=\mu(Pf)\) holds whenever the integrals are well defined. Therefore,
\begin{equation}
d_{\cB_0}(\mu P,\nu P)
= \sup_{\substack{f\in\cB_0\\\norm{f}_{\cB}\leq1}} \LRabs{(\mu-\nu)(Pf)}
\leq\sup_{\substack{g\in\cB_0\\\norm{g}_{\cB}\leq \rho}} \LRabs{(\mu-\nu)(g)}
\leq \rho d_{\cB_0}(\mu,\nu).
\label{eq:dual-ipm-contraction}
\end{equation}
Thus contraction of the Markov operator on a centered space of observables yields one-step contraction in the dual IPM. In particular, if \(\pi P=\pi\), taking \(\nu=\pi\) in \eqref{eq:dual-ipm-contraction} gives contraction toward the invariant measure.

\begin{remark}[Domain of the dual IPM]
The qualification preceding~\eqref{eq:d-b0-sup} is essential. For a general Banach space of measurable functions, a probability measure need not define a continuous functional on \(\cB_0\); in that case, the supremum in~\eqref{eq:d-b0-sup} may be infinite. The natural domain of this IPM is therefore
\begin{equation}
\cP_{\cB_0}(\cX)
:= \left\{ \mu\in\cP(\cX): f\mapsto\mu(f) \text{ defines an element of }\cB_0^\ast \right\}.
\end{equation}
For \(\mu,\nu\in\cP_{\cB_0}(\cX)\),
\begin{equation}
d_{\cB_0}(\mu,\nu)
= \norm{(\mu-\nu)|_{\cB_0}}_{\cB_0^\ast}
<\infty.
\end{equation}
A simple sufficient condition ensuring \(\cP_{\cB_0}(\cX)=\cP(\cX)\) is the continuous embedding
\begin{equation}
\norm{f}_\infty
\leq C_{\cB}\norm{f}_{\cB},
\qquad f\in\cB_0.
\end{equation}
Indeed, this gives \(\abs{\mu(f)}\leq C_{\cB}\norm{f}_{\cB}\) for every probability measure \(\mu\).
\end{remark}

\begin{remark}
By Proposition~\ref{prop:ipm-comparison-relaxation}\textup{(i)}, the inclusion \(\cF_1\subseteq\Lambda\cF_2+\bbR\one\) transfers uniform geometric convergence toward \(\pi\) from \(d_{\cF_2}\) to \(d_{\cF_1}\). Importantly, this comparison alone does not transfer one-step contraction: even if \(d_{\cF_2}\) contracts with factor \(\rho<1\) as in~\eqref{eq:dual-ipm-contraction}, it gives only
\begin{equation}
d_{\cF_1}(\mu P,\nu P)
\leq \Lambda \rho d_{\cF_2}(\mu,\nu),
\end{equation}
whose right-hand side is expressed in the stronger IPM.
\end{remark}

\paragraph{Spectral gap.}
We now specialize to the scalar-valued space \(\cB=L^2(\pi)\). By conditional Jensen's inequality and the invariance of \(\pi\),
\begin{equation}
\norm{Pf}_\pi^2
\leq \pi\bigl(P f^2\bigr)
= \pi(f^2)
= \norm{f}_\pi^2,
\qquad f\in L^2(\pi).
\end{equation}
Thus \(P\) is a bounded, non-expansive operator on \(L^2(\pi)\). Moreover, \(P\one=\one\), and the invariance of \(\pi\) implies that \(L_0^2(\pi)\) is \(P\)-invariant. In the following, let 
\begin{equation}
P_0:=P|_{L_0^2(\pi)},
\qquad \rho:=\opnorm{P_0}_{L_0^2(\pi)}.
\end{equation}
We say that \(P\) admits a positive \emph{absolute spectral gap} if
\begin{equation}
1-\rho>0.
\label{eq:spectral-gap-def}
\end{equation}
We denote by
\begin{equation}
\label{eq:tau-rel-def}
\tau_{\rel}:=\frac{1}{1-\rho}
\end{equation}
the corresponding \emph{relaxation time}, which equals \(1\) when \(\rho=0\) and diverges as \(\rho\uparrow1\).

\begin{remark}
Condition~\eqref{eq:spectral-gap-def} automatically implies \(\Fix(P)=\Span\{\one\}\). Indeed, if \(Pf=f\), then \(g:=f-\pi(f)\one\in L_0^2(\pi)\) satisfies \(Pg=g\), and hence
\begin{equation}
\norm{g}_\pi
=\norm{P_0g}_\pi
\leq \rho\norm{g}_\pi.
\end{equation}
Since \(\rho<1\), it follows that \(g=0\). Consequently, in the notation introduced above,
\begin{equation}
\Pi f=\pi(f)\one,
\qquad \cB_0=L_0^2(\pi).
\end{equation}
\end{remark}

\begin{remark}[On terminology]
Throughout the paper, the term \emph{relaxation} is used in the statistical-physics sense: a system displaced from equilibrium \emph{relaxes} back to it, and the relaxation time is the characteristic timescale of that return. For Markov chains, the term is standard~\citep[Chapter~12]{levin2017markov} for describing the action of \( P \) on centered functions \( f \), as in~\eqref{eq:banach-space-contraction}: any fluctuation around \( \pi(f) \) is damped at a geometrical rate. We reserve \emph{contraction} for the one-step estimate itself and \emph{relaxation} for the resulting geometric decay over \( n \) steps.
\end{remark}

\begin{remark}[On reversibility]
We say Markov operator \( P \) is \( \pi \)-\emph{reversible} if it is self-adjoint on \( L^2(\pi) \), that is, \( \langle f, Pg \rangle_\pi = \langle Pf, g \rangle_\pi \) for \( f,g \in L^2(\pi) \). In such case, the operator \( P_0 \) is also self-adjoint, so \( \sigma(P_0) \subset [-1,1] \) and \( \opnorm{P_0} = \sup\{ \abs{\lambda} : \lambda \in \sigma(P_0) \} \); hence \( 1 - \rho \) is the classical absolute spectral gap.
Without reversibility, \( P_0 \) need not be normal and its spectrum may contain non-real points. Moreover, \( \opnorm{P_0} \) is only bounded below by the spectral radius \( r(P_0) = \lim_{n} \opnorm{P_0^n}^{1/n} \), so \( 1 - \rho \) is a possibly smaller operator-norm gap. We keep the name, following~\citet{fan2021hoeffding}.
Complex scalars play no further role: our arguments use only operator norms, never the spectrum, and for a bounded operator \( A \) on a real Hilbert space it holds that \( \opnorm{A_\bbC} = \opnorm{A} \). We refer to~\citet[Chapter~12]{levin2017markov} for the classical reversible setting and its relation to the relaxation time.
\end{remark}
 
\begin{remark}
It may be tempting to interpret the spectral gap as contraction in the IPM dual to \(L_0^2(\pi)\). This IPM, however, has a restricted domain: if \(\mu,\nu\ll\pi\) have \(L^2(\pi)\)-densities, then
\begin{equation}
d_{2,\pi}(\mu,\nu)
:=\sup_{\substack{f\in L_0^2(\pi)\\ \norm{f}_\pi\leq1}} \LRabs{\mu(f)-\nu(f)}
= \LRnorm{\frac{d\mu}{d\pi}-\frac{d\nu}{d\pi}}_\pi.
\end{equation}
In particular,
\begin{equation}
d_{2,\pi}(\mu,\pi)^2
= \int\left(\frac{d\mu}{d\pi}-1\right)^2d\pi
=: \chi^2(\mu\Vert\pi),
\end{equation}
is the so-called \( \chi^2 \)-divergence. This geometry does not generally allow to  include Dirac measures and is therefore unsuitable for measuring convergence from arbitrary initializations. For this reason, the \( L^2(\pi) \) spectral gap enters our analysis only as a decay bound
\begin{equation}
\norm{P^n g}_\pi
\leq \rho^n\norm{g}_\pi,
\qquad g\in L_0^2(\pi),
\end{equation}
and not as a contraction in the \( L^2(\pi) \)-induced geometry on measures.
\end{remark}

\medskip
Further, we will need the  spectral gap to extend to the Bochner space \(L_0^2(\pi;\cH)\); the following proposition shows that it preserves the gap of the original space \( L^2(\pi) \).

\begin{proposition}
\label{prop:gap-transfer-on-bochner}
Let \(\cH\) be a separable Hilbert space, and let \(P_{\cH}:L^2(\pi;\cH)\to L^2(\pi;\cH)\) denote the Markov operator defined by~\eqref{eq:markov-operator}, with the integral understood in the Bochner sense. Then \(P_{\cH}\) leaves \(L_0^2(\pi;\cH)\) invariant, and
\[
\opnorm{P_{\cH}}_{L_0^2(\pi;\cH)}
= \opnorm{P_0}_{L_0^2(\pi)}.
\]
\end{proposition}
\begin{proof}
By \(\pi\)-invariance, for any \( g \in L_0^2(\pi; \cH) \), it holds that \( \pi(P_\cH g) = \pi(g) \), so \(L_0^2(\pi;\cH)\) is invariant under \(P_{\cH}\). Using the canonical identification \( L_0^2(\pi;\cH) \simeq L_0^2(\pi) \otimes\cH \), the restriction of \(P_{\cH}\) to \(L_0^2(\pi;\cH)\) is \( P_0\otimes I_{\cH} \). Hence, by the standard tensor-product norm identity,
\begin{equation}
\opnorm{P_{\cH}}_{L_0^2(\pi;\cH)}
= \opnorm{P_0\otimes I_{\cH}}_{L_0^2(\pi;\cH)}
= \opnorm{P_0}_{L_0^2(\pi)},
\end{equation}
as claimed.
\end{proof}
\section{Convergence of occupation measure}
\label{sec:kernel-geometry}

In this section, we study properties of Markov chains that contract in the kernel-induced geometry, i.e., with respect to the \emph{maximum mean discrepancy} (MMD) divergence. For the purpose of the learning analysis, this framework will be used primarily as a technical tool within a specific geometric setting. Nevertheless, we emphasize that the high-probability bound presented here for the convergence of the occupation measure \( \wh \mu_T \) to the invariant measure \( \pi \) is, to the best of our knowledge, novel and of independent interest.

\subsection{Uniform relaxation in MMD}

\paragraph{Maximum mean discrepancy.}
Let \(\cH\) be a separable Hilbert space and let \(\phi:\cX\to\cH\) be a measurable feature map. A standard construction starts from a positive-definite kernel \(k:\cX \times \cX \to \bbR\) with associated separable reproducing kernel Hilbert space (RKHS) \(\cH_k\). In this case, one typically takes \(\cH=\cH_k\) and uses the \emph{canonical feature map}
\begin{equation}
\phi(x):=k(x,\cdot),
\quad k(x,y)=\la\phi(x),\phi(y)\ra_{\cH_k},
\qquad x, y \in \cX.
\label{eq:canonical-map-def}
\end{equation}
Define the linear feature operator \(S_\phi \) by
\begin{equation}
(S_\phi h)(x)
:=\la h,\phi(x)\ra_\cH,
\qquad h\in\cH, \; x\in\cX.
\label{eq:s-phi-def}
\end{equation}
Without loss of generality, we may replace \(\cH\) by
\begin{equation}
\cH_\phi
:=\overline{\Span}\{\phi(x):x\in\cX\},
\end{equation}
since \(\ker(S_\phi)=\cH_\phi^\perp\). With this convention, the representation \(h\mapsto S_\phi h\) is injective. The test-function class induced by \(\phi\) is
\begin{equation}
\cF_\phi
:=\left\{S_\phi h: h \in \cH, \, \norm{h}_\cH\leq1\right\}.
\end{equation}
The associated \emph{maximum mean discrepancy} (MMD) is the integral probability metric
\begin{equation}
\MMD_\phi(\mu,\nu)
:=d_{\cF_\phi}(\mu,\nu).
\end{equation}
For any probability measure \(\mu\) such that \(\phi\) is Bochner integrable, define its \emph{mean embedding}~\citep{smola2007hilbert} by \( m_\phi(\mu) := \mu(\phi) \). In particular, the mean embedding is well defined for every \(\mu\in\cP(\cX)\) whenever \(\phi\) is uniformly bounded, i.e.,
\begin{equation}
\label{eq:bounded-feature}
\sup_{x\in\cX}\norm{\phi(x)}_{\cH} \le \kappa < \infty.
\end{equation}
For any \(\mu, \nu \in \cP(\cX) \) whose mean embeddings are well defined, the fact that
the Bochner integral commutes with bounded linear functionals, together with
Hilbert-space duality, gives
\begin{align}
\MMD_\phi(\mu,\nu)
&= \sup_{\norm{h}_\cH\leq1} \LRabs{\int (S_\phi h)(x)\,(\mu-\nu)(dx)} \\
&= \sup_{\norm{h}_\cH\leq1} \LRabs{\la h,m_\phi(\mu)-m_\phi(\nu)\ra_\cH} \\
&= \norm{m_\phi(\mu)-m_\phi(\nu)}_\cH.
\end{align}
For the canonical feature map, this recovers the usual kernel MMD. If \(k\) is \emph{characteristic}~\citep{gretton2012kernel}, its mean embedding is injective and the corresponding MMD is a metric on \(\cP(\cX)\). Under additional regularity conditions on \(k\) and \(\cX\), it also metrizes weak convergence; see, e.g.,~\citet{simon2023metrizing}.

\medskip
The following observation will be used to compare MMDs under centering and bounded linear transformations of feature maps.

\begin{proposition}[Affine transformations of feature maps]
\label{prop:mmd-affine-transformation}
Let \(\cH,\wt\cH\) be separable Hilbert spaces, let \(\phi:\cX\to\cH\) be measurable, and define
\[
\wt\phi(x):=A\phi(x)+b,
\]
where \(A:\cH\to\wt\cH\) is a bounded linear map and \(b\in\wt\cH\). For any \(\mu,\nu\in\cP(\cX)\) with respect to which \(\phi\) is Bochner integrable,
\begin{equation}
\MMD_{\wt\phi}(\mu,\nu)
\leq \norm{A}_{\cH\to\wt\cH}\, \MMD_\phi(\mu,\nu).
\label{eq:mmd-affine-comparison}
\end{equation}
In particular, translations of the feature map leave MMD unchanged.
\end{proposition}

\begin{proof}
Bounded linear operators commute with Bochner integration, so
\begin{equation}
m_{\wt\phi}(\mu)=A m_\phi(\mu)+b.
\end{equation}
Subtracting the corresponding identity for \(\nu\) and taking norms proves~\eqref{eq:mmd-affine-comparison}. Taking \(A=I_\cH\) gives translation invariance.
\end{proof}

\medskip
We begin by stating the minimal assumptions required to establish the main result of this section, followed by a discussion of the conditions under which these assumptions are satisfied.

\begin{assumption}[Uniform geometric MMD relaxation]
\label{ass:uniform-mmd-ergodicity}
There exist constants \( 0 \le C < \infty  \) and \( \theta \in [0, 1) \), such that
\begin{equation}
\sup_{x \in \cX} \MMD_\phi(P^n(x, \cdot), \pi) \leq C \theta^n,
\quad n \ge 0.
\end{equation}
Throughout, we fix constants \(C\) and \(\theta\) satisfying this
assumption and write
\begin{equation}
\tau_\theta :=  C / (1-\theta).
\end{equation}
\end{assumption}

\medskip
A few observations regarding the stated assumption could be made. By Proposition~\ref{prop:ipm-comparison-relaxation}\textup{(ii)}, Assumption~\ref{ass:uniform-mmd-ergodicity} is equivalent to
\begin{equation}
\sup_{\mu_0\in\cP(\cX)}
\MMD_\phi(\mu_0P^n,\pi)
\leq C\theta^n,
\qquad n \ge 0.
\label{eq:mu-0-cond}
\end{equation}
Consequently, one immediately gets that for every \( \mu \in \cP(\cX)\),
\begin{equation}
\MMD_\phi(\mu P^n, \pi) \to 0 \quad \text{as} \; n \to \infty.
\end{equation}
This also gives uniqueness of the invariant measure, provided that the mean embedding map \( m_\phi: \cP(\cX) \to \cH \) is injective; which holds, for example, when the kernel \(k(x,y):=\la\phi(x),\phi(y)\ra_\cH\) is characteristic. Indeed, if \( \nu \in \cP(\cX) \) is another invariant probability measure, then \( \nu P^n = \nu \) for every \( n \ge 0 \). Therefore,
\begin{equation}
\MMD_\phi(\nu P^n, \pi) = \MMD_\phi(\nu, \pi) \leq C \theta^n \xrightarrow[]{n \to \infty} 0.
\end{equation}
By injectivity, it follows that \( \nu = \pi \).

\medskip
The constant \( C \) encodes important property of the chosen geometry. Indeed, applying the condition with \( n = 0 \),
\begin{equation}
\sup_{x \in \cX} \MMD_\phi(\delta_x, \pi) = \sup_{x \in \cX} \norm{ \phi(x) - \pi \phi }_\cH \leq C,
\label{eq:relaxation-mmd-c}
\end{equation}
where we used \( m_\phi(\delta_x) = \phi(x) \). In this sense, \( C \) controls the diameter of the image of \( \cX \) in \( \cH \). In particular, this requirement holds for uniformly bounded \( \phi \) as in~\eqref{eq:bounded-feature}, then one may take \(C=2\kappa\). The same conclusion holds for a canonical feature map of a Mercer kernel on a compact space \(\cX\), since the continuity of \(k\) makes \(x\mapsto k(x,x)\) (effectively) bounded.

\medskip
For a uniformly bounded feature map \( \phi \), Assumption~\ref{ass:uniform-mmd-ergodicity} follows from the stronger one-step contraction toward the invariant measure condition: 
\begin{equation}
\label{eq:mmd-one-step-contraction}
\MMD_\phi(\mu P,\pi)
\leq \theta \, \MMD_\phi(\mu,\pi),
\qquad \forall \mu \in \cP(\cX).
\end{equation}
By Proposition~\ref{prop:ipm-comparison-relaxation}\textup{(iii)},
Condition~\eqref{eq:mmd-one-step-contraction} implies
Assumption~\ref{ass:uniform-mmd-ergodicity} with the same rate
\(\theta\) and \(C=2\kappa\), since
\begin{equation}
\label{eq:mmd-one-step-contraction-kappa}
R_{\cF_\phi}
=\sup_{x\in\cX}\norm{\phi(x)-\pi\phi}_\cH
\leq 2\kappa.
\end{equation}

\subsection{Main results on convergence rate}

We now present the main result of this section, concerning the rate of convergence of the occupation measure \( \wh  \mu_T \) towards the invariant measure \( \pi \).

\begin{theorem}[Empirical occupation measure convergence in MMD]
\label{thm:mmd-occupation-measure-whp}
Let \((X_t)_{t \ge 0}\) be a stationary Markov process with invariant measure \( \pi \), and assume that \(P\) admits an absolute \(L^2(\pi)\) spectral gap \(1 - \rho > 0\). 
Suppose Assumption~\ref{ass:uniform-mmd-ergodicity} hold with constants \(C < \infty\) and \( \theta \in [0, 1) \). 
Then, for every \(T\ge 1\) and every \(\delta\in(0,1)\), with probability at least \(1-\delta\),
\[
\MMD_\phi(\wh \mu_T,\pi)
\leq 2 \norm{\phi}_\pi \sqrt{\frac{\tau_\rel  \log(4/\delta)}{T}}
+ \frac{12 \tau_\theta \sqrt{\tau_\rel} \log(4/\delta)}{T},
\]
where \( \tau_\rel \) is defined in~\eqref{eq:tau-rel-def}.
\end{theorem}

\begin{remark}[Centered feature scale]
Recall that \( \norm{\phi}_\pi = \left( \int \norm{\phi(x)}_\cH^2 \, \pi(dx) \right)^{1/2} \).
The leading term can be sharpened by replacing \(\norm{\phi}_\pi\) with the centered quantity \(\norm{\phi-\pi\phi}_\pi\). Indeed, the proof directly uses \(\norm{\phi-\pi\phi}_\pi\leq\norm{\phi}_\pi\) as a simplification. This sharpened version is invariant under translations of the feature map \(\phi \mapsto \phi+h\), as is the MMD itself. We retain \(\norm{\phi}_\pi\) in the statement for simplicity. This separation will be crucial later for deriving fast rates for statistical learning context in  Section~\ref{sec:stats-learning}. 
\end{remark}

\begin{remark}[Rate for uniformly bounded feature map]
If the feature map is uniformly bounded as in~\eqref{eq:bounded-feature}, and the one-step contraction~\eqref{eq:mmd-one-step-contraction} holds, then Assumption~\ref{ass:uniform-mmd-ergodicity} holds with \(C=2\kappa\) by~\eqref{eq:mmd-one-step-contraction-kappa}, and Theorem~\ref{thm:mmd-occupation-measure-whp} gives, for every \(T\ge1\) and \(\delta\in(0,1)\), with probability at least \(1-\delta\),
\begin{equation}
\MMD_\phi(\widehat\mu_T,\pi)
\le 2\norm{\phi}_\pi\sqrt{\frac{\tau_{\rel}\log(4/\delta)}{T}}
+\frac{24\kappa \sqrt{\tau_{\rel}}\,\log(4/\delta)}{(1-\theta) T}.
\end{equation}
Thus, under~\eqref{eq:mmd-one-step-contraction}, the uniform bound \(\kappa\) enters only the lower-order term, of order \(T^{-1}\), while the leading fluctuation term is governed by the norm \(\norm{\phi}_\pi\) rather than by the worst-case bound. Then, the leading term in the bound depends on the average size of the feature map under \(\pi\), and the uniform bound controls only the non-asymptotic correction.
\end{remark}

\begin{proof}[Sketch of the proof]
Set \(g:=\phi-\pi\phi\). Assumption~\ref{ass:uniform-mmd-ergodicity} gives a solution \(u:=\sum_{n\geq0}P^ng\) of \(u-Pu=g\), with \(\norm{u}_\infty\leq\tau_\theta\). Consequently,
\begin{equation}
\label{mart_dec}
\sum_{t=0}^{T-1}g(X_t)
= M_T+u(X_0)-u(X_T),
\qquad M_T
:=\sum_{t=1}^T\bigl(u(X_t)-Pu(X_{t-1})\bigr)
\end{equation}
is a martingale decomposition with uniformly bounded increments and boundary term. For the conditional variance function \(V_u:=P\norm{u}_\cH^2-\norm{Pu}_\cH^2\), the spectral gap gives
\begin{equation}
\pi(V_u)
\leq\frac{1+\rho}{1-\rho}\norm{\phi-\pi\phi}_\pi^2.
\end{equation}
We first apply the stationary Bernstein inequality to \(\sum_{t=0}^{T-1}V_u(X_t)\), then the localized Pinelis--Bernstein inequality to \(M_T\). Dividing by \(T\) and adding the boundary contribution gives the stated bound. The full proof is given in Appendix~\ref{app:occupation-stationary}.
\end{proof}

We now discuss the result presented here together with a few simple but important corollaries.

\paragraph{Discussion.}
Suppressing universal constants, but retaining the dependence on the confidence level, Theorem~\ref{thm:mmd-occupation-measure-whp} gives
\begin{equation}
\MMD_\phi(\wh\mu_T,\pi)
\lesssim \underbrace{ \norm{\phi}_\pi \sqrt{\frac{\tau_\rel \log(4/\delta)}{T}} }_{\text{variance term}}
+ \underbrace{\frac{ \tau_\theta \sqrt{\tau_\rel} \log(4/\delta)}{T} }_{\text{linear correction}}.
\label{eq:mmd-occupation-discussion}
\end{equation}
The leading term is governed by the absolute \(L^2(\pi)\) spectral gap through the corresponding relaxation time \( \tau_\rel \).
Equivalently, introducing the \emph{effective sample size}
\begin{equation}
T_\eff
:=(1-\rho) \, T
=\frac{T}{\tau_\rel},
\label{eq:t-eff-def}
\end{equation}
the leading term takes the central-limit form \( \norm{\phi}_\pi T_\eff^{-1/2} \). Thus, in the geometry induced by \(\phi\), a trajectory of length \(T\) behaves at the variance level as \(T_\eff\) independent observations. With a spectral gap \( 1 - \rho > 0 \) fixed, the dependence \(\sqrt{\tau_\rel/T}\) is generally unimprovable. Indeed, for a stationary reversible chain and an eigenfunction \(Pf=\rho f\), with \(\pi f=0\) and \(0\leq\rho<1\), stationarity gives
\begin{equation}
\Cov_\pi\bigl(f(X_0),f(X_k)\bigr)
=\la f,P^k f\ra_\pi
=\rho^k\norm{f}_\pi^2.
\end{equation}
Consequently,
\begin{align}
\Var_\pi\!\left(\frac{1}{T}\sum_{t=0}^{T-1}f(X_t)\right)
&=\frac{\norm{f}_\pi^2}{T^2} \left( T+2\sum_{k=1}^{T-1}(T-k)\rho^k \right)
\sim \frac{1+\rho}{1-\rho}\, \frac{\norm{f}_\pi^2}{T}.
\end{align}
This is the eigenfunction case of the Kipnis--Varadhan spectral formula for additive functionals of reversible Markov chains; see \citet{kipnis1986central}. Thus, for chains with nonnegative eigenvalues approaching \(\rho\), the standard deviation is of order \(\sqrt{\tau_\rel/T}\). In this precise sense, our leading dependence on the relaxation time is at the correct covariance scale.

\medskip
The uniform geometric \( \MMD \) relaxation assumption plays a complementary role. It controls not the rate at which the chain forgets its initial condition, but the \emph{total displacement}, measured in the feature space, that this forgetting costs.
Recall the Poisson solution \( u := (I -P)^{-1} (\phi - \pi \phi) \) from the
sketch of the proof: started from the least favorable \( x \in \cX \), the deviations of \( \bbE[\phi(X_n) \mid X_0 = x] \) from \( \pi\phi \) in \( \cH \) accumulate over the entire trajectory to at most \( \norm{u}_\infty \leq \tau_\theta \).

\medskip
The final \(T^{-1}\) term collects three contributions: concentration of the predictable quadratic variation~\eqref{eq:quadratic-variance}, the bounded-increment correction in the Pinelis--Bernstein inequality, and the boundary term in the martingale decomposition, see~\eqref{mart_dec}.
For a fixed confidence level and a fixed feature geometry, the linear correction is therefore asymptotically negligible.
The variance term dominates once
\begin{equation}
T \gtrsim
\frac{\tau_\theta^2 \log(4/\delta)}{\norm{\phi}_\pi^2}.
\end{equation}
Before this threshold, or at very high confidence levels, the linear correction may be significant.

\paragraph{Dependence on dimensionality.}
The bound does not involve the dimension of \(\cX\) explicitly; dimension enters through \(\norm{\phi}_\pi^2 = \Tr\Sigma\), where \(\Sigma := \int\!\phi\otimes\phi\,d\pi\). It is trace class under Assumption~\ref{ass:uniform-mmd-ergodicity} by~\eqref{eq:relaxation-mmd-c}, with \(\Tr\Sigma\le\kappa^2\) for bounded features. For instance, the (unbounded, purely illustrative) identity feature map on \(\bbR^d\) with \(\pi=\cN(0,I_d)\) gives \(\Sigma=I_d\) and \(\norm{\phi}_\pi^2=d\).

\paragraph{Comparison with related results.}
High-probability MMD bounds under weak dependence were used by \citet{cherief2022finite} to analyze minimum-MMD estimators. The closest direct comparison is \citet{wolfer2025variance}, who study empirical kernel mean embeddings of stationary mixing sequences. Under \(\varphi\)-mixing~\citep{doukhan1994mixing}, their variance-aware bound has the schematic form
\begin{equation}
\MMD_\phi(\wh\mu_T,\pi)
\lesssim \sqrt{\frac{v_\phi+\cC_T}{T}} + \sqrt{\frac{v_\phi\norm{\Gamma_T}_2\log(1 / \delta)}{T}}
+ \frac{\kappa \norm{\Gamma_T}_2\log(1/\delta)}{T},
\end{equation}
where \(v_\phi:=\norm{\phi-\pi\phi}_\pi^2\), \(\cC_T\) is a certain cumulative RKHS covariance coefficient, \(\Gamma_T\) is a coupling matrix in the sense of \citet{samson2000concentration}, and \( \kappa \) is a uniform feature map bound. They also derive a blocking analogue under \(\beta\)-mixing. For uniformly ergodic Markov chains --- a condition revisited in Section~\ref{sec:comparison}, where it is shown to imply Assumption~\ref{ass:uniform-mmd-ergodicity} --- \(\norm{\Gamma_T}_2\lesssim t_{\mix}\), where \( t_{\mix} := \min\bigl\{n \ge 1 : \sup_{x \in \cX} \abs{P^n(x,\cdot) - \pi}_{\TV} \le 1/4\bigr\} \); hence their deviation terms scale as \(\sqrt{t_{\mix}/T}\) and \(t_{\mix}/T\). In the language of strong mixing conditions, an absolute \( L^2(\pi) \) spectral gap implies \( \rho \)-mixing at a geometric rate. Since \( \varphi \)-mixing implies \( \rho \)-mixing but not conversely \citep{bradley2005basic}, their \( \sigma \)-field condition is strictly stronger than our gap; what we require in addition is not mixing, but uniform relaxation in the selected MMD geometry. Their \( \beta \)-mixing variant and \( \rho \)-mixing are in general not comparable.

\medskip
We briefly mention related results in stronger IPM geometries. Under a transition-wise \(T_1(K)\) inequality~\citep{bobkov1999exponential} and one-step \(W_1\)-contraction with factor \(\theta<1\), \citet{boissard2011simple} obtains \(W_1\)-concentration at Gaussian scale \(K^{1/2} \tau_\theta/\sqrt{T}\), up to a metric-entropy factor. Under geometric \(W_1\)-contraction, \citet{kloeckner2020empirical} obtains an expected \(C^s\)-IPM error of order \(T_\eff^{-1/2}\), up to logarithmic factors, on bounded subsets of \(\bbR^d\) when \(s>d/2\), with fluctuations of order \(\tau_\theta/\sqrt{T} \). \citet{riekert2022convergence} extends this analysis to unbounded subsets of \(\bbR^d\) under moment assumptions. In our settings, the Hilbertian geometry of MMD permits a direct vector-valued concentration argument and avoids the covering-number prefactor appearing in the \(W_1\) bound.
\paragraph{Non-stationary extension.} Theorem~\ref{thm:mmd-occupation-measure-whp} assumes that the initial condition is distributed according to the invariant measure \(\pi\), which may be restrictive in practice. In what follows, we show that this assumption can be relaxed to a sufficiently regular non-stationary initial distribution \(\nu \in \cP(\cX) \). The result follows from essentially the same argument. The only modification is that the concentration inequality used under stationarity is replaced by its non-stationary counterpart.

\begin{theorem}[Empirical occupation measure convergence in MMD from a non-stationary initial measure]
\label{thm:mmd-occupation-measure-whp-nonstationary}
Let \((X_t)_{t\geq 0}\) be a Markov process with initial distribution \(X_0 \sim \nu\) and invariant measure \(\pi\). Let \(p\in(1,\infty]\), assume that \(\nu \ll \pi\) and that \( d\nu/d\pi \in L^p(\pi) \). Set \( q := p / (p-1) \). Assume that \(P\) admits an absolute \(L^2(\pi)\) spectral gap \(1 - \rho > 0\). Suppose Assumption~\ref{ass:uniform-mmd-ergodicity} holds with constants \( C < \infty \) and \( \theta \in [0,1) \).

\medskip \noindent
For \(t_0 \geq 0\) and \(T \geq 1\), define the occupation measure after a burn-in period of length \(t_0\) by
\[
\wh\mu_{T,t_0}
:= \frac{1}{T}\sum_{t=t_0}^{t_0+T-1}\delta_{X_t}.
\]
Then, for every \(\delta\in(0,1)\), with probability at least \(1-\delta\),
\begin{equation}
\MMD_\phi(\wh\mu_{T,t_0},\pi)
\leq 2\norm{\phi}_\pi \sqrt{\frac{\tau_\rel\log(4/\delta)}{T}}
+ \frac{ 12\tau_\theta \sqrt{ \tau_\rel \, q\log(4/\delta) \log(4D/\delta) } }{T},
\label{eq:mmd-nonstationary-bound}
\end{equation}
where \( D = D(\nu, t_0, p) := \LRnorm{ d(\nu P^{t_0}) / d\pi }_{L^p(\pi)} \) and \( \tau_\rel \) is defined in~\eqref{eq:tau-rel-def}.
\end{theorem}

\begin{proof}[Sketch of the proof]
We use the same Poisson solution and martingale decomposition on the block \(X_{t_0},\ldots,X_{t_0+T}\). The only change is to apply the non-stationary Bernstein inequality of Theorem~\ref{thm:jiang-bernstein-nonstationary} to \(V_u-\pi(V_u)\), with initial distribution \(\nu P^{t_0}\). This introduces \(q\) and \(D\) into the deviation bound for the predictable quadratic variation, while its stationary mean \(T\pi(V_u)\) is unchanged. The localized Pinelis--Bernstein inequality and the boundary estimate then yield~\eqref{eq:mmd-nonstationary-bound}. The full proof is given in Appendix~\ref{app:occupation-nonstationary}.
\end{proof}

\paragraph{Discussion.}
The effect of a non-stationary initial distribution appears only in the lower-order term of order \(T^{-1}\), through the factor
\begin{equation}
\sqrt{q\log(4/\delta)\log\!\big(4D(\nu,t_0,p)/\delta\big)},
\end{equation}
while the leading \(T^{-1/2}\) term remains unchanged.
Taking \(\nu=\pi\) and \(t_0=0\) recovers Theorem~\ref{thm:mmd-occupation-measure-whp}, since \(D(\pi,t_0,p)=1\), and one may take \(q=1\) by choosing \(p=\infty\).
\paragraph{No-gap extension.} We now derive an occupation-measure bound under uniform geometric MMD relaxation alone, to make explicit which conclusions remain valid without stationarity or a spectral gap.

\begin{corollary}[Occupation measure bound without a spectral gap]
\label{cor:no-gap}
Let \((X_t)_{t\geq0}\) be a Markov chain with transition kernel \(P\), invariant measure \(\pi\), and initial distribution \(X_0\sim\nu\), where \(\nu\in\cP(\cX)\) is arbitrary.
Suppose Assumption~\ref{ass:uniform-mmd-ergodicity} holds with constants \(C<\infty\) and \(\theta\in[0,1)\).
Then, for every \(T\geq1\) and every \(\delta\in(0,1)\), with probability at least \(1-\delta\),
\begin{equation}
\label{eq:no-gap}
\MMD_\phi(\wh\mu_T,\pi)
\leq \tau_\theta\left(\sqrt{\frac{2\log(2/\delta)}{T}}
+\frac{5 \log(2/\delta)}{T}\right).
\end{equation}
\end{corollary}

\begin{proof}[Sketch of the proof]
The proof uses the same Poisson equation and martingale decomposition as the main theorem, with the deterministic bound \(\la M \ra_T \leq T \tau_\theta^2\) replacing the spectral gap estimate. The details are given in Appendix~\ref{app:occupation-nogap}.
\end{proof}

\paragraph{Discussion.}
The estimate holds for every initial distribution, including those singular with respect to \(\pi\). For any fixed burn-in period \(t_0\geq0\), it also applies to \( \wh\mu_{T,t_0} \) by replacing \(\nu\) with \(\nu P^{t_0}\). The leading term now depends on the scale \(\tau_\theta\). By comparison, under stationarity and a spectral gap, Theorem~\ref{thm:mmd-occupation-measure-whp} gives the scale \(\norm{\phi}_\pi\sqrt{\tau_\rel}\) in the leading term and \(\tau_\theta\) only enters through \(T^{-1}\) correction.
\begin{corollary}[Almost-sure convergence]
\label{cor:mmd-occupation-weak-as}
Under the conditions of Corollary~\ref{cor:no-gap},
\begin{equation}
\MMD_\phi(\wh\mu_T,\pi)\longrightarrow0
\quad\text{a.s.}.
\end{equation}
If, in addition, \(\cX\) is compact and the induced kernel \(k(x,y):=\la\phi(x),\phi(y)\ra_\cH\) is continuous and characteristic, then
\begin{equation}
\wh\mu_T\rightharpoonup^\star\pi
\quad\text{a.s.}.
\end{equation}
\end{corollary}
\begin{proof}
For \(T\geq2\), let \(r(T)\) be the right-hand side of the bound in Corollary~\ref{cor:no-gap} with \(\delta=T^{-2}\). Then \(r(T)\to0\) and
\begin{equation}
\bbP\!\left\{\MMD_\phi(\wh\mu_T,  \pi)>r(T)\right\}
\leq T^{-2}.
\label{eq:t-2-bound}
\end{equation}
By the Borel--Cantelli lemma, almost surely,
\(\MMD_\phi(\wh\mu_T,\pi)\leq r(T)\) for all sufficiently large \(T\), proving the first claim. Under the additional assumptions, \(\MMD_\phi\) metrizes weak convergence on \(\cP(\cX)\), which gives the second claim; see, e.g.,~\citet{simon2023metrizing}.
\end{proof}
\subsection{Sufficient conditions through metric comparison}
\label{sec:comparison}

In this section we will discuss verifiable conditions under which Assumption~\ref{ass:uniform-mmd-ergodicity} is satisfied together with a few illustrative examples. We first analyze cases where \( \MMD \) is controlled (i.e., upper-bounded) by a different metric, in particular we consider total variation and \( W_1 \).

\paragraph{Control by total variation.}
\label{par:tv-discussion}
Consider \( d_{\cF_{\TV}} \) define in~\eqref{eq:tv-ipm-def}. This choice is particularly natural for Markov chains: every Markov operator preserves \(\cF_{\TV}\), since \(0\leq f\leq1\) implies \(0\leq Pf\leq1\). Consequently, its action on probability measures is non-expansive in total variation. Suppose that there exist constants \(C_{\TV}<\infty\) and \(\theta\in[0,1)\) such that
\begin{equation}
\sup_{x\in\cX}
\norm{P^n(x,\cdot)-\pi}_{\TV}
\leq C_{\TV}\theta^n,
\qquad n \ge 0.
\end{equation}
A Markov chain satisfying this condition is called \emph{uniformly ergodic}. Uniformly ergodic chains and their properties have been studied extensively; see, e.g.,~\citet{meyn2012markov,douc2018markov}. If the feature map is uniformly bounded as in~\eqref{eq:bounded-feature}, then \( \norm{S_\phi h}_\infty\leq\kappa \) whenever \(\norm{h}_{\cH}\leq1\). Consequently, \( \cF_\phi\subseteq 2\kappa \, \cF_{\TV} + \bbR\one \), and~\eqref{eq:ipm-comparison-scaled} yields
\begin{equation}
\MMD_\phi(\mu,\nu)
\leq 2\kappa\norm{\mu-\nu}_{\TV}.
\end{equation}
It follows from~\eqref{eq:weaker-ipm-uniform-convergence} that for uniformly bounded feature maps, uniform ergodicity in total variation implies Assumption~\ref{ass:uniform-mmd-ergodicity}, with \(C=2\kappa C_{\TV}\) and the same geometric rate \(\theta\).

\medskip
Next, we discuss a dynamical system, where convergence in \( \MMD_\phi \) holds while it fails in \( \TV \) norm.

\begin{example}[Deterministic map.]
Take \( \cX = [0, 1] \) and the deterministic map \( T_q(x) = q x \) for \( q \in (0, 1) \). The unique invariant (and therefore unique ergodic) measure is \( \pi = \delta_0 \). Note that \( \norm{P^n(x, \cdot) - \pi}_\TV = 1  \) for any \( x \in (0, 1] \). Now, consider a Gaussian kernel which is characteristic on \( \cX \) and \( \phi: \cX \to \cH \) is known to be Lipschitz, so that
\begin{equation}
\MMD_\phi(P^n(x, \cdot), \pi)
= \norm{ \phi(q^n x) - \phi(0) }_\cH
\leq \Lip(\phi) \abs{ q^n x - 0 }
= \Lip(\phi) \abs{x} q^n
\leq \Lip(\phi) q^n.
\end{equation}
Hence Assumption~\ref{ass:uniform-mmd-ergodicity} holds with \(C=\Lip(\phi)\) and \(\theta=q\). More generally, Lipschitz regularity is not essential. If \(\phi\) is \(\beta\)-Hölder for some \(\beta\in(0,1]\), namely
\begin{equation}
\norm{\phi(x)-\phi(y)}_\cH
\leq L_{\phi,\beta}\abs{x-y}^{\beta},
\qquad x,y\in[0,1],
\end{equation}
then Assumption~\ref{ass:uniform-mmd-ergodicity} holds with
\(C=L_{\phi,\beta}\) and \(\theta=q^\beta\). For example, the canonical feature map of the Laplace kernel is \(1/2\)-H{\"o}lder.
\end{example}

\paragraph{Control by \(W_1\) contraction.}
Suppose that \(\phi:\cX\to\cH\) is \(L_\phi\)-Lipschitz with respect to \(d\). For every \(h\in\cH\) with \(\norm{h}_{\cH}\leq1\),
\begin{equation}
\abs{(S_\phi h)(x)-(S_\phi h)(y)}
\leq \norm{\phi(x)-\phi(y)}_{\cH}
\leq L_\phi d(x,y).
\end{equation}
Consequently, \( \cF_\phi \subseteq L_\phi\cF_{\Lip(d)}+\bbR\one \). It follows immediately from~\eqref{eq:ipm-comparison-scaled} that
\begin{equation}
\label{eq:mmd-controlled-by-w1}
\MMD_\phi(\mu,\nu)
\leq L_\phi W_1(\mu,\nu).
\end{equation}
For more information on the comparison between MMD and Wasserstein distances, including necessary and sufficient conditions for control by \( W_1 \), we refer to \citet{vayer2023controlling}.

\medskip
Now, consider the Markov transition kernel \( P \) such that \emph{Dobrushin condition} (see, e.g., \citealt{ollivier2009ricci,boissard2011simple}) holds with \( \theta < 1 \), i.e.,
\begin{equation}
W_1\bigl(P(x,\cdot),P(y,\cdot)\bigr)
\leq \theta d(x,y),
\qquad x,y\in\cX.
\label{eq:w1-pointwise-contraction}
\end{equation}
By the Kantorovich--Rubinstein duality, the condition lifts to
\begin{equation}
W_1(\mu P,\nu P)
\leq \theta W_1(\mu,\nu),
\qquad \mu,\nu\in\cP_1(\cX).
\end{equation}
Taking \(\mu=\delta_x\), \(\nu=\pi\), and using \(\pi P=\pi\), we obtain
\begin{equation}
W_1\bigl(P^n(x,\cdot),\pi\bigr)
\leq \theta^n W_1(\delta_x,\pi)
= \theta^n\int d(x,y)\,\pi(dy)
\leq \diam(\cX) \theta^n.
\label{eq:w1-uniform-relaxation}
\end{equation}
Using~\eqref{eq:mmd-controlled-by-w1} and taking \( \sup \) over \( \cX \) yields
\begin{equation}
\sup_{x\in\cX}
\MMD_\phi\bigl(P^n(x,\cdot),\pi\bigr)
\leq L_\phi \diam(\cX) \theta^n.
\label{eq:w1-control}
\end{equation}

\begin{remark}
A basic class of examples for which condition~\eqref{eq:w1-pointwise-contraction} holds is given by iterated random functions~\citep{diaconis1999iterated}
\begin{equation}
X_{n+1}=F(X_n,\xi_{n+1}),
\end{equation}
where \( (\xi_n)_{n\geq1} \) is a sequence of i.i.d.\ random innovations, independent of \(X_0\). If, for some \( \theta<1 \),
\begin{equation}
d\bigl(F(x,\xi_1),F(y,\xi_1)\bigr)
\leq \theta d(x,y),
\qquad x,y\in\cX,
\label{eq:w1-coupling}
\end{equation}
almost surely, then the synchronous coupling immediately yields the one-step \(W_1\)-contraction as in~\eqref{eq:w1-pointwise-contraction}. More generally, contraction in Wasserstein-type distances have been studied extensively; see, for example, \citet{hairer2011asymptotic,butkovsky2014subgeometric,
kloeckner2015contraction,durmus2016subgeometric} and the references therein.
\end{remark}

\medskip
The following example shows that Assumption~\ref{ass:uniform-mmd-ergodicity} may hold for Lipschitz feature maps even when convergence in total variation fails: the chain is uniformly contractive in \(W_1\), and this contraction transfers directly to MMD.

\begin{example}[Stochastic dynamical system]
Let \(\cX=[0,1]\) and consider the autoregressive process
\begin{equation}
X_{n+1}=\Lambda X_n+\xi_{n+1},
\qquad \Lambda=\frac13,
\end{equation}
where the innovations \((\xi_n)_{n\geq1}\) are i.i.d., independent of \(X_0\), and take the values \(\{ 0, 2/3 \}\) with probability \(1/2\) each. The unique invariant measure is the \emph{Cantor distribution},
\begin{equation}
\pi
=\Law\!\left( \sum_{j=1}^{\infty}\frac{2B_j}{3^j} \right),
\end{equation}
where \((B_j)_{j\geq1}\) are i.i.d. Bernoulli random variables with parameter \(1/2\). Dobrushin condition~\eqref{eq:w1-pointwise-contraction} holds by synchronous coupling argument in~\eqref{eq:w1-coupling}. Indeed,
\begin{equation}
\abs{\Lambda x+\xi_1-\Lambda y-\xi_1}
=\Lambda \abs{x-y},
\qquad x,y \in [0, 1].
\end{equation}
Since \(\diam(\cX)=1\), the \(W_1\) control criterion~\eqref{eq:w1-control} yields, for every \(L_\phi\)-Lipschitz feature map \(\phi:[0,1]\to\cH\),
\begin{equation}
\sup_{x\in[0,1]}
\MMD_\phi\bigl(P^n(x,\cdot),\pi\bigr)
\leq L_\phi\,\Lambda^n,
\qquad n \ge 0.
\end{equation}
Thus Assumption~\ref{ass:uniform-mmd-ergodicity} holds with \(C=L_\phi\) and \(\theta=\Lambda\).

\medskip
Convergence in total variation nevertheless fails. Indeed, \(\pi\) is atomless and supported on the middle-thirds Cantor set; in particular, it is singular with respect to Lebesgue measure. For every \(x\in[0,1]\) and finite \(n\), the measure \(P^n(x,\cdot)\) has finite support and is therefore mutually singular with \(\pi\). Consequently,
\begin{equation}
\norm{P^n(x,\cdot)-\pi}_{\TV}=1,
\qquad x\in[0,1], \; n \ge 0.
\end{equation}
See~\citet{hutchinson1981fractals} for self-similar sets and measures, \citet{diaconis1999iterated} for the iterated-random-function framework, and \citet{falconer2014fractal} for background on the Cantor set.
\end{example}
\subsection{Operator contraction in kernel-induced spaces}

The previous discussion gives sufficient conditions for Assumption~\ref{ass:uniform-mmd-ergodicity} by comparing the MMD geometry with stronger metrics. We now take a different route and formulate conditions directly in the geometry generated by the feature map. This leads naturally to a spectral-gap condition on the associated power spaces.

\paragraph{Centered feature operator.}
Assume that \(\phi\in L^2(\pi;\cH)\), and define the centered feature map
\begin{equation}
\overline\phi(x):=\phi(x)-\pi\phi,
\qquad x\in\cX.
\end{equation}
By Proposition~\ref{prop:mmd-affine-transformation}, centering preserves the discrepancy: \(\MMD_{\overline\phi}=\MMD_\phi\). Define \(\pi\)-centered version of \( S_\phi \) operator~\eqref{eq:s-phi-def} by
\begin{equation}
\overline S_\phi:\cH\to L_0^2(\pi),
\qquad (\overline S_\phi h)(x)
=\la h,\overline\phi(x)\ra_\cH,
\qquad h \in \cH, \; x \in \cX.
\label{eq:centered-feature-operator}
\end{equation}
Its adjoint is
\begin{equation}
\overline S_\phi^\ast f
=\int f(x)\overline\phi(x)\,\pi(dx),
\qquad f\in L_0^2(\pi).
\end{equation}
Define \( \overline L_\phi :=\overline S_\phi\overline S_\phi^\ast \) acting on \(L_0^2(\pi)\).
\(\overline L_\phi\) is positive, self-adjoint, and trace class, with
\begin{equation}
\Tr(\overline L_\phi)
=\norm{\overline S_\phi}_{\cL_2(\cH,L_0^2(\pi))}^2
=\int\norm{\overline\phi(x)}_\cH^2\,\pi(dx).
\label{eq:centered-feature-operator-trace}
\end{equation}
Let
\begin{equation}
\cG
:=\overline{\Ran(\overline L_\phi)}
=\ker(\overline L_\phi)^\perp
\subseteq L_0^2(\pi),
\label{eq:centered-effective-space}
\end{equation}
and let \((\overline\lambda_j,\overline e_j)_{j\geq1}\) denote the nonzero eigenpairs of \(\overline L_\phi\), where \((\overline e_j)_{j\geq1}\) is an orthonormal basis of \(\cG\). Then
\begin{equation}
\overline L_\phi f
=\sum_{j\geq1}
\overline\lambda_j
\la f,\overline e_j\ra_\pi\overline e_j,
\qquad f\in L_0^2(\pi).
\label{eq:centered-feature-spectral-decomposition}
\end{equation}
This is the standard spectral representation of the integral operator; see~\citet[Section~2]{steinwart2012mercer}.

\paragraph{Power spaces.}
For \(\alpha\geq0\), define
\begin{equation}
\cE_\alpha
:=\Ran\bigl(\overline L_\phi^{\alpha/2}\bigr)
=\left\{ f=\sum_{j\geq1}f_j\overline e_j: \sum_{j\geq1}\overline\lambda_j^{-\alpha}f_j^2<\infty \right\},
\label{eq:centered-hilbert-scale}
\end{equation}
with
\begin{equation}
\norm{f}_\alpha
:=\bnorm{\overline L_\phi^{-\alpha/2}f}_\pi
=\left(
\sum_{j\geq1}\overline\lambda_j^{-\alpha}f_j^2
\right)^{1/2}.
\label{eq:centered-hilbert-scale-norm}
\end{equation}
Here \(\overline L_\phi^0\) is understood as the identity on \(\cG\), so that \(\cE_0=\cG\). This is the standard \(\alpha\)-power space associated with \(\overline L_\phi\); see~\citet[Equation~(36) and Theorem~4.6]{steinwart2012mercer} and~\citet[Section~2]{fischer2020sobolev}. In particular, for \(0<\alpha<1\), \(\cE_\alpha\) agrees, up to equivalence of norms, with the real interpolation space between \(\cG\) and \(\cE_1\).


\paragraph{Dual spaces.}
The negative-order space \(\cE_{-\alpha}\) is the completion of \(\cG\) with respect to
\begin{equation}
\norm{h}_{-\alpha}
:=\bnorm{\overline L_\phi^{\alpha/2}h}_\pi
=\left(
\sum_{j\geq1}\overline\lambda_j^\alpha h_j^2
\right)^{1/2}.
\label{eq:centered-dual-scale-norm}
\end{equation}
Equivalently, \(\cE_{-\alpha}\) is the dual of \(\cE_\alpha\) with respect to the pivot space \(\cG\), giving the Hilbert triple
\begin{equation}
\cE_\alpha\subseteq\cG\subseteq\cE_{-\alpha}.
\end{equation}

\paragraph{Pointwise realization.}
The preceding construction initially produces spaces of
\(L^2(\pi)\)-equivalence classes and does not, by itself, provide pointwise evaluation. We say that \(\cE_\alpha\) admits a pointwise realization on \(\cX\) if there exists a measurable map \( \overline\phi_\alpha:\cX\to\cE_\alpha \) such that
\begin{equation}
f(x)
=\la f,\overline\phi_\alpha(x)\ra_\alpha,
\qquad f\in\cE_\alpha,\quad x\in\cX.
\label{eq:power-space-reproducing-property}
\end{equation}
After measurable representatives of the eigenfunctions have been fixed, the corresponding evaluation representer is
\begin{equation}
\overline\phi_\alpha(x)
=\sum_{j\geq1}
\overline\lambda_j^\alpha
\overline e_j(x)\overline e_j,
\qquad \norm{\overline\phi_\alpha(x)}_\alpha^2
=\sum_{j\geq1}
\overline\lambda_j^\alpha\overline e_j(x)^2.
\label{eq:power-space-feature-map}
\end{equation}
For every finite signed measure \(\upsilon\) such that
\(\int\norm{\overline\phi_\alpha(x)}_\alpha\,\abs{\upsilon}(dx)<\infty\), the functional \(f\mapsto \upsilon(f) \) belongs to \(\cE_{-\alpha}\), and
\begin{align}
\norm{\upsilon}_{-\alpha}
&:=\sup_{\norm{f}_\alpha\leq1}
\LRabs{\int f\,d\upsilon}
=\LRnorm{\int\overline\phi_\alpha(x)\,\upsilon(dx)}_\alpha
=\left( \sum_{j\geq1}\overline\lambda_j^\alpha\upsilon_j^2 \right)^{1/2},
\qquad \upsilon_j:=\int\overline e_j\,d\upsilon.
\label{eq:measure-dual-power-norm}
\end{align}
For \(\mu,\nu\in\cP(\cX)\) for which the corresponding mean embeddings are defined, we write
\begin{equation}
\MMD_\alpha(\mu,\nu)
:=\norm{\mu-\nu}_{-\alpha}
=\LRnorm{
\int\overline\phi_\alpha(x)\,(\mu-\nu)(dx)
}_\alpha.
\label{eq:mmd-alpha-feature-map}
\end{equation}

\medskip
At level \(\alpha=1\), the standard range-space identification identifies \(\cE_1\) isometrically with the \(\pi\)-reduced centered feature space generated by \(\overline\phi\). We choose the pointwise realization compatibly with the original feature representation, so that
\begin{equation}
\MMD_1(\mu,\nu)
=\MMD_\phi(\mu,\nu)
\qquad \mu,\nu\in\cP(\cX),
\label{eq:mmd-one-original-feature}
\end{equation}
whenever the mean embeddings are defined. Equivalently, the original feature representation has no component that is invisible in \(L^2(\pi)\) but visible at individual points. This compatibility is automatic, for example, when \(\overline\phi\) is continuous, \(\supp(\pi)=\cX\), and \(\cH= \cH_{\overline \phi}\).

\medskip
We are now in a position to state the main condition.

\begin{assumption}[Hilbert-scale contraction]
\label{ass:hilbert-scale-gap}
There exist \(\alpha\geq0\), \(\rho_\alpha<1\), and \(\kappa_\alpha<\infty\) such that the following conditions hold:
\begin{enumerate}
\item[(i)] The power space \(\cE_\alpha\) admits a measurable pointwise realization satisfying
\begin{equation}
\label{eq:hilbert-scale-feature-map}
\sup_{x\in\cX}
\norm{\overline\phi_\alpha(x)}_\alpha
\leq\kappa_\alpha.
\end{equation}

\item[(ii)] The Markov operator leaves the fixed pointwise realization of \(\cE_\alpha\) invariant and satisfies
\begin{equation}
\label{eq:hilbert-scale-gap}
\norm{Pf}_\alpha
\leq\rho_\alpha\norm{f}_\alpha,
\qquad f\in\cE_\alpha.
\end{equation}
\end{enumerate}
\end{assumption}

\paragraph{Discussion.} The condition Assumption~\ref{ass:hilbert-scale-gap}~\textup{(i)} is equivalent to the continuous embedding
\begin{equation}
\cE_\alpha\hookrightarrow\ell^\infty(\cX),
\qquad \sup_{x\in\cX} \abs{f(x)}
\leq\kappa_\alpha\norm{f}_\alpha.
\label{eq:power-space-uniform-embedding}
\end{equation}
This is commonly referred to as the \(L^\infty\)-embedding, or bounded-embedding, condition for the power space. Indeed,
\begin{equation}
\sup_{x\in\cX}\norm{\overline\phi_\alpha(x)}_\alpha
= \sup_{x\in\cX}\sup_{\norm{f}_\alpha\leq1}\abs{f(x)}.
\end{equation}
It is an additional property of the power space and does not follow from the abstract range construction alone. Once a pointwise realization is available, its uniform boundedness follows, for example, if \(\cX\) is a compact metric space and \(x \mapsto \overline\phi_\alpha(x)\) is continuous as an \(\cE_\alpha\)-valued map. In that case, \(\cE_\alpha \hookrightarrow C(\cX)\); see, for example,~\citet[Section~4]{Steinwart2008-ut}. Also, under the pointwise realization, the feature map \(\overline\phi_\alpha\) is Bochner integrable with respect to \(\pi\). Moreover, since every \(f\in\cE_\alpha\) is centered,
\begin{equation}
\la f,\pi\overline\phi_\alpha\ra_\alpha
=\pi f
=0,
\qquad f\in\cE_\alpha,
\end{equation}
and hence \(\pi\overline\phi_\alpha=0\). Consequently,
\begin{equation}
\norm{\delta_x-\pi}_{-\alpha}
=\norm{\overline\phi_\alpha(x)}_\alpha,
\qquad x\in\cX.
\label{eq:centered-dirac-dual-norm}
\end{equation}
Thus condition~\eqref{eq:hilbert-scale-feature-map} is equivalent to
\begin{equation}
\sup_{x\in\cX}
\norm{\delta_x-\pi}_{-\alpha}
\leq\kappa_\alpha.
\end{equation}
By Hilbert-space duality, the condition Assumption~\ref{ass:hilbert-scale-gap}~\textup{(ii)} induces a bounded adjoint operator
\begin{equation}
P^\ast:\cE_{-\alpha}\to\cE_{-\alpha},
\qquad \norm{P^\ast h}_{-\alpha}
\leq\rho_\alpha\norm{h}_{-\alpha}.
\end{equation}
Moreover, this adjoint agrees with the push-forward action on embedded finite signed measures. Indeed, for every such measure \(\upsilon\) and every \(f\in\cE_\alpha\),
\begin{equation}
\la P^\ast\upsilon,f\ra
= \la\upsilon,Pf\ra
= \int Pf\,d\upsilon
= \int f\,d(\upsilon P)
= \la\upsilon P,f\ra.
\end{equation}
Hence \(P^\ast\upsilon=\upsilon P\) as elements of \(\cE_{-\alpha}\).

\begin{proposition}[Hilbert-scale contraction implies uniform geometric MMD relaxation]
\label{prop:mmd-alpha-bound}
Let Assumption~\ref{ass:hilbert-scale-gap} hold. Then, for every \(n\geq0\) and \(\mu\in\cP(\cX)\),
\begin{equation}
\MMD_\alpha(\mu P^n,\pi)
\leq \kappa_\alpha \rho_\alpha^n.
\label{eq:mmd-alpha-decay}
\end{equation}
Moreover, if \(\norm{\phi}_\infty\leq\kappa<\infty\), then, for
\(\alpha \in (0,1]\),
\begin{equation}
\MMD_\phi(\mu P^n,\pi)
\leq \kappa^{1-\alpha} \kappa_\alpha \rho_\alpha^n,
\qquad \mu\in\cP(\cX).
\end{equation}
\end{proposition}

\begin{proof}
By~\eqref{eq:dual-ipm-contraction},
\begin{equation}
\MMD_\alpha(\mu P,\pi)
\leq\rho_\alpha\MMD_\alpha(\mu,\pi).
\end{equation}
Moreover, the centered pointwise realization gives
\begin{equation}
\sup_{x\in\cX}\MMD_\alpha(\delta_x,\pi)
=\sup_{x\in\cX}\bnorm{\overline\phi_\alpha(x)}_\alpha
\leq\kappa_\alpha.
\end{equation}
Proposition~\ref{prop:ipm-comparison-relaxation}\textup{(ii)--(iii)} therefore proves the first claim.

\medskip
For the second claim, let \(\upsilon:=\mu P^n-\pi\). By the compatible level-one realization~\eqref{eq:mmd-one-original-feature},
\begin{align}
\MMD_\phi^2(\upsilon,0)
&=\norm{\upsilon}_{-1}^2
=\sum_{j\geq1}\overline\lambda_j\upsilon_j^2
\leq \bnorm{\overline L_\phi}^{1-\alpha} \sum_{j\geq1}\overline\lambda_j^\alpha\upsilon_j^2
= \bnorm{\overline L_\phi}^{1-\alpha} \MMD_\alpha^2(\upsilon,0).
\end{align}
Furthermore, by~\eqref{eq:centered-feature-operator-trace},
\begin{equation}
\bnorm{\overline L_\phi}
\leq\Tr(\overline L_\phi)
=\int\norm{\phi(x)-\pi\phi}_\cH^2\,\pi(dx)
\leq\kappa^2.
\end{equation}
Consequently,
\begin{equation}
\MMD_\phi(\mu P^n,\pi)
\leq \kappa^{1-\alpha} \MMD_\alpha(\mu P^n,\pi),
\end{equation}
and the result follows from~\eqref{eq:mmd-alpha-decay}.
\end{proof}

\begin{remark}
The map
\begin{equation}
J_\alpha:\cE_{-\alpha}\to\cG,
\qquad J_\alpha h:=\overline L_\phi^{\alpha/2}h,
\end{equation}
initially defined on \(\cG\), extends to an isometric isomorphism. Indeed,
if \(h=\sum_{j\geq1}h_j\overline e_j\), then
\begin{equation}
\norm{J_\alpha h}_\pi^2
=\sum_{j\geq1}\overline\lambda_j^\alpha h_j^2
=\norm{h}_{-\alpha}^2.
\end{equation}
Under this identification, the contraction condition is equivalently written as
\begin{equation}
\bnorm{ \overline L_\phi^{\alpha/2} P^\ast \overline L_\phi^{-\alpha/2} }_{\cG \to \cG}
\leq\rho_\alpha.
\label{eq:conjugated-hilbert-scale-gap}
\end{equation}
The conjugated operator is initially defined on \(\Ran(\overline L_\phi^{\alpha/2})\subseteq\cG\), and the condition means that it admits a bounded extension to \(\cG\).
\end{remark}

\begin{remark}[Relation with the \(L^2\)-spectral gap]
As \(\alpha\downarrow0\), the spaces \(\cE_\alpha\) formally approach \(\cG\). More precisely, for every \(h\in\cG\),
\begin{equation}
\norm{h}_{-\alpha}\longrightarrow\norm{h}_\pi.
\end{equation}
This convergence is pointwise, but not uniform over the \(L^2(\pi)\)-unit ball. Indeed, if \(\overline L_\phi\) has infinite rank, then \(\overline\lambda_j\to0\), and for every fixed \(\alpha>0\),
\begin{equation}
\norm{\overline e_j}_{-\alpha}
=\overline\lambda_j^{\alpha/2}\longrightarrow0,
\qquad \norm{\overline e_j}_\pi=1.
\end{equation}
Consequently, contraction in \(\cE_{-\alpha}\) for a fixed \(\alpha>0\), however small, does not by itself imply an absolute spectral gap on \(L_0^2(\pi)\). If \(\overline L_\phi\) is injective on \(L_0^2(\pi)\) and the contraction condition holds along a sequence \(\alpha_n\downarrow0\) with \(\sup_n\rho_{\alpha_n}<1\), then passage to the limit recovers an absolute \(L^2(\pi)\)-spectral gap. In particular, when \(\cG=L_0^2(\pi)\), the contraction condition in Assumption~\ref{ass:hilbert-scale-gap}\textup{(ii)} at \(\alpha=0\) is precisely the usual absolute \(L^2(\pi)\)-spectral-gap condition.
\end{remark}

\medskip
We next give a useful sufficient condition for
Assumption~\ref{ass:hilbert-scale-gap}.

\begin{corollary}[Commuting Markov and feature operators]
\label{cor:commuting-hilbert-scale}
Suppose that \(P\) admits an absolute \(L^2(\pi)\) spectral gap \( 1 - \rho > 0 \) and commutes with the centered feature operator:
\[
P\overline L_\phi=\overline L_\phi P
\qquad\text{on }L_0^2(\pi).
\]
Fix \(\alpha\geq0\) such that Assumption~\ref{ass:hilbert-scale-gap}~\textup{(i)} holds, and suppose that \(P\) preserves the corresponding fixed pointwise realization of \(\cE_\alpha\). Then Assumption~\ref{ass:hilbert-scale-gap} holds with
\begin{equation}
\rho_\alpha
:=\opnorm{P}_{\cE_\alpha}
=\opnorm{P|_{\cG}}_{\cG}
\leq\rho.
\label{eq:commuting-scale-contraction}
\end{equation}
In particular, if \(\cG=L_0^2(\pi)\), then \(\rho_\alpha=\rho\).
\end{corollary}

\begin{proof}
Commutation implies that \(P\) preserves
\(\cG=\overline{\Ran(\overline L_\phi)}\) and, by functional calculus, commutes with \(\overline L_\phi^{\alpha/2}\) on \(\cG\). Every \(f\in\cE_\alpha\) has a unique representation \(f=\overline L_\phi^{\alpha/2}g\), with \(g\in\cG\), and \(\norm{f}_\alpha=\norm{g}_\pi\). Consequently,
\begin{equation}
Pf=\overline L_\phi^{\alpha/2}Pg,
\qquad
\norm{Pf}_\alpha=\norm{Pg}_\pi
\leq\rho\norm{g}_\pi
=\rho\norm{f}_\alpha.
\end{equation}
Since \(g\mapsto\overline L_\phi^{\alpha/2}g\) is an isometric isomorphism from \(\cG\) onto \(\cE_\alpha\), taking the supremum over its unit ball proves \eqref{eq:commuting-scale-contraction}. Together with the assumed pointwise preservation and bounded realization, this verifies Assumption~\ref{ass:hilbert-scale-gap}.
\end{proof}

The pointwise-preservation requirement is automatic if \(\supp(\pi)=\cX\), the chosen realization of \(\cE_\alpha\) consists of bounded continuous functions, and \(P\) maps bounded continuous functions to bounded continuous functions: the equality obtained in \(L^2(\pi)\) then holds everywhere by continuity and full support.

\medskip
In what follows, we analyze these conditions in two different settings, both obtained by adding noise to deterministic maps on periodic state spaces. For a general background on random perturbations of dynamical systems and for a detailed spectral analysis of randomly perturbed expanding maps, we refer the reader to~\citet{Freidlin2012-fp} and~\citet{baladi1993spectra}.

\begin{example}[Noisy rotation on the torus]
\label{ex:noisy-irrational-rotation}
Let \(\cX=\bbT^d\) and consider
\begin{equation}
X_{t+1}=X_t+\omega+\xi_{t+1}\pmod 1,
\end{equation}
where \(\omega\in\bbT^d\) and the noises are i.i.d. with density \(\upsilon\in L^1(\bbT^d)\). The invariant measure is the Lebesgue measure \(\pi(dx)=dx\). For \(s>d/2\), let \(\phi_s(x)=k_s(x,\cdot)\) be the canonical feature
map of the centered Sobolev kernel
\begin{equation}
k_s(x,y)
:=\sum_{m\in\bbZ^d\setminus\{0\}}
(1+\abs{m}^2)^{-s}e^{2\pi i m\cdot(x-y)}.
\end{equation}
For every \(\alpha>0\) such that \(s\alpha>d/2\),
Assumption~\ref{ass:hilbert-scale-gap} holds with
\begin{equation}
\rho_\alpha=\rho
=\sup_{m\neq0}\abs{\wh\upsilon(m)}<1.
\end{equation}
Here the Markov operator and the centered kernel integral operator commute: translation preserves the Sobolev norms, and contraction is determined entirely by the noise.
\end{example}

\begin{example}[Noisy expanding map on the torus]
\label{ex:noisy-expanding-map}
Let \(\cX=\bbT\) and consider
\begin{equation}
X_{t+1}
=\ell X_t+\omega+\eta Z_{t+1}\pmod 1,
\end{equation}
where \(\ell\in\{2,3,\ldots\}\), \(\omega\in\bbT\), \(\eta>0\), and the \(Z_t\) are i.i.d. standard Gaussian random variables. The invariant measure is \(\pi(dx)=dx\). For \(s>1/2\), use the canonical feature map of the centered Sobolev kernel
\begin{equation}
\wt k_s(x,y)
:=\sum_{m\in\bbZ\setminus\{0\}}
\abs{m}^{-2s}e^{2\pi i m(x-y)}.
\end{equation}
The exact operator norms in the two geometries are
\begin{equation}
\rho=e^{-2\pi^2\eta^2},
\qquad \rho_\alpha
=\ell^{s\alpha}e^{-2\pi^2\eta^2}.
\end{equation}
Consequently, Assumption~\ref{ass:hilbert-scale-gap} holds whenever
\begin{equation}
s\alpha>\frac12,
\qquad \ell^{s\alpha}e^{-2\pi^2\eta^2}<1.
\label{eq:noisy-expanding-two-conditions}
\end{equation}
This example demonstrates an explicit separation between the two contraction coefficients: the expansion factor \(\ell\) enters \(\rho_\alpha\), but not \(\rho\). The \(L^2(\pi)\) spectral gap is positive for every \(\eta>0\), whereas one-step contraction in the smoother geometry requires the noise to offset the amplification caused by the deterministic map.
\end{example}
In both examples, under the indicated conditions, Proposition~\ref{prop:mmd-alpha-bound} gives
\begin{equation}
\MMD_\alpha(\mu P^n,\pi)
\leq\kappa_\alpha\rho_\alpha^n,
\qquad \mu\in\cP(\cX),\quad n\geq0.
\end{equation}
The embedding constants and a Fourier analysis proof are given in Appendix~\ref{app:fourier-examples}, together with a discussion of the separate roles of the deterministic map and the noise in Remark~\ref{rem:drift-noise}.

\begin{remark}
The two conditions~\eqref{eq:noisy-expanding-two-conditions} are equivalent to
\begin{equation}
\frac{1}{2s}
<\alpha
<\frac{2\pi^2\eta^2}{s\log\ell}.
\end{equation}
Thus an admissible \(\alpha>0\) exists if and only if
\begin{equation}
\eta^2>\frac{\log\ell}{4\pi^2},
\label{eq:expanding-noise-threshold}
\end{equation}
i.e., the noise level is sufficiently large.
\end{remark}

\begin{remark}[Related interpolation geometries]
At a conceptual level, the Hilbert scale above is part of a broader line of work connecting MMD and \(L^2\)-type geometries through operator regularization. The \emph{de-regularized maximum mean discrepancy} (DrMMD) interpolates between MMD and the \(\chi^2\)-divergence \citep{chen2025regularized}, while the regularized Stein variational gradient flow interpolates between the RKHS-based Stein flow and the \(L^2\)-based Wasserstein flow \citep{he2025regularized}. Related constructions combine RKHS control with an \(L^2\)-gradient penalty \citep{tian2026sobolev}, or use spectral filtering and truncation of covariance operators to strengthen MMD-based kernel tests \citep{hagrass2024spectral,lacroix2026nonasymptotic}. Although developed for different purposes, these approaches share the principle of regularizing a tractable kernel geometry toward stronger \(L^2\)-type information.
\end{remark}

\section{Application to statistical learning of Markov chains}
\label{sec:stats-learning}

We now apply the occupation-measure bound derived in the previous section to the task of learning the one-step prediction function of a Markov chain from a single trajectory with kernel ridge regression (KRR). The main point is that the standard KRR error decomposition naturally exposes the geometry relevant to the problem: the stability of the estimator is controlled by the convergence of the empirical occupation measure in a regularization-dependent MMD induced by the preconditioned covariance operators. This provides a direct link between the statistical learning rate and the convergence properties of the underlying Markov chain.

\medskip
Following up on~\citet{kachaiev2026learning}, we consider the problem of learning the one-step forecasting function
\begin{equation}
f_\star(x) := \int x' \, P(x, dx'),
\qquad x \in \cX,
\end{equation}
from a single trajectory of the Markov chain. More precisely, we observe
\begin{equation}
\cD_T
:= \big\{ x_0 = X_0(\omega),\, x_1 = X_1(\omega),\, \ldots,\, x_T = X_T(\omega) \big\},
\qquad \text{for a fixed} \; \omega \in \Omega.
\end{equation}
For simplicity, in what follows we restrict our attention to the case \( \cX \subseteq \bbR \). The extension to vector-valued states is rather straightforward and is largely a matter of technical bookkeeping, which is not the focus of the present discussion. We assume moreover that the invariant measure has a finite second moment, namely
\begin{equation}
\int \abs{x}^2 \, \pi(dx) < \infty .
\label{eq:pi-second-moment}
\end{equation}
Since \( \pi \) is invariant, this also ensures that
\begin{equation}
\int \abs{x'}^2 \, P(x,dx') \, \pi(dx) < \infty .
\end{equation}

\medskip \noindent
Under condition~\eqref{eq:pi-second-moment}, we define the population objective as
\begin{equation}
\cR(f)
:= \int \left( x' - f(x) \right)^2 \, \pi(dx) P(x,dx'),
\qquad f \in L^2(\pi).
\end{equation}
The function \( f_\star \) is the unique minimizer of this risk over \( L^2(\pi) \), as it encodes the conditional expectation of the next state given the present one. Consequently, for every \( f \in L^2(\pi) \), the \emph{excess risk} satisfies
\begin{equation}
\cE(f,f_\star)
:= \cR(f) - \cR(f_\star)
= \norm{f-f_\star}_\pi^2,
\qquad f \in L^2(\pi).
\end{equation}
The above relation is classical in statistical learning.

\subsection{Regularized least squares}

Let \(\cX\) be a separable metric space and let \(k: \cX \times \cX \to \bbR\) be a continuous positive-definite kernel with separable RKHS \(\cH\), and throughout this section let \( \phi(x):=k(x,\cdot) \) denote its canonical feature map. We consider the space of estimators parametrized as follows, for \( w \in \cH \) let
\begin{equation}
f_w(x) := \la w, \phi(x) \ra_\cH,
\qquad x \in \cX.
\end{equation}
For the clarity of notation we drop the subscript \( w \), where the meaning is clear from the context. The estimator is obtained as a unique minimizer of the regularized least-squares objective, i.e.,
\begin{equation}
\wh w_\lambda := \argmin_{w \in \cH} \left\{ \frac{1}{T} \sum_{t=0}^{T-1} \big(x_{t+1} - \la w, \phi(x_t) \ra_\cH\big)^2 + \lambda \norm{w}_\cH^2  \right\},
\qquad \lambda > 0.
\label{eq:hat-w-lambda-def}
\end{equation}
The corresponding estimator is denoted \( \wh f_\lambda(x) := \la \wh w_\lambda, \phi(x) \ra_\cH \). The classical result~\citep{caponnetto2007optimal} is that \( \wh w_\lambda \) admits closed form given by
\begin{equation}
\wh w_\lambda := \big(\wh \Sigma + \lambda I \big)^{-1} \wh h,
\end{equation}
where
\begin{equation}
\wh \Sigma := \frac{1}{T} \sum_{t=0}^{T-1} \phi(x_t) \otimes \phi(x_t) \in \cL_2(\cH)
\qandq
\wh h := \frac{1}{T} \sum_{t=0}^{T-1} x_{t+1} \, \phi(x_t) \in \cH.
\end{equation}
For the sake of notational simplicity, for \( \lambda > 0 \) we write \( \wh \Sigma_\lambda := \wh \Sigma + \lambda I \). To streamline the discussion, we also introduce population-level quantities
\begin{equation}
\Sigma := \int \phi(x) \otimes \phi(x) \, \pi(dx)
\qandq
h := \int x' \, \phi(x) \, \pi(dx) P(x, dx'),
\end{equation}
together with the population regularized estimator
\begin{equation}
w_\lambda := \Sigma_\lambda^{-1} h,
\qquad \text{where }  \Sigma_\lambda := \Sigma + \lambda I.
\end{equation}
Let \( S:\cH\to L^2(\pi) \) be an operator acting as \( Sw:=f_w \), so that \(\Sigma=S^\ast S\).
\subsection{Learning guarantees}

To derive a bound on the excess risk of the estimator~\eqref{eq:hat-w-lambda-def}, we collect the assumptions into three blocks: assumptions on the Markov chain, on the feature map and its induced geometry, and the properties of the statistical learning problem.

\paragraph{Markov chain.} Let \( (X_t)_{t \ge 0} \) be a time-homogeneous Markov process  taking values in measurable space \( (\cX, \cB(\cX)) \). Let \( P: \cX \times \cB(\cX) \to [0, 1] \) denote corresponding Markov transition kernel.

\begin{assumption}[Markov chain]
\label{ass:markov-chain}
We assume
\begin{enumerate}[label=\textup{(\roman*)}]
\item \textbf{Stationarity.} The Markov operator \(P\) admits an invariant measure \( \pi \in \cP_2(\cX) \). The process \((X_t)_{t\geq 0}\) is initialized at stationarity, i.e., \( X_0 \sim \pi \).

\item \textbf{Spectral gap.} The Markov operator \(P\) admits a strictly positive absolute \(L^2(\pi)\) spectral gap \( 1-\rho>0 \).

\item \textbf{Bounded innovations.} There exist constants \(M,\sigma<\infty\) such that, for every \(x\in\cX\),
\[
\abs{x'-f_\star(x)}\leq M, \; P(x,\cdot)\text{-a.s.},
\qquad \text{and} \qquad
\int \abs{x'-f_\star(x)}^2\,P(x,dx')
\leq \sigma^2.
\]
\end{enumerate}
\end{assumption}

\noindent
The finite-second-moment condition ensures that \( f_\star \in L^2(\pi) \) by conditional Jensen's inequality and invariance of \(\pi\). Notice that bounded innovations are not needed for this conclusion by themselves; they control the conditional regression noise used in the learning argument. Condition~\textup{(iii)} is therefore a transition-kernel assumption relative to the prediction problem, rather than a mixing assumption.

\paragraph{Feature map.} To state the next assumption on the choice of feature map, we define an auxiliary object. Let
\begin{equation}
\Phi(x):=\phi(x)\otimes\phi(x)\in\cL_2(\cH).
\end{equation}
Note that \( \cL_2(\cH) \) endowed with the Hilbert--Schmidt norm is itself a separable Hilbert space. Let \(\left(\cE_\beta,\norm{\cdot}_\beta\right)_{\beta\geq0}\) be the centered Hilbert scale~\eqref{eq:centered-hilbert-scale} induced by \(\Phi\).

\begin{assumption}[Feature map and induced geometry]
\label{ass:feature-map}
We assume the following.
\begin{enumerate}[label=\textup{(\roman*)}]
\item \textbf{Boundedness.}
There exists \(\kappa<\infty\) such that
\begin{equation}
\sup_{x\in\cX}\norm{\phi(x)}_\cH
\leq\kappa.
\end{equation}

\item \textbf{Hilbert-scale contraction for the unregularized covariance feature.} There exist \(\alpha\in(0,1]\), \(\rho_\alpha\in [0,1) \), and \(\kappa_\alpha<\infty\) such that \(\cE_\alpha\) admits a measurable pointwise realization with evaluation representer \( \Phi_\alpha:\cX\to\cE_\alpha \) satisfying
\begin{equation}
\sup_{x\in\cX}
\norm{\Phi_\alpha(x)}_\alpha
\leq \kappa_\alpha.
\end{equation}
Moreover, the Markov operator leaves this fixed pointwise realization of \(\cE_\alpha\) invariant and satisfies
\begin{equation}
\norm{Ph}_\alpha
\leq \rho_\alpha\norm{h}_\alpha,
\qquad h\in\cE_\alpha.
\end{equation}
\end{enumerate}
\end{assumption}

\noindent
Assumption~\ref{ass:feature-map}\textup{(ii)} is precisely the Hilbert-scale contraction condition of Assumption~\ref{ass:hilbert-scale-gap} applied to the feature map \(\Phi\); we refer the reader to the discussion following that assumption for further details. We emphasize that the assumption is induced by the unregularized covariance feature \( \Phi \). In particular, the constants \(\alpha\), \(\rho_\alpha\), and \(\kappa_\alpha\) do not depend on the regularization parameter \(\lambda\).

\paragraph{Learning problem.} The following assumptions describe the interaction between the geometry induced by the feature map and the regularity of the target function. They are standard in the fields of statistical learning and kernel methods.

\begin{assumption}[Learning problem]
\label{ass:learning-problem}
We assume
\begin{enumerate}[label=\textup{(\roman*)}]
\item \textbf{Source condition.}
There exist \(r\in[1/2,1]\) and \(g_\star\in\cH\) such that
\[
w_\star := \Sigma^{r-1/2}g_\star,
\qquad f_\star
= S w_\star.
\]
We denote \(G_\star:=\norm{g_\star}_\cH\).

\item \textbf{Polynomial eigenvalue decay.}
There exist constants \(B>0\) and \(b>1\) such that
\[
\lambda_j(\Sigma)\leq B j^{-b},
\qquad \forall j\geq 1,
\]
where \((\lambda_j(\Sigma))_{j\geq 1}\) is the non-increasing sequence of eigenvalues of \(\Sigma\).
\end{enumerate}
\end{assumption}

\noindent
The source condition implies well-specifiedness. Since \(\ker(S)=\ker(\Sigma)\), we may choose \(g_\star\) so that
\(w_\star \in \overline{\Ran(\Sigma)}=(\ker S)^\perp\). Consequently, \(w_\star\) is the minimum-\(\cH\)-norm representer of the \(L^2(\pi)\)-equivalence class of \(f_\star\):
\begin{equation}
\norm{w_\star}_\cH
= \min\left\{
\norm{w}_\cH:
f_w=f_\star\ \text{in }L^2(\pi)
\right\}.
\end{equation}
In particular, \( h=S^\ast f_\star=\Sigma w_\star \).

\paragraph{Discussion of the assumptions.}
Stationarity and the existence of an absolute \(L^2(\pi)\) spectral gap are standard quantitative ergodicity conditions for Markov chains; see, e.g., ~\citet{douc2018markov,fan2021hoeffding,jiang2026bernstein}. Stationarity avoids an additional dependence on the initial distribution, while the spectral gap controls the variance accumulated along the trajectory. The absolute spectral-gap condition does not, by itself, imply that \(\pi\) is the unique invariant probability measure on the whole state space. When the feature map is characteristic, uniqueness may instead follow from the uniform MMD convergence assumption, as discussed above, but this is not needed for the main arguments. The boundedness of the feature map is standard in KRR analysis and is satisfied, in particular, by commonly used bounded kernels. The Hilbert-scale contraction for the covariance feature plays a different role and is specific to the framework developed here: as will be shown further, it provides control of the stability term in the KRR decomposition. The well-specified model, the source condition, and the polynomial eigenvalue decay are likewise classical assumptions in the analysis of regularized least squares; see, e.g.,~\citet{caponnetto2007optimal,rudi2015less}.

\paragraph{Bounded domain.}
Since \(f_\star=Sw_\star\), the bounds on features and innovations give
\begin{equation}
\abs{X_{t+1}}
\leq \abs{f_\star(X_t)}+M
\leq \kappa\norm{w_\star}_\cH+M
=:R_\star \; \text{a.s.}
\end{equation}
By stationarity, \(\pi([-R_\star,R_\star])=1\), so the assumptions effectively restrict the chain to a compact interval. An extension to unbounded, for example sub-Gaussian or sub-exponential, innovations would require uniform conditional Bernstein-type moment bounds for
\begin{equation}
Z_t
:= \bigl(X_{t+1}-f_\star(X_t)\bigr)\phi_\lambda(X_t),
\end{equation}
together with a corresponding concentration inequality for Hilbert-valued martingales. We leave this extension for future work.

\medskip
With these assumptions in place, we are now ready to state the main result of this section.

\begin{theorem}
\label{thm:final-sample-rate}
Let Assumptions~\ref{ass:markov-chain}--\ref{ass:learning-problem} hold and set \(p:=1/b\). With the choice \( \lambda_T \asymp T^{-\frac{1}{2r+p}} \), for every \(\delta \in (0,1)\) and for \(T\) large enough that~\eqref{eq:t-lower-bound} holds with \( \lambda = \lambda_T \), the estimator~\eqref{eq:hat-w-lambda-def} satisfies, with probability at least \(1-\delta\),
\[
\cE\big(\wh f_{\lambda_T},f_\star\big)
:= \bnorm{\wh f_{\lambda_T}-f_\star}_\pi^2
\lesssim T^{-\frac{2r}{2r+p}},
\]
up to logarithmic factors in \(1/\delta\).
\end{theorem}

The assumptions of Theorem~\ref{thm:final-sample-rate} are simultaneously satisfied by the noisy expanding map of Example~\ref{ex:noisy-expanding-map} with \(\ell=2\), \(\omega=0\), Gaussian noise scale \(\eta:=\sqrt{\log 4 / (2\pi^2)}\), and uniform stationary initialization. Representing \(\bbT\) by \([0,1)\), yields
\begin{equation}
X_{t+1}=2X_t+\eta Z_{t+1}\pmod 1,
\qquad X_0\sim\pi,
\; \pi(dx)=dx,
\end{equation}
where the \(Z_t\) are independent standard Gaussian random variables, independent of \(X_0\). For the periodic Sobolev kernel
\begin{equation}
k(x,y)
=\sum_{m\in\bbZ}(1+m^2)^{-1}e^{2\pi i m(x-y)}
\end{equation}
one may take
\begin{equation}
\rho=\frac14,
\qquad \text{and} \qquad
\rho_\alpha
=\frac14\left(\frac85\right)^{\alpha/2},
\quad \text{for } \alpha\in(1/2,1],
\end{equation}
together with \(M=1\), \(\sigma^2=1/4\), and \(b=2\), hence \(p=1/2\). Here \(\sigma^2\) denotes the conditional innovation-variance bound. The source condition holds for every \(r\in[1/2,1]\), and the feature, embedding, and source constants are finite and independent of the regularization parameter \(\lambda\). In particular, taking \(\alpha=1\) gives \(\rho_1=1/\sqrt{10}>\rho\), illustrating the distinction between contraction in \(L^2(\pi)\) and in the geometry induced by the covariance feature map \(\Phi\). The complete construction and verification are given in Appendix~\ref{app:noisy-doubling-learning}.

\medskip
The proof of Theorem~\ref{thm:final-sample-rate} proceeds by first establishing the following non-asymptotic error bound and then optimizing it with respect to the regularization parameter. The full proof is given in Appendix~\ref{app:sample-rate}.

\begin{proposition}[Non-asymptotic error bound]
\label{prop:krr-error-bound}
Let Assumptions~\ref{ass:markov-chain}--\ref{ass:learning-problem} hold. Fix \( \lambda > 0 \). Set
\[
B_\alpha:=\kappa^{2(1-\alpha)}\kappa_\alpha.
\]
Suppose that
\begin{equation}
T
\geq \max\left\{ \frac{64 \kappa^2 \tau_{\rel} \cN(\lambda) \log(8/\delta)}{\lambda}, \,
\frac{48B_\alpha\sqrt{\tau_{\rel}} \log(8/\delta)} {(1-\rho_\alpha)\lambda} \right\}.
\label{eq:t-lower-bound}
\end{equation}
Then, for every \( \delta \in (0,1) \) with probability at least \(1-\delta\),
\begin{align}
\bnorm{\wh f_\lambda-f_\star}_\pi
\leq\;&
2\lambda^rG_\star
+ 2\sigma\sqrt{\frac{2 \cN(\lambda) \log(8/\delta)}{T}} \notag \\
&+ 4\sigma \left( \frac{\kappa^2 \tau_\rel \, \cN(\lambda)}{\lambda} \right)^{1/4} \frac{\log(8/\delta)^{3/4}}{T^{3/4}}
+ \frac{4\sqrt{5 \tau_\rel}\,\sigma\kappa \log(8/\delta)}{\sqrt{\lambda}\,T}
+ \frac{4\kappa M \log(8/\delta)}{3\sqrt{\lambda}\,T}.
\label{eq:error-bound-final}
\end{align}
\end{proposition}

\begin{proof}[Sketch of the proof]
The empirical normal equation separates the error into a stability factor, a regularization bias, and a stochastic term. The source condition bounds the bias by \(\lambda^r G_\star\). To control stability, set
\begin{equation}
\Delta_\lambda
:=\Sigma_\lambda^{-1/2}(\wh\Sigma-\Sigma)\Sigma_\lambda^{-1/2},
\qquad \Phi_\lambda(x)
:=\Sigma_\lambda^{-1/2}\Phi(x)\Sigma_\lambda^{-1/2}.
\end{equation}
The key identity is
\begin{equation}
\norm{\Delta_\lambda}_{\cL_2(\cH)}
=\MMD_{\Phi_\lambda}(\wh\mu_T,\pi).
\end{equation}
Hilbert-scale contraction transfers from \(\Phi\) to a uniform MMD relaxation bound for \(\Phi_\lambda\). Theorem~\ref{thm:mmd-occupation-measure-whp} and condition~\eqref{eq:t-lower-bound} then ensure \(\norm{\Delta_\lambda}_{\cL_2(\cH)}\leq1/2\) with high probability, which bounds the stability factor by \(2\). The stochastic term is an average of martingale differences
\begin{equation}
\label{eq:reg-residual}
Z_t
:= (X_{t+1}-f_\star(X_t))\Sigma_\lambda^{-1/2}\phi(X_t).
\end{equation}
We control their predictable quadratic variation~\eqref{eq:quadratic-variance} by the stationary Bernstein inequality and apply the localized Pinelis--Bernstein inequality. Combining these estimates by a union bound proves the result. 
\end{proof}

\paragraph{Discussion.}
The non-asymptotic form of the bound in Proposition~\ref{prop:krr-error-bound} contains more information than the asymptotic rate in Theorem~\ref{thm:final-sample-rate}. In particular, it separates the regularization bias from the stochastic error and makes explicit how the properties of the Markov chain affect the finite-sample performance of the estimator.

\medskip
The leading stochastic term is
\begin{equation}
\sqrt{\frac{\cN(\lambda)}{T}},
\end{equation}
up to the noise and confidence factors. Somewhat surprisingly, this term does not contain the relaxation time \(\tau_{\rel} \).
This is a consequence of the martingale structure of the regression residual \( Z_t \) defined in~\eqref{eq:reg-residual}.
Dependence on relaxation time \( \tau_\rel \) enters instead through the stability of the empirical covariance operator. More precisely, up to logarithmic factors, the preconditioned covariance fluctuation satisfies
\begin{equation}
\norm{\Delta_\lambda}_{\cL_2(\cH)}
\lesssim \sqrt{\frac{\tau_{\rel} \, \cN(\lambda)}{\lambda T}}
+ \frac{B_\alpha\sqrt{\tau_{\rel}}}{(1-\rho_\alpha)\lambda T}.
\end{equation}
Thus the covariance component is governed by the effective sample size \( T_\eff \) introduced in~\eqref{eq:t-eff-def}. For a given Markov chain, the quantities \(\tau_{\rel}\), \(B_\alpha\), and \((1-\rho_\alpha)^{-1}\) are constants, so the asymptotic exponent remains the same as in the i.i.d. case. Their effect is nevertheless important at finite sample sizes: the sample size required by \eqref{eq:t-lower-bound} grows linearly in \( \tau_{\rel} \), and the lower-order terms of \eqref{eq:error-bound-final} carry factors \( \tau_{\rel}^{1/4} \) and \( \tau_{\rel}^{1/2} \). Since \( \tau_{\rel} \to \infty \) as \( \rho \uparrow 1 \), these may be arbitrarily large.

\medskip
Finally, the \(T^{-3/4}\) term is a lower-order correction arising from the concentration of the predictable quadratic variation~\eqref{eq:quadratic-variance}. It does not alter the leading statistical rate: under \(r\geq1/2\) and polynomial effective dimension growth, it is dominated by the balance between the regularization bias and the leading stochastic term.

\medskip
In the i.i.d. regime, corresponding formally to \(P(x,\cdot)=\pi\), one has \(\rho=0\) and hence \(\tau_{\rel}=1\). The kernel-induced contraction is then immediate after one step, and the bound reduces, up to lower-order Bernstein corrections, to the familiar KRR form
\begin{equation}
\lambda^r G_\star
+ \sqrt{\frac{\cN(\lambda)}{T}}.
\end{equation}
Under the polynomial effective-dimension bound \(\cN(\lambda)\lesssim\lambda^{-p}\), balancing these two terms gives \( \lambda_T\asymp T^{-1/(2r+p)} \), and, therefore
\begin{equation}
\bnorm{\wh f_{\lambda_T}-f_\star}_\pi
\lesssim T^{-r/(2r+p)},
\qquad \cE(\wh f_{\lambda_T},f_\star)
\lesssim T^{-2r/(2r+p)}.
\end{equation}
This coincides with the classical source--capacity rate for i.i.d. KRR and, over the corresponding classes of i.i.d. regression problems, with the minimax-optimal rate; see, for example,~\citet{caponnetto2007optimal}.

\medskip
For comparison, the real-valued setting in~\citet[Section~3]{kachaiev2026learning} studies the same one-step prediction problem from a single trajectory, but under different assumptions on the underlying Markov process. The present result is not a strict improvement under identical assumptions: it gives a sharper, geometry- and capacity-sensitive rate for stationary trajectories, whereas~\citet{kachaiev2026learning} provides the more general treatment of nonstationary initialization under stronger assumption on convergence.
\section{Conclusion}

In this paper, we studied the convergence of empirical occupation measures of Markov chains in MMD geometry. Under an absolute \(L^2(\pi)\) spectral gap and uniform relaxation in the corresponding MMD, we obtained a non-asymptotic high-probability bound whose leading term has the central-limit scale
\begin{equation}
\frac{\norm{\phi}_\pi}{\sqrt{(1-\rho) \, T}}.
\end{equation}
The result separates the contribution of the usual relaxation time from that of the kernel-induced geometry and does not require reversibility. We then applied this framework to KRR for learning the one-step prediction function from a single trajectory. The KRR decomposition identifies the stability error with convergence of the occupation measure in a regularization-dependent MMD geometry. Under a source condition and covariance operator polynomial eigenvalue decay, this yields the rate
\begin{equation}
\cE\big(\wh f_{\lambda_T},f_\star\big)
\lesssim T^{-\frac{2r}{2r+p}},
\qquad \lambda_T\asymp T^{-\frac{1}{2r+p}},
\end{equation}
up to logarithmic factors.

\medskip
Several questions remain open. It would be useful to replace the boundedness assumptions on the feature map and innovations by suitable moment or tail conditions, thereby covering unbounded state spaces and systems with sub-Gaussian or sub-exponential additive noise. Sharper control of the lower-order concentration terms, as well as extensions to misspecified models, vector-valued states, and more general initial distributions, are also of interest. More fundamentally, a systematic understanding of the relation between contraction in regularized kernel-induced geometries and classical spectral or Wasserstein contraction could provide criteria for choosing the kernel and the regularization parameter adaptively from data.

\section*{Acknowledgments}
S. V. has been supported by the MUR Excellence Department Project awarded to Dipartimento di Matematica, Universita di Genova, CUP D33C23001110001 and the PNRR project "Harmonic Analysis and Optimization in Infinite-Dimensional Statistical - Future Artificial Intelligence  Fair – Spoke 10" (CUP J33C24000410007). S. V. is a member of the Gruppo Nazionale per l'Analisi Matematica, la Probabilità e le loro Applicazioni (GNAMPA) of the Istituto Nazionale di Alta Matematica (INdAM). O.\ K.\ and L.\ R.\ acknowledge the financial support of the European Commission (Horizon Europe grant ELIAS 101120237). L.\ R.\ also acknowledges the financial support of the Ministry of Education, University and Research (FARE grant ML4IP R205T7J2KP). 

{\small
\bibliography{references}
\bibliographystyle{apalike}
}

\appendix
\section{Proofs}
This appendix contains the proofs of the occupation measure bounds and the learning guarantees. We first collect the auxiliary results used in the proofs, then establish the occupation measure bound under stationarity, extend it to non-stationary initial distributions, and derive the bound without a spectral gap. We conclude with the proofs of the learning error bounds and the resulting sample rate.

\subsection{Auxiliary results}
\paragraph{Bernstein inequalities for Markov additive functionals.}
We rewrite the Bernstein inequalities of
\citet[Theorems~1 and~3]{jiang2026bernstein} using  the notation used throughout our paper. Their results are stated on a standard Borel state space, a condition automatically satisfied here because \(\cX\) is Polish, and Polish spaces with their Borel \( \sigma \)-algebra are standard Borel.

\medskip \noindent
For \(\eta\in[0,1)\), define
\begin{equation}
\alpha_1(\eta)
:=\frac{1+\eta}{1-\eta},
\qquad \alpha_2(\eta)
:= \begin{cases}
1/3, & \eta=0,\\[2mm]
5/(1-\eta), & \eta\in(0,1).
\end{cases}
\label{eq:jiang-alpha-constants}
\end{equation}

\begin{theorem}[Stationary Bernstein inequality]
\label{thm:jiang-bernstein-stationary}
Let \((X_t)_{t\geq0}\) be a time-homogeneous Markov chain with transition kernel \(P\), invariant measure \(\pi\), and initial distribution \(X_0\sim\pi\). Assume that \(P\) has a positive absolute \(L^2(\pi)\)-spectral gap \( 1 - \rho > 0 \), where \( \rho := \opnorm{P}_{L_0^2(\pi)} \). Let \(T\geq1\), and let \(f_0,\ldots,f_{T-1}:\cX\to\bbR\) be measurable functions satisfying,
for some \(c\in(0,\infty)\),
\[
\pi(f_t)=0,
\qquad \text{and} \qquad
\abs{f_t(x)}\leq c,
\qquad x\in\cX,\; 0\leq t\leq T-1.
\]
Set
\[
S_T:=\sum_{t=0}^{T-1}f_t(X_t),
\qquad \sigma_T^2
:=\frac{1}{T}\sum_{t=0}^{T-1}\pi(f_t^2).
\]
Then, for every \(\varepsilon>0\),
\begin{equation}
\bbP_\pi\left(S_T\geq\varepsilon\right)
\leq \exp\left\{ -\frac{\varepsilon^2} {2\left( T\alpha_1(\rho)\sigma_T^2 + \alpha_2(\rho)c\varepsilon \right)} \right\}.
\label{eq:jiang-bernstein-stationary-tail}
\end{equation}
\end{theorem}

\begin{theorem}[Nonstationary Bernstein inequality]
\label{thm:jiang-bernstein-nonstationary}
Let \((X_t)_{t\geq0}\) be a time-homogeneous Markov chain with transition kernel \(P\), invariant measure \(\pi\), and initial distribution \(X_0\sim\nu\). Assume again that \( \rho := \opnorm{P}_{L_0^2(\pi)} <1 \). Let \(T\geq1\), and let \(f_0,\ldots,f_{T-1}:\cX\to\bbR\) be measurable functions satisfying
\[
\pi(f_t)=0,
\qquad \text{and} \qquad
\abs{f_t(x)}\leq c,
\qquad x\in\cX,\; 0\leq t\leq T-1,
\]
for some \(c\in(0,\infty)\). Define \(S_T\) and \(\sigma_T^2\) as in Theorem~\ref{thm:jiang-bernstein-stationary}. Suppose, in addition, that \(\nu\ll\pi\) and
\begin{equation}
D_{\nu,p}
:= \Bnorm{\frac{d\nu}{d\pi}}_{L^p(\pi)}
<\infty
\label{eq:jiang-bernstein-dnup}
\end{equation}
for some \(p\in(1,\infty]\), and let \( q:=p/(p-1)\in[1,\infty) \), with the convention \(q=1\) when \(p=\infty\). Then,
for every \(\varepsilon>0\),
\begin{equation}
\bbP_\nu\left(S_T\geq\varepsilon\right)
\leq D_{\nu,p} \exp\left\{ -\frac{\varepsilon^2} {2q\left( T\alpha_1(\rho)\sigma_T^2 + \alpha_2(\rho)c\varepsilon \right)} \right\}.
\label{eq:jiang-bernstein-nonstationary-tail}
\end{equation}
\end{theorem}
Neither theorem requires \(P\) to be reversible. In particular, no self-adjointness assumption is imposed: dependence is controlled by the absolute operator norm \(\rho=\opnorm{P}_{L_0^2(\pi)}\). No separate irreducibility or Harris-recurrence assumption is required. When applying Theorem~\ref{thm:jiang-bernstein-nonstationary} to a trajectory block starting at time \(t_0\), one simply replaces \(\nu\) by \(\nu_{t_0}:=\nu P^{t_0}\).

\paragraph{Martingales.} The following is the localized version of Pinelis--Bernstein inequality. For the original Pinelis--Bernstein inequality with an almost surely bounded predictable quadratic variation we refer the reader to ~\citet[Theorem~3.4]{pinelis1994optimum} and~\citet[Proposition~A.3]{TarresYao}.

\begin{theorem}[Localized Pinelis--Bernstein inequality]
\label{thm:pinelis-bernstein}
\label{thm:pinelis-bernstein-hilbert}
Fix \(T\geq1\), and suppose that there exists \(R<\infty\) such that
\[
\norm{d_t}_\cH\leq R \;\text{a.s.},
\qquad t=1,\ldots,T.
\]
Then, for every \(V\geq0\) and every \(\delta\in(0,1)\),
\begin{equation}
\label{eq:localized-pinelis-bernstein}
\bbP\left\{ \norm{M_T}_\cH
> \sqrt{2V\log(2/\delta)} + \frac{2R\log(2/\delta)}{3},
\; \la M\ra_T\leq V \right\}
\leq\delta.
\end{equation}
\end{theorem}

\begin{proof}
For \(t=1,\ldots,T\), set
\begin{equation}
a_t:=\bbE_{t-1}\norm{d_t}_\cH^2,
\qquad Q_t:=\sum_{s=1}^t a_s,
\qquad Q_0:=0.
\end{equation}
Define the predictably truncated increments \( \wt d_t:=d_t\one_{\{Q_t\leq V\}} \). Since \(Q_t\) is \(\cF_{t-1}\)-measurable,
\begin{equation}
\bbE_{t-1}\wt d_t
= \one_{\{Q_t\leq V\}}\bbE_{t-1}d_t
=0.
\end{equation}
Thus, \( \wt M_s:=\sum_{t=1}^s\wt d_t \) for  \( s=0,\ldots,T \) is a martingale. Moreover, \(\norm{\wt d_t}_\cH\leq R\), and, since \((Q_t)_t\) is nondecreasing,
\begin{equation}
\la\wt M\ra_T
= \sum_{t=1}^T \one_{\{Q_t\leq V\}}a_t
\leq V
\; \text{a.s.}.
\end{equation}
The standard Pinelis--Bernstein inequality therefore gives
\begin{equation}
\bbP\left\{ \norm{\wt M_T}_\cH > \sqrt{2V\log(2/\delta)} +\frac{2R\log(2/\delta)}{3} \right\}
\leq\delta.
\end{equation}
On the event
\begin{equation}
\{\la M\ra_T\leq V\}
= \{Q_T\leq V\},
\end{equation}
we have \(Q_t\leq V\) for every \(t=1,\ldots,T\), and hence \(\wt M_T=M_T\). This proves the claim.
\end{proof}

\subsection{Occupation measure convergence under stationarity}
\label{app:occupation-stationary}
\paragraph{Proof of Theorem~\ref{thm:mmd-occupation-measure-whp}.} We start by noting that by definition,
\begin{equation}
\MMD_\phi(\wh \mu_T, \pi)
= \Bnorm{ \frac{1}{T} \sum_{t=0}^{T-1} \phi(X_t) - \pi \phi }_\cH;
\label{eq:mmd-partial-sum}
\end{equation}
thus the problem reduces to a concentration estimate for a Hilbert-valued additive functional of the Markov chain. We first decompose the partial sum in~\eqref{eq:mmd-partial-sum} into a certain martingale term and a boundary term. The martingale term is then controlled by a suitable Hilbert-valued Bernstein inequality, while the bounded boundary term vanishes at order \(T^{-1}\).

\paragraph{Poisson equation.} For a suitable space of observables and a bounded operator \(P\) acting on it, the \emph{Poisson equation} associated with \(P\) and a function \(f\) is given by
\begin{equation}
\label{eq:poisson-def}
u - P u = f.
\end{equation}
This equation has been used extensively in studies of Markov chains; see, e.g., \citet{glynn1996liapounov, douc2018markov}.

\medskip
Let \(E\) be a Banach space and let \(\cB\) be a Banach space of \(E\)-valued measurable functions such that \(P:\cB \to \cB\) is bounded. We use the fixed-point decomposition introduced in~\eqref{eq:fixed-point-decomposition}, with
\(\cB_0=\ker(\Pi)\). The following standard consequence of the Neumann-series argument is recorded for later use; see also~\citet[Section~21]{douc2018markov} and~\citet[Proposition~C.4]{kachaiev2026learning}.

\begin{lemma}
\label{prop:poisson-admits-solution}
Assume that there exist constants \( C>0 \) and \( \rho \in (0,1) \)
such that
\[
\norm{P^k f}_{\cB}
\leq C\rho^k\norm{f}_{\cB},
\qquad \forall f\in\cB_0,\quad k\geq0.
\]
Then \( (I-P): \cB_0 \to \cB_0 \) is invertible. In particular, for every
\(g \in \cB_0\), the Poisson equation
\[
u-Pu=g
\]
admits a unique solution \(u \in \cB_0\), given by
\[
u := \sum_{k=0}^{\infty}P^kg,
\qquad \norm{u}_{\cB}
\leq \frac{C}{1-\rho}\norm{g}_{\cB}.
\]
\end{lemma}

\begin{proof}
The assumed geometric bound implies
\begin{equation}
\sum_{k=0}^{\infty} \bnorm{(P|_{\cB_0})^k}
\leq \frac{C}{1-\rho},
\end{equation}
so the conclusion follows from the Neumann-series theorem. Uniqueness in \(\cB_0\) follows since \(v=Pv\) implies \(\norm{v}_{\cB}\leq C\rho^k\norm{v}_{\cB}\) for every \(k\geq0\).
\end{proof}

\medskip
By boundedness, \( \phi \in L^2(\pi; \cH) \) so that \( \phi - \pi \phi \in L_0^2(\pi; \cH) \). By Proposition~\ref{prop:gap-transfer-on-bochner}, \( P \) and \( \phi - \pi \phi \) satisfy the condition of Lemma~\ref{prop:poisson-admits-solution}. Therefore, the Poisson equation
\begin{equation}
u - Pu = \phi - \pi \phi
\end{equation}
admits a solution \( u \in L_0^2(\pi; \cH) \) given by
\begin{equation}
u(x)
:= \sum_{n=0}^\infty P^n (\phi(x) - \pi \phi)
= \sum_{n=0}^\infty (P^n \phi(x) - \pi \phi),
\quad x \in \cX.
\end{equation}
Fix \( x \in \cX \). For every \( n \in \bbN_0 \),
\begin{equation}
(P^n \phi)(x) = \int \phi(x') P^n(x, dx') = m_\phi(P^n(x, \cdot)) = m_\phi(\delta_x P^n).
\end{equation}
Therefore,
\begin{equation}
\norm{(P^n \phi)(x) - \pi \phi}_\cH
= \MMD_\phi( \delta_x P^n, \pi )
= \MMD_\phi( P^n(x, \cdot), \pi )
\leq C \theta^n, 
\end{equation}
Now, by triangle inequality,
\begin{equation}
\norm{ u(x) }_\cH
\leq \sum_{n=0}^\infty \norm{(P^n \phi)(x) - \pi \phi}_\cH
\leq C \cdot \sum_{n=0}^\infty \theta^n
\leq \frac{C}{1 - \theta}.
\end{equation}
As the bound holds for any \( x \in \cX \) uniformly,
\begin{equation}
\label{eq:u-linf}
\norm{u}_\infty
:= \sup_{x \in \cX} \norm{u(x)}_\cH
\leq \frac{C}{1 - \theta}.
\end{equation}

\paragraph{Martingale decomposition.}
Let \( \cF_t:=\sigma(X_0,\ldots,X_t) \) for \( t\geq0 \). Using~\citet[Lemma~C.5]{kachaiev2026learning}, the partial sum in~\eqref{eq:mmd-partial-sum} admits the representation
\begin{equation}
\sum_{s=0}^{T-1}\bigl(\phi(X_s)-\pi\phi\bigr)
= \sum_{t=1}^T d_t+R_T,
\end{equation}
where
\begin{equation}
d_t:=u(X_t)-Pu(X_{t-1}),
\qquad R_T:=u(X_0)-u(X_T).
\end{equation}
By the Markov property,
\begin{equation}
\bbE\left[d_t\mid\cF_{t-1}\right]
= \bbE\left[u(X_t)\mid X_{t-1}\right] -Pu(X_{t-1})
=0.
\end{equation}
Consequently, \( (d_t)_{t\geq1} \) is a martingale difference sequence and \( M_T:=\sum_{t=1}^T d_t \) is a martingale. Moreover, for \(t=1,\ldots,T\),
\begin{equation}
\norm{d_t}_\cH
\leq \norm{u(X_t)}_\cH+\norm{Pu(X_{t-1})}_\cH
\leq 2 \norm{u}_\infty.
\label{eq:dt-infty-norm}
\end{equation}
Its predictable quadratic variation is
\begin{align}
V_T
:=\la M\ra_T
&=\sum_{t=1}^T\bbE_{t-1}\norm{d_t}_\cH^2
=\sum_{s=0}^{T-1}
\bbE\left[ \norm{u(X_{s+1})-Pu(X_s)}_\cH^2 \mid X_s \right].
\end{align}
Define the measurable observable \(V_u: \cX \to \bbR\) by
\begin{equation}
V_u(x)
:= \int_\cX \norm{u(x')-Pu(x)}_\cH^2\,P(x,dx'),
\qquad \text{so that} \quad
V_T=\sum_{s=0}^{T-1}V_u(X_s).
\end{equation}
By Markovianity of \( (X_t)_t \), \( V_u \) does not depend on \( t \). By Eq.~\eqref{eq:dt-infty-norm},
\begin{equation}
\norm{V_u}_\infty
\leq 4 \norm{u}_\infty^2
\leq \frac{4 C^2}{(1-\theta)^2}.
\label{eq:vu-infty-norm}
\end{equation}
We now use the Bernstein inequality in Theorem~\ref{thm:jiang-bernstein-stationary}. Since \(V_u\geq0\),
\begin{equation}
\Var_\pi(V_u)
\leq \pi(V_u^2)
\leq \norm{V_u}_\infty \pi(V_u).
\label{eq:vu-variance}
\end{equation}
By \( \pi \)-invariance,
\begin{equation}
\pi(V_u)
:= \int \left( P \norm{u}_\cH^2(x) - \norm{P u(x)}_\cH^2  \right)  \pi(dx)
= \norm{u}_\pi^2 - \norm{P u}_\pi^2.
\end{equation}
Hence
\begin{equation}
\pi(V_u)
=\la u-Pu,u+Pu\ra_\pi
=\la \phi-\pi\phi,u+Pu\ra_\pi
\leq \norm{\phi-\pi\phi}_\pi \big(\norm{u}_\pi+\norm{Pu}_\pi\big).
\end{equation}
Since, by Proposition~\ref{prop:gap-transfer-on-bochner}, \( \opnorm{P}_{L_0^2(\pi;\cH)} = \rho \),
\begin{equation}
\norm{u}_\pi
= \norm{(I-P)^{-1} (\phi-\pi\phi)}_\pi
\leq \frac{1}{1-\rho}\norm{\phi-\pi\phi}_\pi,
\qquad \norm{Pu}_\pi\leq\rho\norm{u}_\pi,
\end{equation}
we obtain
\begin{equation}
\pi(V_u)
\leq \frac{1+\rho}{1-\rho} \norm{\phi-\pi\phi}_\pi^2
\leq \frac{1+\rho}{1-\rho}\norm{\phi}_\pi^2.
\label{eq:pi-vu}
\end{equation}
Combining~\eqref{eq:vu-infty-norm},~\eqref{eq:vu-variance}, and~\eqref{eq:pi-vu} gives
\begin{equation}
\Var_\pi(V_u)
\leq
\frac{4 C^2}{(1-\theta)^2}
\frac{1+\rho}{1-\rho}
\norm{\phi}_\pi^2.
\label{eq:vu-variance-bound}
\end{equation}
Applying Theorem~\ref{thm:jiang-bernstein-stationary} to the centered observable \(\wt V_u:=V_u-\pi(V_u)\), which is bounded and measurable. Set
\begin{equation}
\sum_{t=0}^{T-1}\wt V_u(X_t) = V_T-T\pi(V_u),
\qquad \Var_\pi(\wt V_u) = \Var_\pi(V_u),
\qquad \norm{\wt V_u}_\infty \leq \norm{V_u}_\infty.
\end{equation}
We obtain, for every \(\varepsilon>0\),
\begin{equation}
\label{eq:fan-stationary}
\bbP\big\{ V_T-T\pi(V_u)>\varepsilon \big\}
\leq \exp\left\{ -\frac{\varepsilon^2}{2\left( T\alpha_1(\rho)\Var_\pi(V_u) + \alpha_2(\rho)\norm{V_u}_\infty \, \varepsilon \right)} \right \},
\end{equation}
where \( \alpha_1, \alpha_2 \) are given in~\eqref{eq:jiang-alpha-constants}. Consequently, letting \(L_\delta:=\log(4/\delta)\), with probability at least \(1-\delta/2\),
\begin{align}
V_T
&\leq T\pi(V_u)
+ \sqrt{ 2T\alpha_1(\rho)\Var_\pi(V_u)L_\delta }
+ 2\alpha_2(\rho)\norm{V_u}_\infty L_\delta \\
&\leq \frac{1+\rho}{1-\rho} \, T \norm{\phi}_\pi^2
+ 2\sqrt{2}\, \frac{1+\rho}{1-\rho} \, \tau_\theta\norm{\phi}_\pi \sqrt{T L_\delta}
+ 40\tau_\rel\tau_\theta^2 L_\delta \\
&\leq 2\tau_\rel \left( \sqrt{T}\norm{\phi}_\pi + 2\sqrt{5}\,\tau_\theta\sqrt{L_\delta} \right)^2.
\label{eq:vt-upper-bound}
\end{align}
Here, in the last inequality we used \(1+\rho\leq2\).

\medskip
On the event~\eqref{eq:vt-upper-bound}, invoke Theorem~\ref{thm:pinelis-bernstein-hilbert}. By~\eqref{eq:dt-infty-norm}, the martingale increments satisfy
\begin{equation}
\norm{d_t}_\cH
\leq 2\norm{u}_\infty
\leq 2\tau_\theta.
\end{equation}
Hence, with probability at least \(1-\delta\),
\begin{align}
\frac{1}{T} \norm{M_T}_\cH
\leq
\frac{\sqrt{2 V_T L_\delta}}{T}
+ \frac{4\tau_\theta L_\delta}{3 T}
\leq2\norm{\phi}_\pi \sqrt{\frac{\tau_\rel \, L_\delta}{T}}
+ \frac{4\sqrt{5}\,\tau_\theta\sqrt{\tau_\rel}  \, L_\delta}{T}
+ \frac{4\tau_\theta L_\delta}{3T}
\label{eq:mt-bound}
\end{align}
On the other hand, by the triangle inequality,
\begin{equation}
\frac{1}{T}\norm{R_T}_\cH
\leq \frac{2\norm{u}_\infty}{T}
\leq \frac{2\tau_\theta}{T}.
\label{eq:rt-bound}
\end{equation}
Combining~\eqref{eq:mt-bound} and~\eqref{eq:rt-bound} therefore yields, with probability at least \(1-\delta\),
\begin{align}
\MMD_\phi(\wh\mu_T,\pi)
&\leq 2\norm{\phi}_\pi \sqrt{\frac{\tau_\rel \, L_\delta}{T}}
+ \frac{4\sqrt{5}\,\tau_\theta\sqrt{\tau_\rel} \, L_\delta}{T}
+ \frac{4\tau_\theta L_\delta}{3T}
+ \frac{2\tau_\theta}{T} \\
&\leq 2 \norm{\phi}_\pi \sqrt{\frac{\tau_\rel  \log(4/\delta)}{T}}
+ \frac{12 \tau_\theta \sqrt{\tau_\rel} \log(4/\delta)}{T},
\end{align}
where we used \( L_\delta = \log(4/\delta) > \log(4) > 1 \) and \( \sqrt{\tau_\rel} \ge 1 \).
\qed
\subsection{Extension to non-stationary initial distributions}
\label{app:occupation-nonstationary}
\begin{proof}[Proof of Theorem~\ref{thm:mmd-occupation-measure-whp-nonstationary}]
Up to the concentration estimate for the predictable quadratic variation, the proof proceeds exactly as in Theorem~\ref{thm:mmd-occupation-measure-whp}. Indeed, the solution to the Poisson equation, the martingale decomposition, the bounds on the martingale increments and the boundary term do not require stationarity of the initial distribution. In particular, we retain
\begin{equation}
\norm{V_u}_\infty
\leq 4\tau_\theta^2,
\qquad
\pi(V_u)
\leq \frac{1+\rho}{1-\rho}\norm{\phi}_\pi^2,
\label{eq:vu-nonstationary-basic}
\end{equation}
and, since \(V_u\geq0\),
\begin{equation}
\Var_\pi(V_u)
\leq \pi(V_u^2)
\leq \norm{V_u}_\infty\pi(V_u)
\leq
4\tau_\theta^2
\frac{1+\rho}{1-\rho}
\norm{\phi}_\pi^2.
\label{eq:vu-nonstationary-variance}
\end{equation}
The only modification concerns the concentration of
\begin{equation}
V_{T,t_0}
:= \sum_{t=t_0}^{t_0+T-1}V_u(X_t)
\end{equation}
around \(T\pi(V_u)\). Applying Theorem~\ref{thm:jiang-bernstein-nonstationary} to the centered observable \(V_u-\pi(V_u)\), which is bounded and measurable, gives for every \(\varepsilon>0\),
\begin{equation}
\label{eq:vt-nonstationary-concentration}
\bbP_\nu \left( V_{T,t_0}-T\pi(V_u)>\varepsilon \right)
\leq D \exp\left( -\frac{\varepsilon^2} {2q\left( T \alpha_1(\rho) \Var_\pi(V_u) + \alpha_2(\rho)\norm{V_u}_\infty\varepsilon \right)} \right),
\end{equation}
where \( \alpha_1, \alpha_2 \) are defined in~\eqref{eq:jiang-alpha-constants}. Let \( L_{\nu,\delta} := \log(4D/\delta) \). It follows from~\eqref{eq:vt-nonstationary-concentration} that, with probability at least \(1-\delta/4\),
\begin{align}
V_{T,t_0}
&\leq T\pi(V_u)
+ \sqrt{ 2q \, T \alpha_1(\rho) \Var_\pi(V_u) L_{\nu,\delta} }
+ 2q \, \alpha_2(\rho) \norm{V_u}_\infty L_{\nu,\delta} \\
&\leq \alpha_1(\rho) T\norm{\phi}_\pi^2
+ 2\sqrt{2q}\, \alpha_1(\rho)\tau_\theta\norm{\phi}_\pi \sqrt{T L_{\nu,\delta}}
+ 40q \tau_\rel\tau_\theta^2 L_{\nu,\delta}.
\end{align}
Using \(\alpha_1(\rho)\leq2\tau_\rel\), the latter is bounded by
\begin{equation}
V_{T,t_0}
\leq 2 \tau_\rel \left(
\sqrt{T}\norm{\phi}_\pi + 2\sqrt{5q}\, \tau_\theta\sqrt{L_{\nu,\delta}} \right)^2.
\label{eq:vt-nonstationary-upper-bound}
\end{equation}
We now invoke Theorem~\ref{thm:pinelis-bernstein-hilbert} with \( \norm{d_t}_\cH\leq2\tau_\theta \) and event envelope~\eqref{eq:vt-nonstationary-upper-bound}. Letting \(L_\delta:=\log(4/\delta)\), we obtain with probability at least \(1-\delta\),
\begin{align}
\frac{1}{T}\norm{M_{T,t_0}}_\cH
\leq \frac{\sqrt{2V_{T,t_0}L_\delta}}{T}
+ \frac{4\tau_\theta L_\delta}{3T}
\leq 2\norm{\phi}_\pi \sqrt{\frac{\tau_\rel L_\delta}{T}}
+ 4\sqrt{5q} \, \tau_\theta\sqrt{\tau_\rel}
\frac{\sqrt{L_\delta L_{\nu,\delta}}}{T}
+ \frac{4\tau_\theta L_\delta}{3T}.
\label{eq:mt-nonstationary-bound}
\end{align}
The boundary term is unchanged:
\begin{equation}
\frac{1}{T} \norm{ u(X_{t_0})-u(X_{t_0+T}) }_\cH
\leq \frac{2\tau_\theta}{T}.
\label{eq:rt-nonstationary-bound}
\end{equation}
Combining~\eqref{eq:mt-nonstationary-bound} and~\eqref{eq:rt-nonstationary-bound} gives
\begin{align}
\MMD_\phi(\wh\mu_{T,t_0},\pi)
\leq 2\norm{\phi}_\pi \sqrt{\frac{\tau_\rel \, L_\delta}{T}}
+ \frac{ 4\tau_\theta\ \sqrt{5 \tau_\rel q L_\delta L_{\nu,\delta}} }{T}
+ \frac{4\tau_\theta L_\delta}{3T}
+ \frac{2\tau_\theta}{T}.
\end{align}
Finally, since \(\nu_{t_0} := \nu_0 P^{t_0} \) is a probability measure, monotonicity of \(L^p(\pi)\)-norms on the probability space \((\cX,\pi)\) gives
\begin{equation}
D
= \Bnorm{\frac{d\nu_{t_0}}{d\pi}}_{L^p(\pi)}
\geq \Bnorm{\frac{d\nu_{t_0}}{d\pi}}_{L^1(\pi)}
= \nu_{t_0}(\cX)
=1.
\end{equation}
Since also \(q\geq1\), we have \(qL_\delta L_{\nu,\delta}\geq L_\delta\geq\log 4\). The conclusion now follows from \(4\sqrt{5}+4/3+2/\log 4<12\).
\end{proof}
\subsection{Occupation measure convergence without a spectral gap}
\label{app:occupation-nogap}
\begin{proof}[Proof of Corollary~\ref{cor:no-gap}]
Let \(g:=\phi-\pi\phi\). Assumption~\ref{ass:uniform-mmd-ergodicity} gives \(\norm{P^ng}_\infty\leq C\theta^n\) for every \(n\geq0\). Hence
\begin{equation}
u:=\sum_{n=0}^\infty P^ng
\end{equation}
converges uniformly on \(\cX\), with \(\norm{u}_\infty\leq\tau_\theta\). Term-by-term integration gives \(u-Pu=g\) at every \(x\in\cX\). With \(\cF_t:=\sigma(X_0,\ldots,X_t)\), set
\begin{equation}
d_t:=u(X_t)-Pu(X_{t-1}),
\qquad M_T:=\sum_{t=1}^T d_t.
\end{equation}
By the Markov property, \((d_t)_{t\geq1}\) is a martingale difference sequence, and
\begin{equation}
\sum_{s=0}^{T-1}g(X_s)=M_T+u(X_0)-u(X_T),
\qquad \norm{d_t}_\cH\leq2\tau_\theta,
\end{equation}
and
\begin{equation}
\bbE_{t-1}\norm{d_t}_\cH^2
= P\norm{u}_\cH^2(X_{t-1})-\norm{Pu(X_{t-1})}_\cH^2
\leq\tau_\theta^2.
\end{equation}
Thus \(\la M\ra_T\leq T\tau_\theta^2\) almost surely. Applying Theorem~\ref{thm:pinelis-bernstein} with \(R=2\tau_\theta\) and \(V=T\tau_\theta^2\), and bounding the boundary term by \(2\tau_\theta\), gives, with probability at least \(1-\delta\),
\begin{equation}
\MMD_\phi(\wh\mu_T,\pi)
\leq \tau_\theta\sqrt{\frac{2\log(2/\delta)}{T}}
+\frac{\tau_\theta}{T}\left(\frac43\log(2/\delta)+2\right).
\end{equation}
Since \(\log(2/\delta)\geq\log2\) and \(4/3+2/\log2<5\), this proves~\eqref{eq:no-gap}.
\end{proof}

\subsection{Learning guarantees}
\label{app:learning-proofs}
\label{app:error-bounds}
\label{app:sample-rate}
\begin{proof}[Proof of Proposition~\ref{prop:krr-error-bound}]
Fix \( \lambda > 0 \). By the standard argument, for any \( w \in \cH \) and the corresponding function \( f_w(x) := \la w, \phi(x) \ra_\cH \), it holds that
\begin{equation}
\cE(f_w, f_\star)
:= \bnorm{f_w - f_\star}_\pi^2
= \bnorm{\Sigma^{1/2}(w-w_\star)}_\cH^2.
\end{equation}
Since the model is well specified, \(h=\Sigma w_\star\). Moreover, using the empirical normal equation \( \wh\Sigma_\lambda \wh w_\lambda=\wh h \), we have
\begin{align}
\wh \Sigma_\lambda (\wh w_\lambda-w_\star)
= \wh h-\wh\Sigma_\lambda w_\star
= \wh h-\wh\Sigma w_\star-\lambda w_\star,
\end{align}
so that
\begin{equation}
\Sigma^{1/2} (\wh w_\lambda-w_\star)
= \Sigma^{1/2} \wh \Sigma_\lambda^{-1} \Sigma_\lambda^{1/2} \Sigma_\lambda^{-1/2} \big( \wh h-\wh\Sigma w_\star-\lambda w_\star \big).
\end{equation}
Hence by the triangle inequality,
\begin{equation}
\bnorm{\Sigma^{1/2}(\wh w_\lambda-w_\star)}_\cH
\leq
\underbrace{
\bnorm{ \Sigma^{1/2}\wh\Sigma_\lambda^{-1} \Sigma_\lambda^{1/2} }
}_{\term{stability}}
\Big(
\underbrace{
\bnorm{\Sigma_\lambda^{-1/2} (\wh h-\wh \Sigma w_\star)}_\cH
}_{\term{stochastic}}
+
\underbrace{
\lambda \bnorm{\Sigma_\lambda^{-1/2}w_\star}_\cH
}_{\term{reg}}
\Big).
\label{eq:error-decom-well}
\end{equation}
We proceed by bounding each term separately.

\paragraph{Deterministic bounds.}
From spectral properties of the involved operators, the following bounds hold:
\begin{equation}
\bnorm{\Sigma_\lambda^{-1/2}} \leq \frac{1}{\sqrt{\lambda}}
\qandq
\bnorm{ \Sigma^{1/2} \Sigma_\lambda^{-1/2} } \leq 1.
\label{eq:det-op-bound}
\end{equation}

\paragraph{Regularization bias.}
By Assumption~\ref{ass:learning-problem}~(i),
\begin{align}
\term{reg}
:= \lambda\bnorm{\Sigma_\lambda^{-1/2}w_\star}_\cH
= \lambda \bnorm{ \Sigma_\lambda^{-1/2} \Sigma^{r-1/2}g_\star }_\cH
\leq \lambda^r G_\star,
\qquad \frac12 \leq r \leq 1,
\label{eq:reg-term}
\end{align}
by the spectral theorem.

\paragraph{Stability term.} Introduce the preconditioned covariance fluctuation
\begin{equation}
\Delta_\lambda
:= \Sigma_\lambda^{-1/2} \big(\wh \Sigma - \Sigma\big) \Sigma_\lambda^{-1/2}.
\end{equation}
Then
\begin{equation}
\wh \Sigma_\lambda
= \wh \Sigma + \lambda I
= \Sigma_\lambda + \big( \wh \Sigma - \Sigma \big)
= \Sigma_\lambda^{1/2} (I+\Delta_\lambda) \Sigma_\lambda^{1/2}.
\end{equation}
Therefore, on the event \(\norm{\Delta_\lambda}<1\), the operator
\(I+\Delta_\lambda\) is invertible and
\begin{equation}
\wh\Sigma_\lambda^{-1}
= \Sigma_\lambda^{-1/2} (I+\Delta_\lambda)^{-1} \Sigma_\lambda^{-1/2}.
\end{equation}
Substituting this identity into \(\term{stability}\) and applying the bound from Equation~\eqref{eq:det-op-bound}, we obtain
\begin{equation}
\label{eq:stab-bound}
\term{stability}
:= \bnorm{ \Sigma^{1/2}\wh\Sigma_\lambda^{-1}\Sigma_\lambda^{1/2} }
\leq \bnorm{\Sigma^{1/2}\Sigma_\lambda^{-1/2}} \cdot \bnorm{(I+\Delta_\lambda)^{-1}}
\leq \frac{1}{1-\norm{\Delta_\lambda}}.
\end{equation}
In particular, on the event \(\norm{\Delta_\lambda} \leq 1/2\), we get \( \term{stability} \leq 2 \).

\medskip
To bound the \( \norm{\Delta_\lambda} \) term, define the following preconditioned feature map \( \Phi_\lambda: \cX \to \cL_2(\cH) \),
\begin{equation}
\Phi_\lambda(x)
:= \Sigma_\lambda^{-1/2} \left(\phi(x) \otimes \phi(x)\right) \Sigma_\lambda^{-1/2}
= \Sigma_\lambda^{-1/2} \, \Phi(x) \, \Sigma_\lambda^{-1/2},
\quad x \in \cX.
\end{equation}
An important observation which will be used later: by self-adjointness of \( \Sigma_\lambda^{-1/2} \) we can write \( \Phi_\lambda \) as
\begin{equation}
\Phi_\lambda(x) = \phi_\lambda(x) \otimes \phi_\lambda(x),
\quad \text{where} \; \phi_\lambda(x) := \Sigma_\lambda^{-1/2} \phi(x).
\label{eq:phi-lambda-def}
\end{equation}
In particular, \( \norm{ \Phi_\lambda(x) }_{\cL_2(\cH)} = \norm{\phi_\lambda(x)}_\cH^2 \).

\medskip \noindent
With \( \wh \Sigma = \wh \mu_T(\Phi) \) and \( \Sigma = \pi(\Phi) \), by linearity we get
\begin{equation}
\norm{ \Delta_\lambda }
\leq \norm{ \Delta_\lambda }_{\cL_2(\cH)}
= \MMD_{\Phi_\lambda}(\wh \mu_T, \pi).
\end{equation}
This allows us to apply Theorem~\ref{thm:mmd-occupation-measure-whp} to obtain a high-probability upper bound on this term.

\begin{remark}
The identity above is the key observation behind the learning result. It shows that the stability of the empirical inverse problem is governed directly by the convergence of the occupation measure of the process \( \wh \mu_T \) to the invariant measure \( \pi \) in a certain kernel-induced geometry \( \MMD_{\Phi_\lambda} \).
\end{remark}

\medskip
Write \(\Phi_\lambda=A_\lambda\Phi\), where
\begin{equation}
A_\lambda:\cL_2(\cH)\to\cL_2(\cH),
\qquad A_\lambda B
:=\Sigma_\lambda^{-1/2}B\Sigma_\lambda^{-1/2}.
\end{equation}
By~\eqref{eq:det-op-bound}, \( \opnorm{A_\lambda}_{\cL_2(\cH)} \leq\lambda^{-1} \). Proposition~\ref{prop:mmd-affine-transformation} therefore gives
\begin{equation}
\MMD_{\Phi_\lambda}(\mu,\nu)
\leq\lambda^{-1}\MMD_\Phi(\mu,\nu),
\qquad \mu,\nu\in\cP(\cX).
\end{equation}
Under Assumption~\ref{ass:feature-map}, \(\norm{\Phi}_\infty\leq\kappa^2\), and the Hilbert-scale condition holds for \(\Phi\). Applying Proposition~\ref{prop:mmd-alpha-bound} yields
\begin{equation}
\sup_{x\in\cX}
\MMD_{\Phi_\lambda}\bigl(P^n(x,\cdot),\pi\bigr)
\leq \left(\frac{\kappa^{2(1-\alpha)}\kappa_\alpha}{\lambda}\right) \rho_\alpha^n,
\qquad n \ge 0.
\label{eq:phi-lambda-contraction}
\end{equation}

\medskip \noindent
By spectral properties~\eqref{eq:det-op-bound},
\begin{equation}
\norm{\Phi_\lambda}_\infty
\leq \frac{1}{\lambda} \norm{\Phi}_\infty
\leq \frac{\kappa^2}{\lambda}.
\label{eq:phi-lambda-inf-norm}
\end{equation}
Moreover, by non-negativity of the norm,
\begin{align}
\norm{\Phi_\lambda}_{\pi}^2
:= \int \norm{ \Phi_\lambda(x) }_{\cL_2(\cH)}^2 \, \pi(dx)
\leq \norm{\Phi_\lambda}_{\infty} \cdot \int \norm{ \Phi_\lambda(x) }_{\cL_2(\cH)} \, \pi(dx).
\label{eq:xi1-l2-bound}
\end{align}
Furthermore,
\begin{equation}
\int \norm{ \Phi_\lambda(x) }_{\cL_2(\cH)} \, \pi(dx)
= \int \norm{\phi_\lambda}_\cH^2 \, \pi(dx)
= \Tr \big( \Sigma_\lambda^{-1} \Sigma \big)
=: \cN(\lambda).
\label{eq:phi-lambda-pre-l2-norm}
\end{equation}
Putting~\eqref{eq:phi-lambda-inf-norm} and~\eqref{eq:phi-lambda-pre-l2-norm} into~\eqref{eq:xi1-l2-bound},
\begin{equation}
\norm{\Phi_\lambda}_{\pi}^2 \leq \frac{\kappa^2 \cN(\lambda)}{\lambda}.
\label{eq:phi-lambda-l2-norm}
\end{equation}

\medskip \noindent
Applying Theorem~\ref{thm:mmd-occupation-measure-whp} with~\eqref{eq:phi-lambda-contraction} and~\eqref{eq:phi-lambda-l2-norm} yields, for any \(\delta\in(0,1)\), with probability at least \(1-\delta\),
\begin{equation}
\norm{\Delta_\lambda}_{\cL_2(\cH)}
\leq 2\sqrt{ \frac{ \kappa^2\cN(\lambda)\tau_\rel\log(4/\delta) }{ \lambda T }}
+ \frac{ 12B_\alpha\sqrt{\tau_\rel}\log(4/\delta) }{ (1-\rho_\alpha)\lambda T },
\label{eq:delta-lambda-bound}
\end{equation}
where \( B_\alpha:=\kappa^{2(1-\alpha)}\kappa_\alpha \).

\paragraph{Stochastic error.} To control the term \( \term{stochastic} \), consider the following observable \( \xi: \cX \times \cX \to \cH \),
\begin{equation}
\xi(x, x') := \left(x' - f_\star(x)\right) \phi_\lambda(x) \in \cH,
\qquad x, x' \in \cX.
\end{equation}
Observe, that
\begin{equation}
\frac{1}{T} \sum_{t=0}^{T-1} \xi(x_t, x_{t+1}) = \Sigma_\lambda^{-1/2} \big(\wh h - \wh \Sigma w_\star\big)
\qandq
\int \xi(x, x') \, \pi(dx) P(x, dx') = 0.
\end{equation}
Now, consider the stochastic process \( (d_t)_{t \ge 1} \) given by
\begin{equation}
d_{t}
:= \xi(X_{t-1}, X_{t})
= \left(X_{t} - f_\star(X_{t-1})\right) \phi_\lambda(X_{t-1}),
\qquad t\geq1.
\end{equation}
Let
\begin{equation}
\cF_t
:= \sigma\left((X_0, X_1), \ldots (X_{t-1}, X_{t})\right)
= \sigma(X_0,\ldots,X_t),
\qquad t\geq0.
\end{equation}
Then \(d_t\) is \(\cF_t\)-measurable and, by Markovianity,
\begin{align}
\bbE \left[ d_{t} \mid \cF_{t-1} \right]
&= \bbE \left[ \left(X_{t} - f_\star(X_{t-1})\right) \phi_\lambda(X_{t-1}) \mid X_{t-1} \right] \\
&= \underbrace{\left( \bbE[ X_{t} \mid X_{t-1} ] - f_\star(X_{t-1}) \right)}_{= 0} \phi_\lambda(X_{t-1}).
\label{eq:dt-martingale-exp}
\end{align}
Therefore, \( (d_t)_{t \ge 1} \) is a martingale difference sequence with respect to \( (\cF_t)_t \) and \( M_T := \sum_{t=1}^T d_t \) is a martingale. Moreover, by measurability,
\begin{equation}
\bbE_{t-1}\norm{d_t}_\cH^2
= \norm{\phi_\lambda(X_{t-1})}_\cH^2 \cdot \bbE\left[ \abs{X_t-f_\star(X_{t-1})}^2 \mid X_{t-1} \right].
\end{equation}
Hence, under Assumption~\ref{ass:markov-chain}\textup{(iii)},
\begin{equation}
\la M \ra_T := \sum_{t=1}^{T} \bbE_{t-1} \norm{d_{t}}_\cH^2
\leq \sigma^2 \cdot \sum_{t=1}^{T} \norm{ \phi_\lambda(X_{t-1}) }_\cH^2
= \sigma^2 \cdot \sum_{s=0}^{T-1} \norm{ \phi_\lambda(X_s) }_\cH^2.
\label{eq:m-t-pre-bound}
\end{equation}
Define a measurable bounded real-valued observable \( g_\lambda(x) :=  \norm{ \phi_\lambda(x) }_\cH^2 = \norm{\Phi_\lambda(x)}_{\cL_2(\cH)} \). Its stationary variance, by~\eqref{eq:phi-lambda-l2-norm}, satisfies
\begin{align}
v_\lambda
:= \pi\left( \big(g_\lambda-\pi(g_\lambda)\big)^2 \right)
\leq \pi \big(g_\lambda^2\big)
= \norm{\Phi_\lambda}_\pi^2
\leq \frac{\kappa^2 \cN(\lambda)}{\lambda}.
\label{eq:g-lambda-variance}
\end{align}
Applying Theorem~\ref{thm:jiang-bernstein-stationary} to the centered observable \(g_\lambda-\pi(g_\lambda)\), which is bounded and measurable, gives, for every \(\varepsilon>0\),
\begin{equation}
\bbP\left\{ \sum_{t=0}^{T-1}g_\lambda(X_t) - T\cN(\lambda) > \varepsilon \right\}
\leq \exp\left\{ -\frac{\varepsilon^2} {2\left( \alpha_1(\rho) T v_\lambda + \frac{5\kappa^2}{(1-\rho)\lambda}\varepsilon \right)} \right\}.
\label{eq:g-lambda-bernstein}
\end{equation}
Consequently, letting \(L_\delta:=\log(4/\delta)\), with probability at
least \(1-\delta/2\),
\begin{align}
\sum_{t=0}^{T-1}g_\lambda(X_t)
\leq T\cN(\lambda)
+ \sqrt{2\alpha_1(\rho) T v_\lambda L_\delta}
+ \frac{10\kappa^2L_\delta}{(1-\rho)\lambda}
\leq T\cN(\lambda)
+ 2\kappa \sqrt{ \frac{T\cN(\lambda)L_\delta} {(1-\rho)\lambda} }
+ \frac{10\kappa^2L_\delta} {(1-\rho)\lambda},
\label{eq:g-lambda-sum-bound}
\end{align}
where in the second inequality we used~\eqref{eq:g-lambda-variance} and \(1+\rho\leq2\). Putting~\eqref{eq:g-lambda-sum-bound} back into Equation~\eqref{eq:m-t-pre-bound}, we obtain
\begin{align}
\la M\ra_T
\leq T\sigma^2\cN(\lambda)
+ 2\sigma^2\kappa \sqrt{ \frac{T\cN(\lambda)L_\delta}{(1-\rho)\lambda} }
+ \frac{10\sigma^2\kappa^2L_\delta} {(1-\rho)\lambda}
=: V_T
\label{eq:t3-quad-var-bound}
\end{align}
with probability at least \(1-\delta/2\). We now apply Theorem~\ref{thm:pinelis-bernstein-hilbert} with
\begin{equation}
R_\lambda:=\frac{\kappa M}{\sqrt{\lambda}}.
\end{equation}
Combining the two estimates and using \(\sqrt{a+b+c}\leq\sqrt a+\sqrt b+\sqrt c\), we obtain, with probability at least \(1-\delta\),
\begin{align}
\term{stochastic}
&= \frac{1}{T}\norm{M_T}_\cH \\
&\leq \frac{\sqrt{2V_T L_\delta}}{T}
+ \frac{2 R_\lambda L_\delta}{3T} \\
&\leq \sqrt{ \frac{2\sigma^2\cN(\lambda)L_\delta}{T} }
+ 2\sigma \left( \frac{\kappa^2\cN(\lambda)}{(1-\rho)\lambda} \right)^{1/4} \frac{L_\delta^{3/4}}{T^{3/4}}
+ \frac{2\sqrt{5}\,\sigma\kappa L_\delta} {\sqrt{(1-\rho)\lambda}\,T}
+ \frac{2\kappa M L_\delta} {3\sqrt{\lambda}\,T}.
\label{eq:t3-bound}
\end{align}

\paragraph{Putting the bounds together.}
We now collect the estimates obtained above. Apply \eqref{eq:delta-lambda-bound} and \eqref{eq:t3-bound} with $\delta/2$ in place of $\delta$: each then holds with probability at least $1-\delta/2$, and in both the confidence factor becomes \( L_\delta := \log(8/\delta) \). By~\eqref{eq:delta-lambda-bound}, a sufficient condition for \(\norm{\Delta_\lambda}_{\cL_2(\cH)}\leq1/2\) is
\begin{equation}
T\geq
\max\left\{
\frac{64 \kappa^2 \tau_{\rel} \cN(\lambda) L_\delta}{\lambda},
\,
\frac{48B_\alpha\sqrt{\tau_{\rel}} L_\delta}
{(1-\rho_\alpha)\lambda}
\right\}.
\end{equation}
On this event \( \term{stability} \le 2 \) by \eqref{eq:stab-bound}. Substituting \eqref{eq:reg-term} and \eqref{eq:t3-bound} into \eqref{eq:error-decom-well} and taking a union bound over the two events, we obtain, with probability at least \( 1-\delta \),
\begin{align}
\bnorm{\wh f_\lambda-f_\star}_\pi
\leq\;&
2\lambda^rG_\star
+ 2\sigma\sqrt{\frac{2 \cN(\lambda) L_\delta}{T}} \notag \\
&+ 4\sigma \left( \frac{\kappa^2 \tau_\rel \, \cN(\lambda)}{\lambda} \right)^{1/4} \frac{L_\delta^{3/4}}{T^{3/4}}
+ \frac{4\sqrt{5 \tau_\rel}\,\sigma\kappa L_\delta}{\sqrt{\lambda}\,T}
+ \frac{4\kappa M L_\delta}{3\sqrt{\lambda}\,T}
\label{eq:error-bound-final-app}
\end{align}
as claimed.
\end{proof}

\begin{remark}[Sharpness of the stochastic term correction]
The correction of order
\begin{equation}
T^{-3/4} \left( \frac{\cN(\lambda)}{\lambda} \right)^{1/4}
\end{equation}
in the term~\eqref{eq:t3-bound} is natural for the specific proof strategy but is unlikely to be sharp. In particular, if \(\cN(\lambda)\lesssim \lambda^{-p}\), the correction is of order \(T^{-3/4}\lambda^{-(p+1)/4}\). Asymptotically, this term is still dominated by
\begin{equation}
\sqrt{\frac{\cN(\lambda)}{T}}
\end{equation}
under \( T \lambda^{1-p} \to \infty \) regularization schedule, which coincides with the expected scale from classical i.i.d. analysis.
\end{remark}

\begin{proof}[Proof of Theorem~\ref{thm:final-sample-rate}]
By Assumption~\ref{ass:learning-problem}\textup{(ii)}, the effective dimension satisfies
\begin{equation}
\cN(\lambda)\lesssim \lambda^{-p},
\qquad p=\frac{1}{b}\in(0,1).
\label{eq:effective-dimension-polynomial}
\end{equation}
Therefore, up to multiplicative constants and logarithmic factors, the bound in~\eqref{eq:error-bound-final} becomes
\begin{equation}
\bnorm{\wh f_\lambda-f_\star}_\pi
\lesssim \lambda^r
+ \frac{1}{\sqrt{\lambda^pT}}
+ \frac{1}{\lambda^{(p+1)/4}T^{3/4}}
+ \frac{1}{\sqrt{\lambda}T}.
\label{eq:error-bound-rate}
\end{equation}
The first two terms are the usual regularization bias and stochastic error. Balancing them gives
\begin{equation}
\lambda^r
\asymp \frac{1}{\sqrt{\lambda^pT}},
\qquad \text{so that} \qquad
\lambda_T^r
\asymp T^{-\frac{r}{2r+p}}.
\label{eq:lambda-rate-proof}
\end{equation}
We now need to check that the two Bernstein correction terms in~\eqref{eq:error-bound-rate} are of no larger order. First,
\begin{align}
\frac{1}{\lambda_T^{(p+1)/4}T^{3/4}}
= T^{-\frac34+\frac{p+1}{4(2r+p)}}
= T^{-\frac{6r+2p-1}{4(2r+p)}}.
\end{align}
Since \(r\geq1/2\) and \(p>0\), we have \( (6r+2p-1)/4 \geq r, \) or equivalently \(2r+2p\geq1\). Therefore,
\begin{equation}
\frac{1}{\lambda_T^{(p+1)/4}T^{3/4}}
\lesssim T^{-\frac{r}{2r+p}},
\qquad \text{and} \qquad
\frac{1}{\sqrt{\lambda_T}T} = T^{-1+\frac{1}{2(2r+p)}}.
\end{equation}
Since \(r+p\geq1/2\),
\begin{equation}
-1+\frac{1}{2(2r+p)}
\leq -\frac{r}{2r+p},
\implies \frac{1}{\sqrt{\lambda_T}T}
\lesssim T^{-\frac{r}{2r+p}}.
\end{equation}
Thus all the terms in~\eqref{eq:error-bound-rate} are of order at most \(T^{-r/(2r+p)}\).

\medskip
It remains to verify the stability condition. By
Assumption~\ref{ass:learning-problem}\textup{(ii)} and Equation~\eqref{eq:delta-lambda-bound}, up to multiplicative constants and logarithmic factors,
\begin{equation}
\norm{\Delta_\lambda}
\lesssim \frac{1}{\sqrt{\lambda^{p+1}T}}
+ \frac{1}{\lambda T}.
\label{eq:delta-rate-proof}
\end{equation}
For \( \lambda_T\asymp T^{-1/(2r+p)} \) we have
\begin{equation}
\frac{1}{\sqrt{\lambda_T^{p+1}T}}
= T^{-\frac{2r-1}{2(2r+p)}},
\qquad \text{and} \qquad
\frac{1}{\lambda_T T}
= T^{-1+\frac{1}{2r+p}}.
\end{equation}
Both terms converge to zero for \(r>1/2\). At the endpoint \(r=1/2\), the first term is of constant order, and the multiplicative constant in the choice of \(\lambda_T\) can be taken sufficiently large so that \(\norm{\Delta_{\lambda_T}}\leq1/2\).

\medskip
Consequently, on the stability event,
\begin{equation}
\bnorm{\wh f_{\lambda_T}-f_\star}_\pi
\lesssim T^{-\frac{r}{2r+p}},
\end{equation}
and squaring both sides yields
\begin{equation}
\cE\big(\wh f_{\lambda_T},f_\star\big)
\lesssim T^{-\frac{2r}{2r+p}},
\end{equation}
which proves the claim.
\end{proof}

\section{Fourier analysis of randomly perturbed torus maps}
\label{app:fourier-examples}
We give a common Fourier representation for Examples~\ref{ex:noisy-irrational-rotation} and~\ref{ex:noisy-expanding-map}. The same representation also applies to the covariance kernel used in Appendix~\ref{app:noisy-doubling-learning}.

\subsection{Hilbert-scale space}
\label{app:fourier-common}

Let \(\pi\) be normalized Lebesgue measure on \(\bbT^d\), and write
\begin{equation}
e_m(x):=e^{2\pi i m\cdot x},
\qquad
\wh f(m):=\int_{\bbT^d}f(x)\overline{e_m(x)}\,\pi(dx).
\end{equation}
Complex Fourier notation is used for convenience; all constructions can equivalently be expressed using real sine and cosine coordinates.

\medskip
Let \((a_m)_{m\in\bbZ^d}\) satisfy
\begin{equation}
a_0\geq0,
\qquad
a_{-m}=a_m>0 \quad (m\neq0),
\qquad
\sum_{m\in\bbZ^d}a_m<\infty.
\end{equation}
Define the continuous real-valued positive-definite kernel
\begin{equation}
k_a(x,y)
:=\sum_{m\in\bbZ^d}a_m e^{2\pi i m\cdot(x-y)}
\label{eq:fourier-general-kernel}
\end{equation}
and its canonical feature map
\(\phi_a(x):=k_a(x,\cdot)\).
The kernel is uniformly bounded, with
\begin{equation}
\norm{\phi_a(x)}_{\cH_{k_a}}^2
=k_a(x,x)
=\sum_{m\in\bbZ^d}a_m.
\end{equation}
Since \(\int k_a(x,y)\,\pi(dy)=a_0\), centering removes precisely the constant Fourier mode:
\begin{equation}
\overline k_a(x,y)
=k_a(x,y)-a_0
=\sum_{m\neq0}a_m e^{2\pi i m\cdot(x-y)}.
\end{equation}
Thus including or omitting the constant mode does not change the centered Hilbert scale or the MMD between probability measures.

\medskip
The centered kernel integral operator satisfies
\begin{equation}
\overline L_a e_m=a_m e_m,
\qquad m\neq0.
\end{equation}
For \(\alpha>0\), its positive Hilbert scale is therefore
\begin{equation}
\cE_\alpha
= \left\{ f\in L_0^2(\pi): \sum_{m\neq0}a_m^{-\alpha}\abs{\wh f(m)}^2<\infty \right\},
\qquad \text{with} \quad
\norm{f}_\alpha^2
=\sum_{m\neq0}a_m^{-\alpha}\abs{\wh f(m)}^2.
\label{eq:fourier-power-norm}
\end{equation}
Suppose that
\begin{equation}
\kappa_\alpha^2
:=\sum_{m\neq0}a_m^\alpha<\infty.
\label{eq:fourier-embedding-constant}
\end{equation}
Then the series
\begin{equation}
\overline k_{a,\alpha}(x,y)
:=\sum_{m\neq0}a_m^\alpha e^{2\pi i m\cdot(x-y)}
\end{equation}
converges absolutely and uniformly, and its RKHS is \(\cE_\alpha\).
The embedding into bounded continuous functions follows directly from Cauchy--Schwarz:
\begin{equation}
\sum_{m\neq0}\abs{\wh f(m)}
\leq
\left(\sum_{m\neq0}a_m^\alpha\right)^{1/2}
\left(\sum_{m\neq0}a_m^{-\alpha}\abs{\wh f(m)}^2\right)^{1/2}
=\kappa_\alpha\norm{f}_\alpha.
\end{equation}
Thus every \(f\in\cE_\alpha\) has a uniformly convergent Fourier representation and
\begin{equation}
\norm{f}_\infty\leq\kappa_\alpha\norm{f}_\alpha.
\label{eq:fourier-uniform-embedding}
\end{equation}
For a finite signed measure \(\zeta\), write \(\wh\zeta(m):=\int\overline{e_m(x)}\,\zeta(dx)\). Its restriction to \(\cE_\alpha\) has dual norm
\begin{equation}
\norm{\zeta}_{-\alpha}^2
=\sum_{m\neq0}a_m^\alpha\abs{\wh\zeta(m)}^2.
\end{equation}
In particular, since \(\wh\pi(m)=0\) for \(m\neq0\),
\begin{equation}
\norm{\delta_x-\pi}_{-\alpha}=\kappa_\alpha,
\qquad
\norm{\mu-\pi}_{-\alpha}\leq\kappa_\alpha,
\quad \mu\in\cP(\bbT^d).
\label{eq:fourier-measure-bound}
\end{equation}
The second bound uses \(\abs{\wh\mu(m)}\leq1\). This explains why the convergence bound toward \(\pi\) has prefactor \(\kappa_\alpha\), without an additional factor of two.

\subsection{Contraction formula}
\label{app:fourier-contraction}

Consider
\begin{equation}
X_{t+1}=\ell X_t+\omega+\xi_{t+1}\pmod1,
\label{eq:fourier-general-dynamics}
\end{equation}
where \(\ell\geq1\) is an integer, \(\omega\in\bbT^d\), and the noises are i.i.d. with probability density \(\upsilon\). Multiplication by \(\ell\), translation, and independent addition each preserve normalized Lebesgue measure, so \(\pi P=\pi\). With the Fourier convention
\begin{equation}
\wh\upsilon(m)
:=\int_{\bbT^d}e^{-2\pi i m\cdot u} \, \upsilon(u) du,
\end{equation}
the Markov operator acts as
\begin{equation}
Pe_m
=e^{2\pi i m\cdot\omega}\,\wh\upsilon(-m)e_{\ell m}.
\label{eq:fourier-markov-action}
\end{equation}

\begin{lemma}[Fourier contraction formula]
\label{lem:fourier-contraction}
For the dynamics \eqref{eq:fourier-general-dynamics},
\begin{equation}
\rho
:=\opnorm{P}_{L_0^2(\pi)}
=\sup_{m\neq0}\abs{\wh\upsilon(m)}<1.
\label{eq:fourier-l2-contraction}
\end{equation}
Suppose that \eqref{eq:fourier-embedding-constant} holds and that
\begin{equation}
R_\alpha
:=\sup_{m\neq0}
\abs{\wh\upsilon(m)}
\left(\frac{a_m}{a_{\ell m}}\right)^{\alpha/2}
<\infty.
\label{eq:fourier-weighted-contraction}
\end{equation}
Then \(P\) acts boundedly on \(\cE_\alpha\), and \( \opnorm{P}_{\cE_\alpha} = R_\alpha \). Furthermore, if \(R_\alpha<1\), Assumption~\ref{ass:hilbert-scale-gap} holds with \(\rho_\alpha=R_\alpha\), and
\begin{equation}
\MMD_\alpha(\mu P^n,\pi)
\leq\kappa_\alpha R_\alpha^n,
\qquad \mu\in\cP(\bbT^d), \; n\geq0.
\label{eq:fourier-uniform-mmd}
\end{equation}
\end{lemma}

\begin{proof}
The map \(m\mapsto\ell m\) is injective. Thus \eqref{eq:fourier-markov-action} and orthogonality give
\begin{equation}
\norm{Pf}_\pi^2
=\sum_{m\neq0}\abs{\wh\upsilon(m)}^2
\abs{\wh f(m)}^2,
\qquad f\in L_0^2(\pi),
\end{equation}
which proves the norm identity in \eqref{eq:fourier-l2-contraction}.

\medskip \noindent
For each \(m\neq0\), one has \(\abs{\wh\upsilon(m)}<1\). Indeed, equality would force the unit-modulus random variable \(e^{-2\pi i m\cdot\xi_1}\) to be constant almost surely. The noise would then be supported on a level set of a nonconstant torus character, which has Lebesgue measure zero, contradicting absolute continuity. The Riemann--Lebesgue lemma gives
\(\wh\upsilon(m)\to0\) as \(\abs{m}\to\infty\).
Combining this decay with the strict inequality for each of the finitely many remaining modes yields \(\rho<1\).

\medskip \noindent
For a centered trigonometric polynomial, the same calculation in the weighted norm gives
\begin{align}
\norm{Pf}_\alpha^2
&=\sum_{m\neq0}
a_{\ell m}^{-\alpha}
\abs{\wh\upsilon(m)}^2
\abs{\wh f(m)}^2
\leq R_\alpha^2
\sum_{m\neq0}a_m^{-\alpha}\abs{\wh f(m)}^2
=R_\alpha^2\norm{f}_\alpha^2.
\end{align}
Testing on individual Fourier modes gives the reverse inequality for the operator norm. The same norm is obtained on the real space by using the corresponding sine and cosine modes. The estimate extends by density to \(\cE_\alpha\). By \eqref{eq:fourier-uniform-embedding}, convergence in this space implies uniform convergence, so the extension agrees with the pointwise Markov action. The MMD bound follows from Proposition~\ref{prop:mmd-alpha-bound} with \(\rho_\alpha=R_\alpha\).
\end{proof}

\subsection{Noisy rotations}
\label{app:noisy-rotation-details}

For Example~\ref{ex:noisy-irrational-rotation}, take \(\ell=1\), \(a_0=0\), and
\begin{equation}
a_m=(1+\abs{m}^2)^{-s},
\qquad m\neq0.
\end{equation}
The embedding constant is
\begin{equation}
\kappa_\alpha^2
=\sum_{m\neq0}(1+\abs{m}^2)^{-s\alpha}.
\end{equation}
Comparison with lattice shells shows that this series is finite if and only if \(2s\alpha>d\). 
Since \(a_{\ell m}=a_m\), Lemma~\ref{lem:fourier-contraction} gives
\begin{equation}
\rho
=\sup_{m\neq0}\abs{\wh\upsilon(m)}<1.
\end{equation}
Since \(P\) and \(\overline L_{\phi_s}\) are Fourier multipliers and \(\cG=L_0^2(\pi)\), Corollary~\ref{cor:commuting-hilbert-scale} gives
\(\rho_\alpha=\rho\). The pointwise-preservation requirement holds because \(P\) is Feller, \(\pi\) has full support, and the Sobolev realization of \(\cE_\alpha\) consists of continuous functions whenever \(s\alpha>d/2\). In particular, \(s>d/2\) permits \(\alpha=1\). The canonical feature map of the centered kernel then satisfies by Proposition~\ref{prop:mmd-alpha-bound} 
\begin{equation}
\MMD_{\phi_s}(\mu P^n,\pi)
\leq\kappa_1\rho^n,
\qquad \mu\in\cP(\bbT^d),\quad n\geq0.
\end{equation}

\subsection{Noisy expanding maps}
\label{app:noisy-expanding-details}

For Example~\ref{ex:noisy-expanding-map}, take \(d=1\), \(\ell\geq2\), \(a_0=0\), and
\begin{equation}
a_m=\abs{m}^{-2s},
\qquad m\neq0.
\end{equation}
The wrapped Gaussian noise \(\eta Z\pmod1\) has Fourier coefficients
\begin{equation}
\wh\upsilon(m)
=\bbE e^{-2\pi i m\eta Z}
=e^{-2\pi^2\eta^2m^2}.
\end{equation}
Hence
\begin{equation}
Pe_m
=e^{2\pi i m\omega}
e^{-2\pi^2\eta^2m^2}e_{\ell m},
\end{equation}
and Lemma~\ref{lem:fourier-contraction} yields \( \rho=e^{-2\pi^2\eta^2} \). For the power scale, \( a_m / a_{\ell m} =\ell^{2s} \), independently of \(m\neq0\). Therefore the exact norm is
\begin{equation}
\rho_\alpha
=\ell^{s\alpha}
\sup_{m\neq0}e^{-2\pi^2\eta^2m^2}
=\ell^{s\alpha}e^{-2\pi^2\eta^2}.
\end{equation}
The embedding constant is
\begin{equation}
\kappa_\alpha^2
=\sum_{m\neq0}\abs{m}^{-2s\alpha}
=2\sum_{m\geq1}m^{-2s\alpha},
\end{equation}
which is finite if and only if \(s\alpha>1/2\). Under~\eqref{eq:noisy-expanding-two-conditions}, Lemma~\ref{lem:fourier-contraction} verifies Assumption~\ref{ass:hilbert-scale-gap} with \(\rho_\alpha=\ell^{s\alpha}e^{-2\pi^2\eta^2}\). The claimed MMD bound then follows from Proposition~\ref{prop:mmd-alpha-bound}.

\paragraph{One-step contraction versus contraction of iterates.}
The threshold \eqref{eq:expanding-noise-threshold} concerns the one-step Hilbert-scale condition. It is not necessary for contraction of sufficiently high iterates. Indeed, iterating the Fourier action gives
\begin{equation}
P^n e_m
=
\exp\left(
2\pi i m\omega\frac{\ell^n-1}{\ell-1}
\right)
\exp\left(
-2\pi^2\eta^2m^2\frac{\ell^{2n}-1}{\ell^2-1}
\right)
e_{\ell^n m}.
\end{equation}
Consequently,
\begin{equation}
\opnorm{P^n}_{\cE_\alpha}
=
\ell^{ns\alpha}
\exp\left(
-2\pi^2\eta^2\frac{\ell^{2n}-1}{\ell^2-1}
\right).
\end{equation}
For every \(\eta>0\), this tends to zero as \(n\to\infty\), even when \(\rho_\alpha\geq1\). Whenever \(s\alpha>1/2\), the same argument as in \eqref{eq:fourier-uniform-mmd} gives
\begin{equation}
\MMD_\alpha(\mu P^n,\pi)
\leq \kappa_\alpha\ell^{ns\alpha} \exp\left( -2\pi^2\eta^2\frac{\ell^{2n}-1}{\ell^2-1} \right).
\end{equation}
This distinguishes the sufficient one-step assumption used in the main text from the convergence properties of the chain itself.

\begin{remark}[Drift--noise factorization]
\label{rem:drift-noise}
More generally, on an additive state space \( \cX \), consider
\begin{equation}
X_{t+1}=F(X_t)+\xi_{t+1},
\end{equation}
where \(F:\cX\to\cX\) is measurable and \((\xi_t)_{t\geq1}\) are i.i.d. noise variables. Define the deterministic Koopman operator \(U_F\) and the noise operator \(Q_\xi\) by
\begin{equation}
(U_Ff)(x):=f(F(x)),
\qquad (Q_\xi f)(x):=\bbE[f(x+\xi_1)].
\end{equation}
Whenever the expectations are well defined,
\begin{equation}
Pf(x)
=\bbE[f(F(x)+\xi_1)]
=(U_FQ_\xi f)(x),
\end{equation}
so \(P=U_FQ_\xi\). Suppose that both operators preserve the chosen centered space \(\cE_\alpha\) and act boundedly on it. Then
\begin{equation}
\opnorm{P}_{\cE_\alpha}
\leq \rho_{F,\alpha}\rho_{\xi,\alpha},
\qquad \rho_{F,\alpha}
:=\opnorm{U_F}_{\cE_\alpha},
\quad \rho_{\xi,\alpha}
:=\opnorm{Q_\xi}_{\cE_\alpha}.
\end{equation}
These are substantive requirements: invariance of \(\pi\) under \(P\) alone need not imply that the two factors separately preserve the centered space. Moreover, boundedness of \(U_F\) in a smooth function space requires appropriate regularity of \(F\). By duality, the same bounds hold for the push-forward action on \(\cE_{-\alpha}\). Together with the bounded pointwise realization, a product strictly below one implies Assumption~\ref{ass:hilbert-scale-gap}.

\medskip
For the noisy expanding map in Example~\ref{ex:noisy-expanding-map}, \( \rho_{F,\alpha}=\ell^{s\alpha} \),  \( \rho_{\xi,\alpha}=e^{-2\pi^2\eta^2} \), and the product bound is sharp. The deterministic factor measures the amplification of frequencies, while the noise factor measures their attenuation.  This also clarifies the relation with Example~\ref{ex:noisy-irrational-rotation}. There the deterministic map is the translation \(x\mapsto x+\omega\), which plays the role of \(\ell=1\). Thus the deterministic Koopman operator is an isometry on the translation-invariant power scale, and the amplification factor is absent. For a general noise density \(\upsilon\), the contraction factor is \( \sup_{m\neq0}\abs{\wh\upsilon(m)} \), which is precisely the Fourier multiplier appearing in the noisy rotation example.
\end{remark}

\subsection{Statistical learning assumptions for a noisy doubling map}
\label{app:noisy-doubling-learning}

We specialize Example~\ref{ex:noisy-expanding-map} to \(\ell=2\), \(\omega=0\) and Gaussian noise scale
\begin{equation}
\eta:=\sqrt{\frac{\log 4}{2\pi^2}}.
\end{equation}
Representing \(\bbT\) by \([0,1)\), consider
\begin{equation}
X_{t+1}=2X_t+\eta Z_{t+1}\pmod 1,
\qquad X_0\sim\pi, \; \pi(dx)=dx,
\label{eq:nd-dynamics}
\end{equation}
where the \(Z_t\) are independent standard Gaussian random variables, independent of \(X_0\). The modulo operation takes values in \([0,1)\), and the response in the learning problem is this real-valued representative of \(X_{t+1}\). We use the periodic Sobolev kernel
\begin{equation}
k(x,y)
:=\sum_{m\in\bbZ}a_m e^{2\pi i m(x-y)}
\qquad \text{where} \quad
a_m:=\frac{1}{1+m^2},
\label{eq:nd-kernel}
\end{equation}
with real RKHS \(\cH\) and canonical feature map \(\phi(x)=k(x,\cdot)\). We verify all assumptions of Theorem~\ref{thm:final-sample-rate}.

\paragraph{Noise distribution and stationarity.}
The wrapped Gaussian perturbation \(\eta Z\pmod 1\) has density
\begin{equation}
\upsilon_\eta(u)
=\frac{1}{\sqrt{2\pi}\eta}
\sum_{j\in\bbZ}
\exp\left(-\frac{(u+j)^2}{2\eta^2}\right),
\qquad u\in[0,1).
\label{eq:nd-noise}
\end{equation}
Its Fourier coefficients follow from the Gaussian characteristic function:
\begin{align}
\wh \upsilon_\eta(m)
&=\int_0^1 e^{-2\pi i m u}\upsilon_\eta(u)\,du
=\int_{\bbR}e^{-2\pi i m z}
\frac{1}{\sqrt{2\pi}\eta}
\exp\left(-\frac{z^2}{2\eta^2}\right)\,dz
=e^{-2\pi^2\eta^2m^2}
=4^{-m^2}.
\end{align}
Consequently,
\begin{equation}
\upsilon_\eta(u)
=1+2\sum_{m\geq1}4^{-m^2}\cos(2\pi m u),
\label{eq:nd-noise-fourier}
\end{equation}
with absolute and uniform convergence. Since
\begin{equation}
\sum_{m\geq1}4^{-m^2}
\leq\sum_{m\geq1}4^{-m}
=\frac13,
\end{equation}
we obtain the uniform bounds
\begin{equation}
\frac13\leq\upsilon_\eta(u)\leq\frac53.
\label{eq:nd-noise-bounds}
\end{equation}
Extending \(\upsilon_\eta\) periodically, the transition kernel is
\begin{equation}
P(x,dy)=\upsilon_\eta(y-2x)\,dy.
\end{equation}
For every \(y\in[0,1)\),
\begin{equation}
\int_0^1\upsilon_\eta(y-2x)\,dx
=\frac12\int_{y-2}^{y}\upsilon_\eta(u)\,du
=1.
\end{equation}
Thus \(\pi P=\pi\), and the initialization in \eqref{eq:nd-dynamics} makes the chain stationary. The lower bound in \eqref{eq:nd-noise-bounds} gives
\(P(x,\cdot)\geq(1/3)\pi\).
Writing
\begin{equation}
P(x,\cdot)=\frac13\pi+\frac23R(x,\cdot)
\end{equation}
defines a Markov kernel \(R\) with \(\pi R=\pi\). It follows that
\begin{equation}
\norm{\mu P^n-\pi}_{\TV}
\leq\left(\frac23\right)^n\norm{\mu-\pi}_{\TV},
\qquad n\geq0.
\end{equation}
In particular, \(\pi\) is the unique invariant probability measure. It also has finite second moment, since the scalar state lies in \([0,1)\).

\paragraph{The \(L^2(\pi)\) spectral gap.}
Since \(\wh \upsilon_\eta (m)=4^{-m^2}\), the common Fourier formula~\eqref{eq:fourier-markov-action} gives
\begin{equation}
Pe_m=4^{-m^2}e_{2m}.
\label{eq:nd-fourier-action}
\end{equation}
Lemma~\ref{lem:fourier-contraction} therefore yields
\begin{equation}
\rho
=\opnorm{P}_{L_0^2(\pi)}
=\sup_{m\neq0}4^{-m^2}
=\frac14,
\qquad \text{and} \qquad
\tau_{\rel}=\frac43.
\label{eq:nd-l2-gap}
\end{equation}


\paragraph{Bounded features and the covariance spectrum.}
Since \(\sum_{m\in\bbZ}a_m<\infty\), the series defining \(k\) converges absolutely and uniformly. The nonnegative Fourier coefficients make \(k\) positive definite, and their symmetry makes it real-valued. The kernel is continuous and its RKHS is separable. Its canonical feature map is continuous and satisfies
\begin{equation}
\norm{\phi(x)}_\cH^2
=k(x,x)
=\sum_{m\in\bbZ} a_m
=:\kappa^2<\infty,
\qquad x\in\cX.
\label{eq:nd-feature-bound}
\end{equation}
Introduce the real orthonormal basis of \(L^2(\pi)\)
\begin{equation}
u_0(x):=1,
\qquad u_m^{\mathrm c}(x):=\sqrt2\cos(2\pi m x),
\qquad u_m^{\mathrm s}(x):=\sqrt2\sin(2\pi m x),
\quad m\geq1.
\end{equation}
The corresponding orthonormal basis of \(\cH\) is
\begin{equation}
h_0:=u_0,
\qquad h_m^{\mathrm c}:=\sqrt{a_m}\,u_m^{\mathrm c},
\qquad h_m^{\mathrm s}:=\sqrt{a_m}\,u_m^{\mathrm s}.
\end{equation}
For \(\Sigma=\int\phi(x)\otimes\phi(x)\,\pi(dx)\), the reproducing property gives
\begin{equation}
\la\Sigma h,h'\ra_\cH
=\int h(x)h'(x)\,\pi(dx).
\end{equation}
Consequently,
\begin{equation}
\Sigma h_0=h_0,
\qquad
\Sigma h_m^{\mathrm c}=a_m h_m^{\mathrm c},
\qquad
\Sigma h_m^{\mathrm s}=a_m h_m^{\mathrm s}.
\end{equation}
Its eigenvalues, listed in non-increasing order, are therefore
\begin{equation}
\lambda_j(\Sigma)
=\frac{1}{1+\lfloor j/2\rfloor^2},
\qquad j\geq1.
\end{equation}
Using \((2m+1)^2\leq5(1+m^2)\), we obtain
\begin{equation}
\lambda_j(\Sigma)\leq5j^{-2}.
\label{eq:nd-capacity}
\end{equation}
Thus the polynomial eigenvalue-decay assumption holds with \(B=5\), \(b=2\), and \(p=1/2\). Moreover,
\begin{equation}
\cN(\lambda)
=\sum_{m\in\bbZ}\frac{a_m}{a_m+\lambda}
=\sum_{m\in\bbZ}\frac{1}{1+\lambda(1+m^2)}
\asymp\lambda^{-1/2},
\qquad \lambda\downarrow0.
\end{equation}
The upper bound follows by integral comparison, while the lower bound follows by retaining the terms with \(\abs{m}\leq\lambda^{-1/2}\).

\paragraph{The covariance kernel.}
The feature map \( \Phi \) governing covariance estimation induces the kernel
\begin{equation}
K(x,y)
=\la\Phi(x),\Phi(y)\ra_{\cL_2(\cH)}
=k(x,y)^2.
\end{equation}
Absolute convergence permits multiplication of the Fourier series:
\begin{equation}
K(x,y)
=\sum_{m\in\bbZ}c_m e^{2\pi i m(x-y)},
\qquad \text{where} \quad
c_m:=\sum_{j\in\bbZ}a_j a_{m-j}.
\label{eq:nd-covariance-kernel}
\end{equation}
In particular,
\begin{equation}
c_0=\sum_{j\in\bbZ}(1+j^2)^{-2},
\qquad
\sum_{m\in\bbZ}c_m=\kappa^4.
\end{equation}
For \(m\neq0\), we claim that
\begin{equation}
c_m=\frac{2\kappa^2}{m^2+4}.
\label{eq:nd-tensor-coefficients}
\end{equation}
Indeed, the partial-fraction identity
\begin{align}
\frac{1}{(1+j^2)(1+(m-j)^2)}
=\frac{1}{m^2+4}
\bigg[
&\frac{1}{1+j^2}+\frac{1}{1+(m-j)^2}
+\frac{2}{m}
\left(
\frac{j}{1+j^2}
+\frac{m-j}{1+(m-j)^2}
\right)
\bigg]
\end{align}
can be summed over \(-N\leq j\leq N\). As \(N\to\infty\), the first two sums converge to \(\kappa^2\). The sum of \(j/(1+j^2)\) vanishes by symmetry. The shifted sum of \((m-j)/(1+(m-j)^2)\) tends to zero because its difference from a symmetric sum consists of finitely many boundary terms tending to zero. This proves \eqref{eq:nd-tensor-coefficients}.

\medskip \noindent
Since \(\int K(x,y)\,\pi(dy)=c_0\), the centered covariance kernel is
\begin{equation}
\overline K(x,y)
=K(x,y)-c_0
=\sum_{m\neq0}c_m e^{2\pi i m(x-y)}.
\end{equation}
Its nonzero eigenvalues are \(c_m\), \(m\neq0\).

\paragraph{Hilbert-scale boundedness and contraction.}
We apply Lemma~\ref{lem:fourier-contraction} to the centered covariance kernel \(\overline K\), using its Fourier coefficients \(c_m\) as the
weights in that lemma. By~\eqref{eq:nd-tensor-coefficients}, the embedding constant is
\begin{equation}
\kappa_\alpha^2
=\sum_{m\neq0}c_m^\alpha
=(2\kappa^2)^\alpha
\sum_{m\neq0}(m^2+4)^{-\alpha}
<\infty,
\qquad \alpha>\frac12.
\label{eq:nd-scale-bound}
\end{equation}
Since \(\wh \upsilon_\eta (m)=4^{-m^2}\), the same lemma gives, for \(\alpha\in(1/2,1]\),
\begin{align}
\rho_\alpha
:=\opnorm{P}_{\cE}
=\sup_{m\geq1}
4^{-m^2}\left(\frac{c_m}{c_{2m}}\right)^{\alpha/2}
=\sup_{m\geq1}
4^{-m^2}
\left(\frac{4m^2+4}{m^2+4}\right)^{\alpha/2}
=\frac14\left(\frac85\right)^{\alpha/2}
\leq\frac{1}{\sqrt{10}}<1.
\label{eq:nd-scale-contraction}
\end{align}
Indeed, the supremum is attained at \(m=1\): for \(m\geq2\), the expression is at most \(4^{-4}2^\alpha\leq1/128\), whereas its value at \(m=1\) exceeds \(1/4\). This verifies Assumption~\ref{ass:hilbert-scale-gap} for \(K=k^2\). By~\eqref{eq:fourier-measure-bound},
\(\norm{\delta_x-\pi}_{-\alpha}=\kappa_\alpha\) for every \(x\). In particular,
\begin{equation}
\rho_1=\frac{1}{\sqrt{10}},
\qquad
\kappa_1^2
=\sum_{m\neq0}c_m
=\kappa^4-c_0.
\end{equation}

\paragraph{The regression function and innovations.}
The one-step prediction function, for responses represented in \([0,1)\), is
\begin{equation}
f_\star(x)
=\int_0^1y\,\upsilon_\eta(y-2x)\,dy.
\end{equation}
The uniform convergence in \eqref{eq:nd-noise-fourier} permits termwise integration. Using
\begin{equation}
\int_0^1y\cos(2\pi m y)\,dy=0,
\qquad \text{and} \qquad
\int_0^1y\sin(2\pi m y)\,dy=-\frac{1}{2\pi m},
\end{equation}
we obtain
\begin{equation}
f_\star(x)
=\frac12-\frac1\pi
\sum_{m\geq1}\frac{4^{-m^2}}{m}\sin(4\pi m x).
\label{eq:nd-regression}
\end{equation}
This series and its derivatives of every order converge uniformly.

\medskip \noindent
The innovation appearing in the learning assumptions is
\begin{equation}
\varepsilon_{t+1}
:=X_{t+1}-f_\star(X_t),
\end{equation}
so that \(\bbE[\varepsilon_{t+1}\mid X_t]=0\). Since both the response after reduction modulo one and its conditional mean lie in \([0,1]\), \( \abs{\varepsilon_{t+1}}\leq1 \). Furthermore, \(y^2\leq y\) on \([0,1]\), and hence, for every \(x\),
\begin{align}
\int\abs{y-f_\star(x)}^2 P(x,dy)
&=\int y^2P(x,dy)-f_\star(x)^2
\leq f_\star(x)\bigl(1-f_\star(x)\bigr)
\leq\frac14.
\end{align}
Thus the bounded-innovation assumption holds with \(M=1\) and \(\sigma^2=1/4\). These constants concern the prediction innovation, whereas \(\eta\) is the scale of the Gaussian perturbation before reduction modulo one.

\paragraph{Well-specifiedness and the source condition.}
In the orthonormal RKHS basis introduced above, define
\begin{equation}
w_\star
:=\frac12h_0
-\frac{1}{\sqrt2\pi}
\sum_{m\geq1}
\frac{4^{-m^2}}{m\sqrt{a_{2m}}}\,h_{2m}^{\mathrm s}.
\label{eq:nd-target-parameter}
\end{equation}
This series converges in \(\cH\), because
\begin{equation}
\norm{w_\star}_\cH^2
=\frac14+\frac{1}{2\pi^2}
\sum_{m\geq1}\frac{16^{-m^2}(1+4m^2)}{m^2}
<\infty.
\end{equation}
Since \(h_{2m}^{\mathrm s}(x)=\sqrt{2a_{2m}}\sin(4\pi m x)\), its pointwise realization agrees with \eqref{eq:nd-regression}. Convergence in \(\cH\) also implies uniform convergence by \eqref{eq:nd-feature-bound}.
Thus
\begin{equation}
f_\star(x)=\la w_\star,\phi(x)\ra_\cH,
\qquad x\in\cX,
\end{equation}
which verifies well-specifiedness.

\medskip
For any fixed \(r\in[1/2,1]\), define
\begin{equation}
g_\star
:=\frac12h_0
-\frac{1}{\sqrt2\pi}
\sum_{m\geq1}
\frac{4^{-m^2}}{m a_{2m}^{r}}\,h_{2m}^{\mathrm s}.
\end{equation}
Then
\begin{equation}
G_\star^2
:=\norm{g_\star}_\cH^2
=\frac14+\frac{1}{2\pi^2}
\sum_{m\geq1}
\frac{16^{-m^2}(1+4m^2)^{2r}}{m^2}
<\infty.
\label{eq:nd-source-constant}
\end{equation}
Since \(\Sigma h_{2m}^{\mathrm s} =a_{2m}h_{2m}^{\mathrm s}\), applying \(\Sigma^{r-1/2}\) yields
\begin{equation}
\Sigma^{r-1/2}g_\star=w_\star.
\end{equation}
The source condition therefore holds for every \(r\in[1/2,1]\).


\medskip
All assumptions of Theorem~\ref{thm:final-sample-rate} hold with
\begin{equation}
\rho=\frac14,\qquad
\rho_\alpha=\frac14\left(\frac85\right)^{\alpha/2},
\qquad
M=1,\qquad \sigma^2=\frac14,\qquad
B=5,\qquad b=2,
\end{equation}
for every \(\alpha\in(1/2,1]\) and \(r\in[1/2,1]\). The finite constants \(\kappa\), \(\kappa_\alpha\), and \(G_\star\) are given in \eqref{eq:nd-feature-bound}, \eqref{eq:nd-scale-bound}, and \eqref{eq:nd-source-constant}, respectively.

\end{document}